\documentclass[11pt,reqno]{amsart}
 \UseRawInputEncoding
\usepackage{amsthm, amsfonts, amssymb, color}
 \usepackage{mathrsfs}
\usepackage{enumerate}
\usepackage{amsmath}
 \usepackage{amstext, amsxtra}
  \usepackage{txfonts}
 \usepackage[colorlinks, linkcolor=black, citecolor=blue, pagebackref, hypertexnames=false]{hyperref}
\usepackage{graphicx}
\usepackage[symbol]{footmisc}
\usepackage{cite}
 \allowdisplaybreaks
\newtheorem{theorem}{Theorem}[section]
\newtheorem{proposition}[theorem]{Proposition}
\newtheorem{corollary}[theorem]{Corollary}
\newtheorem{lemma}[theorem]{Lemma}

\newtheorem{remark}[theorem]{Remark}

\newcommand\R{\mathbb{R}}

\renewcommand\Re{\operatorname{Re}}
\renewcommand\Im{\operatorname{Im}}

\numberwithin{equation}{section}
\theoremstyle{definition}

\title[]
 {Dispersive estimates for the Schr\"{o}dinger equation with finite rank perturbations}

\author{Han Cheng,\, Shanlin Huang,\, Quan Zheng}

\address{Han Cheng, School of Mathematics and Statistics, Huazhong University of Science and Technology, Wuhan 430074, Hubei, PR China }
\email{chh@hust.edu.cn}

\address{Shanlin Huang (corresponding author), School of Mathematics and Statistics, Hubei Key Laboratory of Engineering Modeling and Scientific Computing, Huazhong University of Science and Technology, Wuhan 430074, Hubei, PR China}
\email{shanlin\_huang@hust.edu.cn}

\address{Quan Zheng, School of Mathematics and Statistics, Huazhong University of Science and Technology, Wuhan 430074, Hubei, PR China }
\email{qzheng@hust.edu.cn}

\begin{document}

\begin{abstract}
In this paper, we investigate dispersive estimates for the time evolution of Hamiltonians
$$
H=-\Delta+\sum_{j=1}^N\langle\cdot\,, \varphi_j\rangle \varphi_j\quad\,\,\,\text{in}\,\,\,\mathbb{R}^d,\,\, d\ge 1,
$$
where each $\varphi_j$ satisfies certain smoothness and decay conditions. We show that, under a spectral assumption,  there exists a constant $C=C(N, d, \varphi_1,\ldots, \varphi_N)>0$ such that
$$
\|e^{-itH}\|_{L^1-L^{\infty}}\leq C t^{-\frac{d}{2}}, \,\,\,\text{for}\,\,\, t>0.
$$
 As far as we are aware, this seems to provide the first study of  $L^1-L^{\infty}$ estimates for finite rank perturbations of the Laplacian in any dimension.

 We first deal with rank one perturbations ($N=1$). Then we turn to the general case. The new idea in our approach is to establish the Aronszajn-Krein type formula for finite rank perturbations. This allows us to reduce the analysis to the rank one case and solve the problem in a unified manner. Moreover, we show that in some specific situations, the constant $C(N, d, \varphi_1,\ldots, \varphi_N)$ grows  polynomially in $N$. Finally, as an application, we are able to extend  the results to $N=\infty$ and deal with some trace class perturbations.
\end{abstract}
\maketitle


\tableofcontents



\section{Introduction and main results}\label{sec1}
\subsection{Background and motivation}\label{sec1.0}
For the free Schr\"{o}dinger equation in $\mathbb{R}^d$, $d\ge 1$,  it's  well-known that the solution satisfies the following dispersive estimate
\begin{equation}\label{eq0.1}
\|e^{it\Delta}\|_{L^1-L^{\infty}}\leq Ct^{-\frac{d}{2}},\qquad\,\,t>0,
\end{equation}
where $\Delta$ denotes the Laplacian in $\mathbb{R}^d$. This follows immediately from the explicit expression of the fundamental solution
\begin{equation*}\label{eq0.2}
e^{it\Delta}f(x)=\frac{1}{(-4\pi it)^{\frac{d}{2}}}\int_{\mathbb{R}^d}{e^{i|x-y|^2/4t}f(y)\,dy}.
\end{equation*}

Motivated by nonlinear problems, the extension of \eqref{eq0.1} to the Schr\"{o}dinger operators $H=-\Delta+V(x)$ has evoked considerable interest in the past three decades, which we briefly recall. The first breakthrough in establishing global $L^1-L^{\infty}$ estimates was due to Journ\'{e}, Soffer and Sogge \cite{JSS}. By using a time dependent approach (Duhamel's formula), they proved that, under certain smoothness and decay assumption of the potential $V$, the dispersive bound
\begin{equation*}\label{eq0.3}
\|e^{-itH}P_{ac}(H)\|_{L^1-L^{\infty}}\leq Ct^{-\frac{d}{2}},\qquad\,\,t>0
\end{equation*}
holds  for all $d\ge3$ provided that zero is neither an eigenvalue nor a resonance. Here $P_{ac}(H)$ denotes the projection onto the absolutely continuous spectral subspace of $H$. Later, Yajima \cite{Ya,Ya95} developed a stationary approach by proving stronger estimates concerning $L^p$ boundedness of  wave operators with weaker assumptions on the potential. Another significant improvement was made by Rodnianski and Schlag \cite{RS}. Thanks to the spectral representation and Stone's formula (with $f, g\in\mathcal{S}(\mathbb{R}^d)$),
\begin{equation*}
\langle e^{-itH}P_{ac}(H)f,\, g\rangle=\frac{1}{2\pi i}\int_0^{\infty}e^{-it\lambda}\langle[R^+(\lambda)-R^-(\lambda)]f,\,g\rangle d\lambda,
\end{equation*}
they reduce matters to the analysis of the (perturbed) resolvent as well as oscillatory integrals.
Inspired by these pioneering  works, a large number of important results were built, see e.g. \cite{BG,D,Ya2,EGG18,ES,ES2,EGG,EG10,GV,EG13B,Gol,GW16,GS,Gol2,GW15,GW17,EG13,RSS, Wed
} and references therein. In particular, Goldberg \cite{Gol,Gol2} proved dispersive bounds under (almost) critical decay conditions on the potential in three dimension; Erdo\v{g}an and Green \cite{EG10} established dispersive estimates under the optimal smoothness condition for the potential in dimensions $d=5, 7$. Further results for higher dimensions are established more recently by Goldberg and Green \cite{GW15,GW17}. For more details and related history, we refer to the survey papers \cite{Sch07,Sch21}.
We also mention that very recently, local and global dispersive estimates for higher order Schr\"{o}dinger operators have been studied in \cite{EGT,FSWY,FSY,EG22} by various authors.

\textbf{Aim and Motivation.} The main goal of this paper is to investigate the stability of dispersive estimates \eqref{eq0.1} under finite rank perturbations, i.e.,
\begin{equation}\label{eq0.4}
H=H_0+\sum_{j=1}^NP_j, \,\,\,\,\, H_0=-\Delta, \,\,\, P_j=\langle\cdot\,, \varphi_j\rangle \varphi_j.
\end{equation}
The theory of finite rank perturbations can be traced back to a seminal paper in
1910 by Weyl \cite{Wey}, where they were introduced as a tool to determine the spectrum of
Sturm-Liouville operators. They also  arise in a number of problems in mathematical physics.
For example, Simon and Wolff \cite{SW} found a relationship between rank one perturbation and discrete random Sch\"{o}dinger operators, and applied it to the Anderson localization conjecture.
We note that \eqref{eq0.4} can be viewed as the Sch\"{o}dinger hamiltonian with a large number of obstacles, and each obstacle is described by a rank one perturbation.
It deserves to mention that despite the formal simplicity, the study of this restrictive perturbation is very rich and interesting, we refer to \cite{Simon} for the spectral analysis and the survey paper \cite{Liaw} for other problems.

 We are partially motivated by the result of Nier and Soffer \cite[Theorem 1.1]{NS}, where they have considered  the $L^p$  decay inequality for \eqref{eq0.4}. More precisely, they showed that in dimensions $d\ge 3$,
\begin{equation}\label{eq1.3.1}
\|e^{-itH}\|_{L^p-L^{p'}}\leq C(N, d, \varphi_1,\ldots, \varphi_N)\cdot t^{-d(\frac{1}{p}-\frac{1}{2})},\quad\,\,\,\,t>0,\quad \,1<p\leq 2
\end{equation}
holds under suitable conditions on each $\varphi_j$ (see Remark ($\textbf{a}_1$) and ($\textbf{b}_1$) below). We note that the endpoint case $p=1$ is still unknown, and the problems for $d=1, 2$ have not been touched upon. Therefore, compared to existing results for
Sch\"{o}dinger operators mentioned above, it is natural to ask whether \eqref{eq1.3.1} can be  strengthened to the point-wise dispersive estimates for all $d\ge1$ (which implies $p>1$ by interpolation and the trivial $L^2$ conservation). Meanwhile, since trace class operators can be approximated by finite rank ones, it's also interesting to see whether the dispersive estimates can be carried out in such general perturbations.

In this study we aim at filling this gap by establishing $L^1-L^{\infty}$ estimates for the Hamiltonian \eqref{eq0.4} in all dimensions. We mention that our approach is very different from \cite{NS}  both in rank one and finite rank perturbations. Moreover, The methods used here allow us to relax the smoothness and decay assumptions (on $\varphi_j$) in \cite{NS} and in particular, work for all $d\ge 1$ in a unified manner. Another advantage of our approach is that we are able to extend our results to trace class perturbations in a very natural way.

\subsection{Dispersive estimates for rank one perturbations}\label{sec1.1}

Let us start with the case of rank  one perturbations. More precisely, consider
\begin{equation}\label{eq1.1}
H_{\alpha}=H_0+\alpha \langle\cdot\,, \varphi\rangle\varphi, \,\,\,\,\, H_0=-\Delta,\quad\,\,\alpha>0,
\end{equation}
where $\varphi\in L^2$. Here and in the rest of the paper, $\langle\cdot\,, \cdot\rangle$ denotes the usual inner product in $L^2(\mathbb{R}^d)$.

Before stating our assumptions on $\varphi$, we mention that a key role will be played by the following function
\begin{equation}\label{eq1.1.1}
F(z):=\langle(-\Delta-z)^{-1}\varphi,\varphi\rangle=\int{\frac{d\mu_{H_0}^{\varphi}(\lambda)}{\lambda-z}},
\end{equation}
where $z\in \mathbb{C}\setminus[0, \infty)$. This can be viewed as the Borel transform of the spectral measure $d\mu_{H_0}^{\varphi}$ for $\varphi$ (here $\mu_{H_0}^{\varphi}((a, b))=\langle E_{(a,b)}(H_0)\varphi,\varphi\rangle$). Furthermore, if $\varphi$ is assumed to belong to the weighted $L^2$ space $L^{2}_{\sigma}$ with $\sigma>\frac12$, where $L^{2}_{\sigma}:=\{(1+|x|)^{-\sigma}f:\, f\in L^2\}$, then by the well-know limiting absorption principle (see e.g. in \cite{Ag,KK}), the boundary values
\begin{equation}\label{eq1.11.1}
F^\pm(\lambda):=\langle(-\Delta-\lambda\mp i0)^{-1}\varphi,\varphi\rangle
\end{equation}
are defined when $\lambda\ne 0$ and coincide on $(-\infty, 0)$.

\textbf{Condition $H_{\alpha}$}.\,\, ~~~Suppose that $\alpha>0$ and let $\varphi$ be a $L^2$-normalized real-valued function on $\mathbb{R}^d$ satisfying:

\emph{(\romannumeral1) } (decay and smoothness) \,\, For $d\ge 1$, there is some absolute constant $M>0$ such that
\begin{equation}\label{eq1.2}
|\varphi(x)|\leq M\langle x \rangle^{-\delta},\,\,\,\,\,  \delta>d+\frac{3}{2}.
\end{equation}
Moreover, when  $d>3$, we assume that $\varphi(x)\in C^{\beta_0}(\mathbb{R}^d)$ with $\beta_0=[\frac{d-3}{4}]+1$, and
\begin{equation}\label{eq1.3}
|\partial^{\beta}_x\varphi(x)|\leq M\langle x \rangle^{-\delta}, \,\,\,\,\,|\beta|\leq \beta_0,\,\,\,\,\,  \delta>d+\frac{3}{2}.
\end{equation}

\emph{(\romannumeral2) } (spectral assumption) \,\, For $d\ge 1$,  there is an absolute constant $c_0>0$ such that
\begin{equation}\label{eq1.4}
|1+\alpha F^\pm(\lambda^2)|\ge c_0,\,\,\,\,\,\,\text{for any}\,\,\,\lambda>0,
\end{equation}
where $F^\pm(\lambda^2)$ is given by \eqref{eq1.11.1}.

Nier and Soffer  \cite[Lemma 2.6]{NS} showed that the assumption \eqref{eq1.4}, together with some decay condition on $\varphi$ , yields that the spectrum of $H_{\alpha}$ is absolutely continuous for all $d\ge 3$.
In fact, we shall prove that this is true for all $d\ge1$ under \eqref{eq1.4} and \eqref{eq1.2}, see Lemma \ref{lmA.1}.

Our first main result in this paper is the following:

\begin{theorem}\label{thm-1.1}
Let $d\ge 1$ and $H_{\alpha}$ be defined in \eqref{eq1.1}. Assume  that $\varphi$ satisfies the \textbf{Condition $H_{\alpha}$} above. Then $H_{\alpha}$ has only absolutely continuous spectrum, and
\begin{equation}\label{eq1.5}
\|e^{-itH_{\alpha}}-e^{it\Delta}\|_{L^1-L^{\infty}}\leq C(\alpha, \varphi)\cdot t^{-\frac{d}{2}}, \,\,\,\text{for}\,\,\, t>0.
\end{equation}
where $C(\alpha, \varphi)>0$\footnote{Strictly speaking,  the constant here also depends on $d$ and $c_0$, but sometimes we omit them and focus on the dependence on $\alpha$ and $\varphi$.} can be chosen as
$$
C(\alpha, \varphi)=C(\varphi)\cdot(1+\alpha)^{[\frac{d}{2}]+2}
$$
for some absolute constant $C(\varphi)>0$  independent with $\alpha$. Furthermore, if $0<\alpha<1$, then in the case $d\ge 3$ or $d=1,2$, $\int_{\mathbb{R}^d}\varphi(x)\,dx=0$, one can choose
\begin{equation}\label{eq1.5.0}
C(\alpha, \varphi)=CM^{2([\frac{d}{2}]+3)}\cdot\alpha\cdot (1+\alpha)^{[\frac{d}{2}]+1}
\end{equation}
for some absolute constant $C>0$  depending only on $d, c_0$.
\end{theorem}
We make the following remarks related to Theorem \ref{thm-1.1}.
\begin{itemize}
\item [($\textbf{a}_1$)] The $L^1-L^{\infty}$ estimates for $e^{-itH_{\alpha}}$ follows immediately from \eqref{eq1.5} and \eqref{eq0.1}. The reason that we write in the form \eqref{eq1.5}  is due to the constant.
During the proof, we can keep track of the upper bound with respect to $\alpha$ and $M$. In particular, if we let $\alpha\rightarrow 0$, then the upper bound \eqref{eq1.5.0} plays a critical role when we extend the result to a class of compact perturbations, see Theorem \ref{thm1.6}.
\item [($\textbf{a}_2$)] {To the authors' knowledge, Theorem  \ref{thm-1.1} seems to provide the first result on $L^1-L^{\infty}$ estimates for the Hamiltonian \eqref{eq1.1} in any dimension $d\ge1$. Actually,  the only previous related work that we are aware of is due to Nier and Soffer \cite[Theorem 2.1]{NS},  where the weaker $L^p-L^{p'}$ estimate ($1<p\le 2$) was proved  for $d\ge 3$ under the spectral  assumption  \eqref{eq1.4} and the following
   \begin{equation}\label{eq1.6}
    \langle x\rangle^s \langle D\rangle^{\beta}\varphi\in L^2(\mathbb{R}^d), \,\,\text{for}\,\,s>[\frac{d}{2}]+2,\,\,\,\beta\ge\max\{\frac{d-3}{2}, \frac12[\frac{d}{2}]\}.
   \end{equation}
    Roughly speaking, this means that $\varphi$ satisfies \eqref{eq1.2} with $\delta> d+2$ and  \eqref{eq1.3} with $\beta_0\approx[\frac{d-3}{2}]+1$, thus the regularity requirements for $\varphi$ in our conditions are only about half of theirs in higher dimensions.
   On the other hand,  when $1\le d\le 3$, we don't need any smoothness assumption on $\varphi$, and this is consistent with the case of $H=-\Delta+V(x)$ (see \cite{GS,Gol,Gol2,Sch}).
    }

\item [($\textbf{a}_3$)] We briefly outline the strategy for proving Theorem \ref{thm-1.1}. Inspired by the work on Sch\"{o}dinger operators, our starting point is to represent the solution via the Stone's formula. The peculiarity of the rank one perturbation is the so-called Aronszajn-Krein formula (see \eqref{eq3.2}). It is  exactly this formula that allows us to derive dispersive estimates in a unified manner. Indeed, we shall write the kernel of $e^{-itH_{\alpha}}-e^{it\Delta}$ in the form
    \begin{equation}\label{eq1.6.1}
 \alpha\cdot \int_{\mathbb{R}^{2d}}{I_d(t, |x-x_1|, |x_2-y|)G_1(x, x_1)G_2(x_2, y)\,dx_1dx_2},
    \end{equation}
where $I_d(t, |x-x_1|, |x_2-y|)$  stands for the following type of oscillatory integrals
 \begin{equation}\label{eq1.6.2}
 \int_\Omega e^{i(t\lambda^2+\lambda\cdot (|x-x_1|\pm|x_2-y|) )}\frac{\psi(\lambda,|x-x_1|, |x_2-y|)}{1+\alpha F^{\pm}(\lambda^2)}d\lambda,
    \end{equation}
The explicit expression of $G_1, G_2,$ and $\psi$ depend on the dimension. As usual, we treat $\Omega=(0, r_0)$ and $\Omega=(r_0, \infty)$ ($r_0>0$) separately, which correspond to the low and high energy part respectively.
 Since the behavior of the free resolvent differs in different dimensions, we shall divide the analysis into several cases. One of the key ingredients throughout the proof is to obtain the decay $t^{-\frac{d}{2}}$ from \eqref{eq1.6.2} and keep $x, y\in\mathbb{R}^d$ uniformly bounded in the kernel \eqref{eq1.6.1}. From the perspective of the oscillatory integral \eqref{eq1.6.2}, the  main difficulty in the high energy part is that the growth of $\lambda$ may be "too large" and this happens  when $d>3$. While  in the low energy part, one has to face the singularity of $\lambda$ near origin when $d=1, 2$, in particular, we  point out that  the methods would be different depending whether $\int_{\mathbb{R}^d}{\varphi}=0$ ($d=1, 2$) or not. In these cases, we shall exploit ideas and techniques used for Sch\"{o}dinger operators (see e.g. in \cite{GW15,GW16,GW17}).
Finally we mention that  the constant \eqref{eq1.5.0} is not obtained in the case $0<\alpha<1$, $d=1, 2$ and $\int_{\mathbb{R}^d}{\varphi}\ne0$, see Remark \ref{rmk3.3} for an explanation.

\end{itemize}

\subsection{Dispersive estimates for finite rank perturbations}\label{sec1.2}

%
In the second part of the paper, we address the same problem for finite rank perturbations \eqref{eq0.4}.
 In this case, we assume that $\varphi_1,\ldots, \varphi_N$ are $L^2$-normalized and mutually orthogonal, i.e.,
 \begin{equation}\label{eq1.5.1}
\langle\varphi_i, \varphi_j\rangle=\delta_{ij},\,\,\,\qquad 1\leq i, j\leq N,
 \end{equation}
where $\delta_{ij}$ denotes the Kronecker delta function. Besides the decay and smoothness condition on each $\varphi_j$, we shall need an analog of the spectral assumption \eqref{eq1.4}. Instead of the scalar function defined in \eqref{eq1.11.1}, we define the following $N\times N$ matrix
\begin{equation}\label{eq1.11}
  F_{N\times N}(z)=(f_{ij}(z))_{N\times N},\qquad f_{ij}(z)=\langle(-\Delta-z)^{-1}\varphi_i, \varphi_j\rangle,\,\,\, z\in \mathbb{C}\setminus[0, \infty).
\end{equation}
When $\varphi_i\in L^{2}_{\sigma}$, $\sigma>\frac12$, $1\le i\le N$, it's meaningful to define the boundary value
\begin{equation}\label{eq1.12}
F_{N\times N}^{\pm}(\lambda^2):=F_{N\times N}(\lambda^2\pm i0),\qquad \lambda>0.
\end{equation}
In view of \eqref{eq1.4}, it's natural to make the following assumption: There is some absolute constant $c_0>0$ such that
\begin{equation}\label{eq1.13.1}
|\det{(I+F_{N\times N}^{\pm}(\lambda^2))}|\ge c_0>0,\,\,\,\,\,\,\text{for any}\,\,\,\lambda>0.
\end{equation}
In Lemma  \ref{lmA.1}, we shall prove that this assumption implies  $\sigma(H)=\sigma_{ac}(H)=[0, \infty)$.

We first present the following  general result.
\begin{theorem}\label{thm1.2}
Let $d\ge 1$ and \eqref{eq1.5.1} hold. Assume that each $\varphi_i$ ($1\le i\le N$) satisfies  \emph{(\romannumeral1)} in the \textbf{Condition $H_{\alpha}$} and the spectral assumption  \eqref{eq1.13.1} holds.
Then $H$ has only absolutely continuous spectrum and there exists a constant $C=C(N, d, \varphi_1,\ldots, \varphi_N)>0$ such that
\begin{equation}\label{eq1.13}
\|e^{-itH}\|_{L^1-L^{\infty}}\leq C\cdot t^{-\frac{d}{2}},\quad \,\,\,\text{for}\,\,\, t>0.
\end{equation}


\end{theorem}

It deserves to mention that in our approach, the upper bound  in \eqref{eq1.13} is closely related to the behavior of each element in the inverse matrix $(I+F_{N\times N}^{\pm}(\lambda^2))^{-1}$. We point out that the constant $C(N, d, \varphi_1,\ldots, \varphi_N)$ may grow faster than $N!$ (see Remark \ref{rmk3.5}) in general. However, we shall
discuss two concrete models of Theorem \ref{thm1.2},
in which the above matrix has special structures (see Remark ($\textbf{b}_2$) below). Then we  are able to prove  polynomial upper bounds with respect to  $N$, and this is important when we further extend to the case $N=\infty$.

\begin{theorem}\label{thm1.4}
Let $d\ge 1$  and \eqref{eq1.5.1} hold. Assume that each $\varphi_i$ ($1\le i\le N$) satisfies the \textbf{Condition $H_{\alpha}$} with $\alpha=1$. In addition, if
\begin{equation}\label{eq0.19}
 \text{supp} \hat{\varphi_i}\cap \text{supp} \hat{\varphi_j}=\emptyset, \quad\,\,\,\,i\ne j,
\end{equation}
then there exists a constant $C=C(d, \varphi_1,\ldots, \varphi_N)$ independent with $N$ such that
\begin{equation}\label{eq1.16}
\|e^{-itH}\|_{L^1-L^{\infty}}\leq CNt^{-\frac{d}{2}}, \quad\,\,\,\text{for}\,\,\, t>0.
\end{equation}
\end{theorem}

\begin{theorem}\label{thm1.5}
Let $d\ge 3$  and \eqref{eq1.5.1} hold. Assume that $\varphi_i(x)=\varphi(x-\tau_i)$, where $\varphi$ satisfies  the \textbf{Condition $H_{\alpha}$} with $\alpha=1$. If, in addition, the following spreading condition
\begin{equation}\label{eq1.16.1}
  |\tau_i-\tau_j|\ge \tau_0,\qquad\text{where}\quad \tau_0>(CN)^{\frac{2}{d-1}}
\end{equation}
holds for some $C>0$ independent with $N$. Then there exists another constant $C=C(d, \varphi)>0$ independent with $N$ such that
\begin{equation}\label{eq1.17}
\|e^{-itH}\|_{L^1-L^{\infty}}\leq CN^2t^{-\frac{d}{2}},\quad \,\,\,\text{for}\,\,\, t>0.
\end{equation}
\end{theorem}

Some remarks on Theorem \ref{thm1.2}-Theorem \ref{thm1.5} are as follows:

\begin{itemize}
\item [($\textbf{b}_1$)]  As far as we are aware, this is the first study of  $L^1-L^{\infty}$ estimates for the Hamiltonian \eqref{eq0.4} in any dimension $d\ge 1$.  In \cite[Theorem 1.1]{NS}, Nier and Soffer consider the case  $\varphi_i(x)=\varphi(x-\tau_i)$, where the spreading condition $ |\tau_i-\tau_j|\ge \tau_0$ is assumed.  They show that if  $\varphi$ satisfies  \eqref{eq1.4}, and \eqref{eq1.6} with  $s>[\frac{d}{2}]+2,\,\beta>\frac{d}{2}$ (our  smoothness assumptions on  $\varphi$ are weaker), then  for $d\ge 3$,  $1<p\leq2$,  $1<r<\min\{p,\frac{2d}{d+2}\}$ and
    $\left(  \frac{1/p-1/2}{1/r-1/2}\right)<\theta\leq 1$, one has
\begin{equation}\label{eq1.18}
\tau_0> (CN)^{\frac{1}{d(r-1)}}\Longrightarrow	\|e^{-itH}\|_{L^P-L^{p'}}\leq CN^{\theta}t^{-d(\frac{1}{p}-\frac12)}, \,\,\,\,\, t>0.
\end{equation}
We mention that their arguments fail at the endpoint $p=1$. On the other hand, in Remark \ref{rm4.2}, we show that one can slightly modify the proof in Theorem \ref{thm1.5}  to obtain  \eqref{eq1.18}.

\item [($\textbf{b}_2$)] We point out that our  strategy to prove Theorems \ref{thm1.2}--\ref{thm1.5} is very different from the one used in \cite[Theorem 1.1]{NS} (they use  a bootstrap argument based on Duhamel's principle and induction on $N$). The novelty of our approach is that we first establish an Aronszajn-Krein type formula for finite rank perturbations.
    This allows us to reduce matters to the rank one case by studying the inverse of the matrix $I+F_{N\times N}^{\pm}(\lambda^2)$. In general, we don't have favorable control on $N$ for each element in $(I+F_{N\times N}^{\pm}(\lambda^2))^{-1}$ and the upper bound may grow faster than $N!$  (see Remark \ref{rmk3.5}). However, we point out that in Theorem \ref{thm1.4}, the  matrix  $I+F_{N\times N}^{\pm}(\lambda^2)$ becomes diagonal by assumption \eqref{eq0.19} and in Theorem \ref{thm1.5}, the assumption  \eqref{eq1.16.1} implies it's strictly diagonally-dominant by rows if $\tau_0$ is sufficiently large. We will take advantage of these properties to get more precise estimates concerning $N$.

\item [($\textbf{b}_3$)] It appears that the growth order of $N$ in Theorem  \ref{thm1.4} and \ref{thm1.5} are optimal. But currently we don't know how to construct counterexamples to prove or disprove the sharpness.
 \end{itemize}

\subsection{Dispersive estimates for  trace class perturbations}\label{sec1.4}
As an application, we show that our previous results
on rank one and finite rank perturbations can be extended
to the following more general case:
\begin{equation}\label{eq0.25}
H=-\Delta+A, \,\,\,\text{where}\,\, A:=\sum_{j=1}^{+\infty}\lambda_jP_j, \,\,\,\,\, P_j=\langle\cdot\,, \varphi_j\rangle \varphi_j,\,\, \lambda_j>0.
\end{equation}
To be more precise, we have
\begin{theorem} \label{thm1.6}
Let $d\ge 1$ and $H$ be defined in \eqref{eq0.25}. Assume that  \eqref{eq1.5.1} and \eqref{eq0.19} hold respectively. Furthermore,  suppose that

(\romannumeral1) each $\varphi_j$  satisfies  \eqref{eq1.2} and  \eqref{eq1.3} with the upper bound $M$ replaced by $M_j$;

(\romannumeral2) the spectral condition \eqref{eq1.4} holds uniformly for $\{\varphi_j\}^{+\infty}_{j=1}$, i.e., there exists some absolute constant  $c_0>0$, such that
\begin{equation}\label{eq1.88.0}
 |1+\lambda_j f^{\pm}_{jj}(\lambda^2)|>c_0,\qquad \forall \,\,j\ge1;
\end{equation}

(\romannumeral3) \begin{equation} \label{eq1.88.1}
\sum_{j=1}^{+\infty}\lambda_j\cdot {M_j}^{2[\frac d2]+6}<+\infty.	
\end{equation}
Then  the spectrum of $H$ is purely absolutely continuous and there exists some constant $C(A)>0$ such that
\begin{equation}\label{eq1.88.22}
\|e^{-itH}\|_{L^1-L^{\infty}}\leq C(A)\cdot t^{-\frac{d}{2}},\quad \,\,\,\text{for}\,\,\, t>0.
\end{equation}
\end{theorem}
We present the following remarks concerning Theorem \ref{thm1.6}:
\begin{itemize}
\item [($\textbf{c}_1$)]  First we point out that our perturbations  belong to the trace class operators in $L^2({\R^d})$. Indeed,  by \eqref{eq1.5.1} and \eqref{eq1.2}  one has $M_j\geq\left\| \langle x \rangle^{-\delta} \right\|_{L^2}^{-1}$, then  \eqref{eq1.88.1}  implies that $\sum_{j=1}^{+\infty}\lambda_j<+\infty$. Concrete examples that satisfy all the assumptions above are constructed in Remark \ref{rmk5.1}.
    On the other hand, from a spectrum point of view, the  trace class perturbations are natural and in some sense are optimal. Indeed, according to the Weyl-von Neumann-Kuroda Theorem (see e.g. in \cite[Chapter  \uppercase\expandafter{\romannumeral10}]{Ka}), there exists a self-adjoint perturbation $A$ in any trace ideal strictly bigger than trace class, such that  $-\Delta+A$ has only pure point spectrum.

\item [($\textbf{c}_2$)]  Theorem \ref{thm1.6} can be viewed as an extension of Theorem \ref{thm-1.1} and \ref{thm1.4} to the case $N=\infty$. One of the major observations is that under conditions \eqref{eq0.19}, (\romannumeral1), and (\romannumeral2) above, we have a simple form of the Aronszajn-Krein type formula for such compact perturbations (see Proposition \ref{pro5.1}).
This allows us to represent the difference $e^{-itH}-e^{it\Delta}$ as an infinite sum of the form $e^{-itH_{\lambda_j}}-e^{-itH_0}$, where each $H_{\lambda_j}$ is a rank one perturbation, see \eqref{eq5.1.5}. Another crucial point is the explicit upper bound \eqref{eq1.5.0} obtained in Theorem \ref{thm-1.1}, based on this and the assumption \eqref{eq1.88.1}, we are able to handle the issue of convergence in \eqref{eq5.1.5}.

\end {itemize}

\subsection{Notation}
Throughout the paper, $C(\cdots)$ and $C_{\{\cdot\}}$ will denote positive constants depending on what are enclosed in the brackets and they may differ on different occasions.
For $l>0$, we denote by $[l]$ the greatest integer at most $l$.
We introduce the following symbols which are very important and will be frequently used in what follows.
 \begin{itemize}
   \item For $b\in\mathbb{R}$,  $K\in\mathbb{N}_0:=\{0, 1,\cdots\}$, and an open set $\Omega\subset\mathbb{R}$, we define $S_K^b(\Omega)$ to be the class of all functions $f(\lambda)\in C^K(\Omega)$ such that
\begin{equation}\label{eq1}
|\frac{d^j}{d\lambda^j} f(\lambda)|\leq C_j|\lambda|^{b-j}\quad\mathrm{if}~\,\,\,\lambda\in\Omega~\,\,\,\mathrm{and}~\,\,\,0\leq j\leq K.
\end{equation}
   \item Let $\alpha, M\in\mathbb{R}$,  $x, y\in\mathbb{R}^d$ be parameters,  we write
   \begin{equation}\label{eq1.111}
 f(\cdot, \alpha, M)\in S_K^b(\Omega)~\,\,\,\,\mathrm{or}~\,\,\,\, g(\cdot, x, y)\in S_K^b(\Omega),
   \end{equation}
if \eqref{eq1} holds and each upper bound $C_j$ doesn't depend on the parameter  $\alpha, M$ or $x, y$.

 \end{itemize}

The rest of the paper is organized as follows:
In section \ref{sec2}, we present the expansions of the free resolvent, properties of the functions $F^{\pm}(\lambda^2)$ given in \eqref{eq1.11.1}, as well as the oscillatory integrals which will be frequently used later.
 Section \ref{sec3} is devoted to the proof of  Theorem \ref{thm-1.1}.
Section \ref{sec4} presents the proofs of  Theorem \ref{thm1.2}-Theorem \ref{thm1.5}.
In Section \ref{sec5}, we prove Theorem \ref{thm1.6}.
Finally, in Appendix \ref{app}, we give the proofs of Lemma  \ref{lm2.6}  and Corollary \ref{cor2.6}, and in  Appendix \ref{app2}, we show that all operators considered in Theorem \ref{thm-1.1}-\ref{thm1.6} have only absolutely continuous spectrum.

\section{Resolvent, properties of $F^\pm(\lambda^2)$  and oscillatory integrals}\label{sec2}
\subsection{Resolvent expansion and properties of $F^\pm(\lambda^2)$}

We start with  the free resolvent $R_0(z)=(-\Delta-z)^{-1}$. Let $R^{\pm}_0(\lambda^2)$ denote the boundary value $R_0(\lambda^2\pm i0)$. Recall that $R^{\pm}_0(\lambda^2)$ is an integral operator with kernel
\begin{equation}\label{eq2.1}
R^{\pm}_0(\lambda^2, x, y)=\frac{\pm i}{2\lambda}e^{\pm i\lambda|x-y|}, \,\,\,\,\,\, d=1,
\end{equation}
and
\begin{equation}\label{eq2.2}
R^{\pm}_0(\lambda^2, x, y)=\frac{\pm i}{4}(\frac{\lambda}{2\pi|x-y|})^{\frac{d}{2}-1}H^{\pm}_{\frac{d}{2}-1}(\lambda|x-y|), \,\,\,\,\,\, d\ge 2,
\end{equation}
where $H^{\pm}_{\frac{d}{2}-1}(z)=J_{\frac{d}{2}-1}(z)\pm i Y_{\frac{d}{2}-1}(z)$ are Hankel functions of order $\frac{d}{2}-1$.

In odd dimensions $d\ge 3$, these Hankel functions can be expressed explicitly as follows (see e.g. \cite{J80})
\begin{equation}\label{eq2.3}
R^{\pm}_0(\lambda^2, x, y)=C_d\frac{e^{\pm i\lambda|x-y|}}{|x-y|^{d-2}}\sum^{\frac{d-3}{2}}_{k=0}{\frac{(d-3-k)!}{k!(\frac{d-3}{2}+k)!}(\mp2i\lambda|x-y|)^k},
\end{equation}
where $C_d=1/(4\pi)^{\frac{d-1}{2}}$.

In even dimensions $d\ge 2$, the resolvent kernel can be decomposed as (see e.g. \cite{Sch,GW17,EGG})
\begin{equation}\label{eq3.50}
 R^{\pm}_0(\lambda^2, x, y)=\lambda^{\frac{d}{2}-1}e^{\pm i\lambda|x-y|}\frac{w_{\pm, >}(\lambda|x-y|)}{|x-y|^{\frac{d}{2}-1}}+\frac{w_{\pm, <}(\lambda|x-y|)}{|x-y|^{d-2}},
\end{equation}
where $\text{supp}\,w_{\pm, <}(z)\subset[0, 1]$ and  $\text{supp}\,w_{\pm, >}(z)\subset[\frac12, \infty]$. In fact, we can pick a bump function $\omega(\lambda)\in C_0^{\infty}$ satisfying $\omega(\lambda)=1$, if $|\lambda|<\frac12$ and  $\omega(\lambda)=0$, if $|\lambda|>1$. Then
$$w_{\pm, <}(\lambda|x-y|)=\frac{\pm i}{4}(\frac{\lambda|x-y|}{2\pi})^{\frac{d}{2}-1}H^{\pm}_{\frac{d}{2}-1}(\lambda|x-y|)\omega(\lambda|x-y|).$$
Furthermore, by the properties of Hankel functions (see e.g. in \cite[p. 364]{AS}, \cite[Sect. 2]{GW17}, \cite[Sect. 2]{Sch}), one has
\begin{equation}\label{eq3.51}
  |\frac{d^l}{dz^l}w_{\pm, >}(z)|\leq C_l\cdot(1+|z|)^{-\frac12-l},
\end{equation}
where $l\in\mathbb{N}_0:=\{0,1,2,\ldots\}$.  When $d=2$,
\begin{equation}\label{eq3.25.11}
 \left|\frac{d^l}{dz^l}w_{\pm, <}(z)\right|\leq\begin{cases}
 C\cdot|\log z|,\qquad \,\text{if}\,\,\,l=0;\\
 C_l\cdot|z|^{-l},\qquad\,\,\,\,\, \text{if}\,\,\,l>0.
 \end{cases}
\end{equation}
When $d\ge 4$,
\begin{equation}\label{eq3.52}
 \left|\frac{d^l}{dz^l}w_{\pm, <}(z)\right|\leq\begin{cases}
C_l,\qquad \qquad\qquad\qquad\quad\,\,\,\,\text{if}\,\,\,l< d-2;\\
 C_l\cdot|z|^{d-2-l}\cdot|\log z|,\qquad\,\quad\text{if}\,\,\,l\ge d-2.
 \end{cases}
\end{equation}
By \eqref{eq2.2} and properties of Bessel functions (see e.g. \cite{AS}, \cite{GW17}), we also have the following decomposition of $R_0^{+}-R_0^{-}$ for dimensions $d \geq 2$,
\begin{align}\label{equ6.2.8}
	&R^{+}_0(\lambda^2, x, y)-R^{-}_0(\lambda^2, x, y)=\frac{\pm i}{2}(\frac{\lambda}{2\pi|x-y|})^{\frac{d}{2}-1}J_{\frac{d}{2}-1}(\lambda|x-y|)\nonumber\\
	&=\frac{\lambda^{\frac{d}{2}-1}}{|x-y|^{\frac{d}{2}-1}}\sum_{\pm}e^{\pm i\lambda|x-y|}\left(J_{\pm, >}(\lambda|x-y|)+J_{\pm, <}(\lambda|x-y|)\right),
\end{align}
where  $\text{supp}\,J_{\pm, <}(z)\subset[0, 1]$,  $\text{supp}\,J_{\pm, >}(z)\subset[\frac12, \infty]$ and they satisfy
\begin{equation}\label{equ6.2.9}
	|\frac{d^l}{dz^l}J_{\pm, >}(z)|\leq C_l(1+|z|)^{-\frac12-l},\quad\,\,\,l\in\mathbb{N}_0,
\end{equation}
and
\begin{equation}\label{equ6.2.10}
	|\frac{d^l}{dz^l}J_{\pm, <}(z)|\leq C_l|z|^{\frac{d}{2}-1-l},\,\,\,\quad\quad\,\,\,\,l\in\mathbb{N}_0.
\end{equation}

Next, we consider the expansion of the kernel $R^{\pm}_0(\lambda^2, x, y)$ near $\lambda=0$.

In one dimension, using \eqref{eq2.1} we can write
\begin{equation}\label{eq2.4}
R^{\pm}_0(\lambda^2, x, y)=\frac{\pm i}{2\lambda}-\frac{|x-y|}{2}+r^{\pm}_1(\lambda, |x-y|),\,\,\,\,\, d=1.
\end{equation}
To estimate the remainder term $r^{\pm}_1$, we
note that
$$e^{\pm i\lambda|x-y|}-1\mp i\lambda|x-y|=(\pm i|x-y|)^2\int_0^{\lambda}e^{\pm is|x-y|}(\lambda-s)ds,$$
 this yields that for $k=0, 1$,
$$|\frac{d^k}{d\lambda^k}\left(e^{\pm i\lambda|x-y|}-1\mp i\lambda|x-y|\right)|\leq C \lambda^{2-k}|x-y|^2\leq C\lambda^{\frac32-k}|x-y|^{\frac32},\,\,\,\, \text{if}\,\,\lambda|x-y|\leq 1,
$$
on the other hand, a direct computation shows that
$$
|\frac{d^k}{d\lambda^k}\left(e^{\pm i\lambda|x-y|}-1\mp i\lambda|x-y|\right)|\leq C \lambda^{1-k}|x-y|\leq C\lambda^{\frac32-k}|x-y|^{\frac32},\,\,\, \text{if}\,\,\lambda|x-y|\geq 1.
$$
Then it follows by \eqref{eq2.1} and \eqref{eq2.4}  that the reminder term $r^{\pm}_1$ satisfies that for $x,y\in\mathbb{R}$,
\begin{equation}\label{eq2.5}
|\frac{d^k}{d\lambda^k}r^{\pm}_1(\lambda, |x-y|)|\leq C_k\lambda^{\frac{1}{2}-k}|x-y|^{\frac{3}{2}},\,\,\,\text{for}\,\,0<\lambda<1,\,\,k=0, 1.
\end{equation}

In two dimension, one has (see e.g. in \cite[Lemma 5]{Sch})
\begin{equation}\label{eq2.6}
R^{\pm}_0(\lambda^2, x, y)=\frac{\pm i}{4}-\frac{\gamma}{2\pi}-\frac{1}{2\pi}\log{\frac{\lambda\cdot|x-y|}{2}}+r^{\pm}_2(\lambda, |x-y|),\,\,\,\,\, d=2,
\end{equation}
where $\gamma$ denotes the Euler constant and the reminder term $r^{\pm}_2$ satisfies for $x,y\in\mathbb{R}^2$,
\begin{equation}\label{eq2.7}
|\frac{d^k}{d\lambda^k}r^{\pm}_2(\lambda, |x-y|)|\leq C_k\lambda^{\frac32-k}|x-y|^{\frac32},\,\,\,\text{for}\,\,0<\lambda<1,\,\, k=0, 1, 2.
\end{equation}

In odd dimensions $d\ge 3$, by Taylor's expansion one has
\begin{equation}\label{eq2.8}
R^{\pm}_0(\lambda^2, x, y)=C_d\sum^{d-2}_{l=0}{d_l(\pm i\lambda)^l|x-y|^{l+2-d}}+r^{\pm}_d(\lambda, |x-y|),\,\,\,\,\, d\ge 3,\,\, \text{odd},
\end{equation}
where
$$
d_l=\sum^{\min\{l,\, \frac{d-3}{2}\}}_{k=0}{\frac{(d-3-k)!}{k!(\frac{d-3}{2}-k)!}(-2)^k\frac{1}{(l-k)!}}.
$$
It follows from \cite[Lemma 3.3]{J80} (see also \cite[Sect. 2]{GW15})
$$
d_l=0,\,\,\,\text{for}\,\,l=1, 3,\ldots, d-4.
$$
Moreover, the reminder term $r_d$ satisfies for $x,y\in\mathbb{R}^d$
\begin{equation}\label{eq2.9}
|\frac{d^k}{d\lambda^k}r^{\pm}_d(\lambda, |x-y|)|\leq C_k\lambda^{d-1-k}|x-y|,\,\,\,\text{for}\,\,0<\lambda<1,\,\,0\le k\le \frac{d+1}{2}.
\end{equation}

Similarly, in even dimensions $d\ge 4$, one has (see  \cite[Lemma 2.1]{GW17} and \cite{EGG,J80})
\begin{align}\label{eq2.10}
R^{\pm}_0(\lambda^2, x, y)&=\sum^{\frac{d}{2}-2}_{l=0}{a_l\lambda^{2l}|x-y|^{2l+2-d}}+a_{\frac{d}{2}-1}\lambda^{d-2}\log{\lambda}\nonumber\\
&+b_{\pm}\lambda^{d-2}+c_{\frac{d}{2}-1}\lambda^{d-2}\log{|x-y|}+r^{\pm}_d(\lambda, |x-y|),\,\,\,\,\, d\ge 4, \,\, \text{even},
\end{align}
where $a_1,\cdots, a_{\frac{d}{2}-1}, c_{\frac{d}{2}-1}\in \mathbb{R}$ are independent with the sign $\pm$ and $b_{\pm}\in \mathbb{C}$.
Moreover, the reminder term $r^{\pm}_d$ satisfies for $0\le k\le \frac{d}{2}+1$,
\begin{equation}\label{eq2.11}
|\frac{d^k}{d\lambda^k}r^{\pm}_d(\lambda, |x-y|)|\leq C_k\lambda^{d-\frac12-k}|x-y|^{\frac32},\,\,\,\text{for}\,\,0<\lambda<1,\,\,x,y\in\mathbb{R}^d.
\end{equation}
We remark that the estimates for $r^{\pm}_d(\lambda, |x-y|)$ in even dimensions is not sharp, but it is enough for our use here.

The following lemma plays an important role  in the high energy part of the proof of Theorem \ref{thm-1.1}.
\begin{lemma}\label{lm2.1}
Let $d\ge 1$ and $\varphi$ satisfy  \eqref{eq1.2} in \textbf{Condition $H_{\alpha}$}. Assume that $0\le k\le [\frac{d}{2}]+1$ and $\lambda_0>0$. Then one has
\begin{equation}\label{eq2.11.2}
  |\frac{d^k}{d\lambda^k}F^\pm(\lambda^2)|\leq C(k, \lambda_0)\cdot M^2\cdot\lambda^{-1},\,\,\,\quad \text{for}\,\,\lambda>\lambda_0,
\end{equation}
where $M>0$ is the same constant given in \eqref{eq1.2}.
\end{lemma}
\begin{proof}
First, it'well known that the free resolvent $R^{\pm}_0(\lambda^2)$ satisfies the following limiting absorption principle (see e.g. in \cite{KK}):
$$
\|\frac{d^k}{d\lambda^k}R^{\pm}_0(\lambda^2)\|_{L^2_{\sigma}-L^2_{-\sigma}}\leq C(k, \lambda_0)\cdot\lambda^{-1},\,\,\,\,\,\,\, \lambda>\lambda_0,\,\quad \sigma>k+\frac12.
$$
Second, it follows from our assumption on $\varphi$ that when  $\sigma=[\frac{d}{2}]+\frac32+\epsilon$ and  $\epsilon>0$ is sufficiently  small, then $\|\varphi\|_{L^2_{\sigma}}\le CM$ for some absolute constant $C>0$.
Therefore the desired estimate \eqref{eq2.11.2} follows immediately from the definition \eqref{eq1.11.1}.
\end{proof}

The following three  lemmas are very useful in the low energy part of the proof of Theorem \ref{thm-1.1}.
\begin{lemma}\label{lm2.2}
Let $d\ge 1$ and $\varphi$ satisfy  \eqref{eq1.2} in \textbf{Condition $H_{\alpha}$}. Then there exists some small constant $\lambda_0\in (0, 1)$  depending on $\varphi$ such that for  $0<\lambda<\lambda_0$, we have

(\romannumeral1)  If $d=1, 2$ and $\int_{\mathbb{R}^d}{\varphi(x)\,dx}\ne 0$, then there exist absolute constants $0<C_1(\varphi)<C_2(\varphi)$ such that
\begin{equation}\label{eq2.14}
  \frac{C_1(\varphi)}{\lambda}\leq |\Im F^\pm(\lambda^2)|\leq  \frac{C_2(\varphi)}{\lambda},\qquad\qquad\quad\,\,\,\,\,\,\,\,\text{when}\,\,\,  d=1,
\end{equation}
and
\begin{equation}\label{eq2.15}
 C_1(\varphi)\cdot\log{\frac{2}{\lambda}}\leq \Re F^\pm(\lambda^2)\leq C_2(\varphi)\cdot\log{\frac{2}{\lambda }},\qquad\,\,\,\text{when}\,\,\, d=2.
\end{equation}
(\romannumeral2) If $d=1, 2$ and $\int_{\mathbb{R}^d}{\varphi(x)\,dx}=0$, or if $d\ge 3$ , then there exists an absolute constant $C_3(\varphi)>0$ such that
\begin{equation}\label{eq2.16}
  \Re F^\pm(\lambda^2)\ge   C_3(\varphi).
\end{equation}

\end{lemma}

\begin{proof}
First we prove (\romannumeral1). When $d=1$,  it follows from \eqref{eq2.4} that
\begin{equation}\label{eq2.16.1}
		F^\pm(\lambda^2)=\frac{\pm i}{2\lambda}\left(\int{\varphi(x)\,dx}\right)^2-\frac12\left\langle\int{ |x-y|\varphi(y)\,dy},\,\varphi(x) \right\rangle +\left\langle \int{r^{\pm}_1(\lambda, |x-y|)\varphi(y)\,dy},\,\varphi(x)\right\rangle.
\end{equation}
Using \eqref{eq2.5} (with $k=0$) and the decay assumption \eqref{eq1.2}, we obtain that for some absolute constant $C>0$,
$$
\left|\langle \int{r^{\pm}_1(\lambda, |x-y|)\varphi(y)\,dy},\,\varphi(x)\rangle\right|\le CM^2,\,\,\,\,\,0<\lambda<1.
$$
Since $\int_{\mathbb{R}}{\varphi(x)\,dx}\ne 0$, we choose $\lambda_0=\min\left\{\frac{(\int{\varphi(x)\,dx})^2}{4CM^2},\,\frac12\right\}$, then for any $0<\lambda<\lambda_0$, by \eqref{eq2.16.1} we have
$$
  \frac{1}{4\lambda}\left(\int{\varphi(x)\,dx}\right)^2\leq |\Im F^\pm(\lambda^2)|\leq  \frac{1}{\lambda}\left(\int{\varphi(x)\,dx}\right)^2,
$$
which implies \eqref{eq2.14}.  Similarly, when $d=2$ and $\int_{\mathbb{R}^2}{\varphi(x)\,dx}\ne 0$, it follows from \eqref{eq2.6}-\eqref{eq2.7} that  there exists some small constant $\lambda_0\in (0, 1)$  depending on $\varphi$ such that for  $0<\lambda<\lambda_0$, we have
$$
 -\frac{1}{2\pi}\log{\frac{\lambda}{2}}\left(\int{\varphi(x)\,dx}\right)^2\leq \Re F^\pm(\lambda^2)\leq  -\frac{1}{\pi}\log{\frac{\lambda}{2}}\left(\int{\varphi(x)\,dx}\right)^2.$$
 This indicates \eqref{eq2.15}.

Next, we prove (\romannumeral2). When $d=1, 2$ and $\int_{\mathbb{R}^d}{\varphi(x)\,dx}=0$, then $\hat{\varphi}(0)=0$, our decay assumption on $\varphi$ implies that $\hat{\varphi}\in C^1$ and  $|\hat{\varphi}(\xi)|\leq C|\xi|$ when $|\xi|\ll 1$. Thus
\begin{equation}\label{eq2.16.2}
  \lim_{\lambda\rightarrow 0}{F^\pm(\lambda^2)}=\langle(-\Delta)^{-1}\varphi,\, \varphi\rangle=\int{\frac{|\hat{\varphi}(\xi)|^2}{|\xi|^2}d\xi}>0.
\end{equation}
Note that this corresponds to the second term $-\frac12\left\langle\int{ |x-y|\varphi(y)\,dy},\,\varphi(x) \right\rangle $ in \eqref{eq2.16.1} when $d=1$ and $-\frac{1}{2\pi}\left\langle\int{ \log|x-y|\varphi(y)\,dy},\,\varphi(x) \right\rangle $ by \eqref{eq2.6} when $d=2$.
Then similar to the proof of  (\romannumeral1), \eqref{eq2.16} follows from \eqref{eq2.4} and \eqref{eq2.6}. When $d\ge 3$, the property \eqref{eq2.16.2} also holds, then  \eqref{eq2.16} follows from \eqref{eq2.8} and \eqref{eq2.10}. This completes the proof.
\end{proof}

\begin{lemma}\label{lm2.3}
Let $d\ge 1$ and $\varphi$ satisfy  \eqref{eq1.2} in \textbf{Condition $H_{\alpha}$}. For $0<\lambda<\frac12$, there exist constants $C_0, C_1, \ldots, C_k$  independent of $\varphi$ such that when $d=1$,
\begin{equation}\label{eq2.17}
 |\frac{d^k}{d\lambda^k}F^\pm(\lambda^2)|\leq
  \begin{cases}
 C_k\cdot M^2\cdot \lambda^{-1-k},\qquad k=0, 1,\quad~~\qquad\text{if}\,\,\int_{\mathbb{R}}{\varphi(x)\,dx}\ne 0;\\
 C_k\cdot M^2,\,\,\,\,\,\qquad\qquad k=0, 1,\quad~~\qquad\text{if}\,\,\int_{\mathbb{R}}{\varphi(x)\,dx}=0.
 \end{cases}
\end{equation}
When $d=2$, and $\int_{\mathbb{R}^2}{\varphi(x)\,dx}\ne 0$,
\begin{equation}\label{eq2.17.1}
 |\frac{d^k}{d\lambda^k}F^\pm(\lambda^2)|\leq
 \begin{cases} C_0\cdot M^2\cdot \log{\frac{1}{\lambda}},\qquad\quad~~~k=0;\\
 C_k\cdot M^2\lambda^{-k},\,\,\quad\qquad\quad~~~k=1,2.
 \end{cases}
\end{equation}
When $d=2$  and $\int_{\mathbb{R}^2}{\varphi(x)\,dx}= 0$,
\begin{equation}\label{eq2.17.1'}
  |\frac{d^k}{d\lambda^k}F^\pm(\lambda^2)|\leq C_k\cdot M^2\cdot\lambda^{-k},\,\,\quad\qquad\quad~~~k=0, 1, 2.
\end{equation}
When $d\ge 3$ is odd,
\begin{equation}\label{eq2.17.2}
 |\frac{d^k}{d\lambda^k}F^\pm(\lambda^2)|\leq C_k\cdot M^2,\,\,\,\,\,\,\qquad\quad~~~~0\leq k\leq \frac{d+1}{2}.
\end{equation}
When $d\ge4$ is even,
\begin{equation}\label{eq2.17.3}
 |\frac{d^k}{d\lambda^k}F^\pm(\lambda^2)|\leq
 \begin{cases}
  C_k\cdot M^2,\qquad\quad~~~&0\leq k\leq \frac{d}{2}-1;\\
 C_{k}\cdot M^2\cdot \lambda^{-1},\qquad\quad~~~&k=\frac{d}{2},\,\frac{d}{2}+1.
 \end{cases}
 \end{equation}
\end{lemma}
\begin{proof}
By assumption \eqref{eq1.2},  \eqref{eq2.17} follows from \eqref{eq2.4} and \eqref{eq2.5}; \eqref{eq2.17.1}, \eqref{eq2.17.1'} follow from \eqref{eq2.6} and \eqref{eq2.7}; \eqref{eq2.17.2} follows from \eqref{eq2.8} and \eqref{eq2.9}; finally \eqref{eq2.17.3} follows from \eqref{eq2.10} and \eqref{eq2.11}.
\end{proof}
\begin{lemma}\label{lm2.4}
Let $d\ge 1$ and $\varphi$ satisfy  \eqref{eq1.2} in \textbf{Condition $H_{\alpha}$}. For $0<\lambda<\frac12$, we have
\begin{equation}\label{eq2.18}
 |\frac{d^k}{d\lambda^k}\left(F^+(\lambda^2)-F^-(\lambda^2)\right)|\leq C_k\cdot M^2\cdot\lambda^{d-2-k},\qquad\quad~~~0\leq k\leq [\frac{d}{2}]+1,
\end{equation}
where each $C_k>0$ is independent of $\varphi$.
\end{lemma}

\begin{proof}
When $d=1$, \eqref{eq2.18} follows directly from \eqref{eq2.17}.

When $d=2$, by \eqref{eq2.6} we have
$$
R^{+}_0(\lambda^2, x, y)-R^{-}_0(\lambda^2, x, y)=\frac{i}{2}+r_2^+(\lambda, |x-y|)-r_2^-(\lambda, |x-y|).
$$
Then \eqref{eq2.18} follows from \eqref{eq2.7} and \eqref{eq1.2}.

When $d\ge 3$, by \eqref{eq2.8} and \eqref{eq2.10} we have
$$
R^{+}_0(\lambda^2, x, y)-R^{-}_0(\lambda^2, x, y)=C\lambda^{d-2}+r_d^+(\lambda, |x-y|)-r_d^-(\lambda, |x-y|).
$$
Therefore \eqref{eq2.18} follows from \eqref{eq2.9},  \eqref{eq2.11} and  \eqref{eq1.2}.
\end{proof}

\subsection{Estimates for some (oscillatory) integrals}

The purpose of this subsection is to provide the estimates for some (oscillatory) integrals that will be frequently used in what follows.
More precisely, in order to obtain decay in time from the integral kernel of the propagator $e^{-itH_{\alpha}}$ or $e^{-itH}$, we will reduce things to the estimates of the following one dimensional oscillatory integrals over $\Omega \subset \mathbb{R}$:
\begin{equation}\label{eq2}
I(t,x):=\int_\Omega e^{i(t\lambda^2+\lambda\cdot x)}\psi(\lambda)d\lambda,\quad t\in\mathbb{R}\setminus\{0\},\,\,\,\,x\in \mathbb{R}.
\end{equation}

\begin{lemma}\label{lm2.6}
Let $\Omega=(0, r_0)$ or $\Omega=(r_0, \infty)$ for some $r_0>0$ and $\psi(\lambda)\in S^b_K(\Omega)$ (see \eqref{eq1} for definition), where  $-1<b<2K-1$.  Then
	\begin{equation}\label{eq6}
		\begin{split}
			|I(t,x)|\leq\begin{cases}C|t|^{-\frac 12-b}\cdot |x|^{b},\,\,\,\, &\text{if}\,\,\,~|t|^{-\frac 12}|x|>1;\\
			C|t|^{-\frac{1+b}{2}}, \,\,\,\,             &\text{if}\,\,\,~|t|^{-\frac 12}|x|\leq 1,
			\end{cases}
		\end{split}
	\end{equation}
where the constant $C$ depends on $C_j$ ($0\leq j\leq K$) in \eqref{eq1}.
\end{lemma}
\begin{proof}
The estimate \eqref{eq6} follows  from \cite[Lemma 2.1, Lemma 2.2]{HHZ}, where oscillatory integrals with more general phase functions are studied. For the sake of completeness, we present a simplified proof in Appendix \ref{app}.
\end{proof}

As a corollary, we state the following result, which
 will be used in our proof of low and high energy part respectively.
\begin{corollary}\label{cor2.6}
(\romannumeral1)  Let $\Omega=(0, r_0)$  for some $r_0>0$ and $\psi(\lambda)\in S_K^b(\Omega)$ with $-1<b<2K-1$. Then
\begin{equation}\label{eq2.22}
|I(t,x)|\leq Ct^{-\frac{1+b}{2}}(1+|x|)^\frac{b}{2}.
\end{equation}
(\romannumeral2)  Let $\Omega=(r_0, \infty)$  for some $r_0>0$ and $d\in{N}$. Assume that $\psi(\lambda)=\psi_1(\lambda)\cdot
\psi_2(\lambda)$ such that $\frac{d^k}{d\lambda^k}\psi_1(\lambda)\in S_1^{0}(\Omega)$ for ecah $1 \leq k\leq [\frac d2] $ and $\psi_2(\lambda) \in S_{[ \frac d2] +1}^{[ \frac{d-1}{2}]}(\Omega)$,  then we have
\begin{equation}\label{eq2.23}
|I(t,x)|\leq Ct^{-\frac{d}{2}}(1+|x|)^{\frac{d-1}{2}}.
\end{equation}
\end{corollary}
\begin{proof}
	The  proof is given in Appendix \ref{app}.
\end{proof}
\begin{remark}\label{rmk2.7}
In view of the Stone's formula \eqref{eq3.1}, oscillatory integrals of the form \eqref{eq6} appear naturally in the study of dispersive estimates for Schr\"{o}dinger operators. We mention that similar integrals  have been studied in \cite[Lemma 2.4]{RS} and \cite[Lemma 2.4]{Ya2}.

\end{remark}

 The following lemma can be seen in \cite[Lemma 3.8]{GV} or \cite[Lemma 6.3]{EG10}
\begin{lemma}\label{lm2.8}
Let $d\ge 1$. Then there is some absolute constant $C>0$ such that
\begin{equation}\label{eq2.20}
\int_{\mathbb{R}^d}|x-y|^{-k}\langle y\rangle^{-l}\,dy\leq C\langle x\rangle^{-\min\{k,\, k+l-d\}},
\end{equation}
provided  $l\ge 0$, $0\le k<d$ and $k+l>d$. As a consequence, we also have
\begin{equation}\label{eq2.21}
\int_{\mathbb{R}^d}\int_{\mathbb{R}^d}|x-y|^{-\sigma}\langle x-\tau\rangle^{-\beta_1}\langle y\rangle^{-\beta_2}\,dxdy\leq C\langle\tau\rangle^{-\sigma},
\end{equation}
provided $\beta_1>d$,  $\beta_2>d$ and $0\le \sigma<d$.
\end{lemma}

\section{$L^1-L^{\infty}$ estimates for rank one perturbations--Proof of Theorem \ref{thm-1.1}}\label{sec3}
Inspired by the dispersive estimates for Schr\"{o}dinger operators $-\Delta+V$ (see e.g. \cite{RS,Sch}), our starting point is the spectral representation of the propagator via the Stone's formula:
\begin{equation}\label{eq3.1}
  \langle e^{-itH_{\alpha}}P_{ac}f, g\rangle=\frac{1}{2\pi i}\int_0^{\infty}e^{-it\lambda}\langle [R^{+}_{\alpha}(\lambda)-R^{-}_{\alpha}(\lambda)]f,\, g\rangle\,d\lambda,
\end{equation}
where $f, g$ are Schwartz functions, $P_{ac}$ denotes the projection onto the absolutely continuous spectrum of $H_{\alpha}$, and $R^{\pm}_{\alpha}(\lambda):=(H_{\alpha}-(\lambda\pm i0))^{-1}$ is the resolvent of the perturbed operator.

Compared with the case of Schr\"{o}dinger operators, our main difference here is the following Aronszajn-Krein formula (see \cite{Simon}): For $z\in\mathbb{C}\setminus[0, \infty)$,
\begin{equation}\label{eq3.2}
 R_{\alpha}(z)=R_{0}(z)-\frac{\alpha}{1+\alpha F(z)}R_{0}(z)\varphi\left\langle R_{0}(z)\cdot, \varphi\right\rangle ,
\end{equation}
where $F(z)$ is given by \eqref{eq1.1.1}.
From Lemma \ref{lmA.1} we know that the spectrum of $H_{\alpha}$  is purely absolutely continuous. Moreover,
plugging the Aronszajn-Krein formula \eqref{eq3.2} into \eqref{eq3.1} and  noting that
\begin{equation}\label{eq3.2.1}
e^{-itH_{0}}=\frac{1}{2\pi i}\int_0^{\infty}e^{-it\lambda}(R^{+}_{0}(\lambda)-R^{-}_{0}(\lambda))\,d\lambda,
\end{equation}
then the change of variables $\lambda\rightarrow\lambda^2$ in \eqref{eq3.1} leads to
\begin{align}\label{eq3.3}
e^{-itH_{\alpha}}-e^{-itH_{0}}&=-\frac{1}{\pi i}\int_0^{\infty}e^{-it\lambda^2}\frac{\alpha}{1+\alpha F^+(\lambda^2)}R^{+}_{0}(\lambda^2)\varphi\langle R^{+}_{0}(\lambda^2)\, \cdot, \varphi\rangle\lambda\,d\lambda\nonumber\\
&+\frac{1}{\pi i}\int_0^{\infty}e^{-it\lambda^2}\frac{\alpha}{1+\alpha F^-(\lambda^2)}R^{-}_{0}(\lambda^2)\varphi\langle R^{-}_{0}(\lambda^2)\, \cdot, \varphi\rangle\lambda\,d\lambda\nonumber\\
&:=U_d^+-U_d^-.
\end{align}
Hence it suffices to establish the $L^1-L^{\infty}$ estimates for the operator $U_d^{+}-U_d^{-}$. To this end, we first fix a small $0<\lambda_0<1$ (the specific value will be determined later), then we choose a smooth cut-off function $\chi(\lambda)$ such that
\begin{equation}\label{eq3.3.1}
  \chi(\lambda)=1\,\,\,\,\,\text{if}\,\,\,\lambda<\frac{\lambda_0}{2}; \qquad \chi(\lambda)=0\,\,\,\,\,\text{if}\,\,\,\lambda>\lambda_0.
\end{equation}
As usual we split $U_d^{\pm}$ into the  high energy part
\begin{equation}\label{eq3.4}
  U_d^{\pm, h}=\frac{-\alpha}{\pi i}\int_0^{\infty}e^{-it\lambda^2}\frac{1-\chi(\lambda)}{1+\alpha F^{\pm}(\lambda^2)}R^{\pm}_{0}(\lambda^2)\varphi\langle R^{\pm}_{0}(\lambda^2)\, \cdot, \varphi\rangle\lambda\,d\lambda,
\end{equation}
 and the  low energy part
\begin{equation}\label{eq3.5}
   U_d^{\pm, l}=\frac{-\alpha}{\pi i}\int_0^{\infty}e^{-it\lambda^2}\frac{\chi(\lambda)}{1+\alpha F^{\pm}(\lambda^2)}R^{\pm}_{0}(\lambda^2)\varphi\langle R^{\pm}_{0}(\lambda^2)\, \cdot, \varphi\rangle\lambda\,d\lambda.
\end{equation}

 The rest of this section is devoted to the study of $L^1-L^{\infty}$ bounds of $U_d^{\pm, h}$ and $U_d^{+, l}-U_d^{-, l}$ respectively. Theorem \ref{thm-1.1} follows immediately from Theorem \ref{thm3.1} and \ref{thm3.2} below.

\subsection{The high-energy contribution}\label{sec3.1}
\begin{theorem}\label{thm3.1}
Assume $d\ge 1$. Let $\varphi$ satisfy the  \textbf{Condition $H_{\alpha}$}. Then,
\begin{equation}\label{eq3.6}
\|U_d^{\pm, h}\|_{L^1-L^{\infty}}\leq C(d, c_0, \lambda_0, \alpha, \varphi)\cdot t^{-\frac{d}{2}}, \,\,\,\text{for}\,\,\, t>0,
\end{equation}
where $C(d, c_0, \lambda_0,  \alpha, \varphi)>0$ is a constant depending  on $d, c_0, \lambda_0$, $\alpha, \varphi$ and can be chosen as
$$
C(d, c_0, \lambda_0,  \alpha, \varphi)=C(d, c_0, \lambda_0)\cdot M^{2([\frac{d}{2}]+2)}\cdot \alpha (1+\alpha)^{[\frac{d}{2}]+1}
$$
for some absolute constant $C(d, c_0, \lambda_0)>0$ independent with $\alpha$ and $\varphi$.
\end{theorem}

\begin{proof}
Let $U_d^{\pm, h}(t, x, y)$ denote the integral kernel of the operator $U_d^{\pm, h}$. Our aim is to prove point-wise estimates for $U_d^{\pm, h}(t, x, y)$, $t>0$.
To this end, it's convenient to define
\begin{equation}\label{eq3.4.1}
F^{\alpha, h}_{\pm}(\lambda^2):=\frac{1-\chi(\lambda)}{1+\alpha F^{\pm}(\lambda^2)}.
\end{equation}
First it follows from \eqref{eq1.4}  that
$|F^{\alpha, h}_{\pm}(\lambda^2)|\leq \frac{1}{c_0}$.
Next, we notice that for any $m=1,2,\cdots$,
\begin{equation}\label{eq.11.1.5}
\frac{d^m}{d\lambda^m}\left( \frac{1}{1+\alpha F^{\pm}(\lambda^2)}\right) =\sum_{n=1}^m\sum_{\mu_1+\mu_2+\cdots+\mu_n=m}
C_{\mu_1,\mu_2,\ldots,\mu_n}\,\cdot\,\frac{\alpha^n\prod^n_{s=1}\frac{d^{\mu_s}}{d\lambda^{\mu_s}}F^{\pm}(\lambda^2)}{\left( 1+\alpha F^{\pm}(\lambda^2)\right)^{n+1}},
\end{equation}
where $\mu_1, \mu_2,\cdots, \mu_n\ge 1$.
Moreover, by Lemma  \ref{lm2.1} and \eqref{eq1.4}, each  term in the RHS of \eqref{eq.11.1.5} satisfies
$$\left|\frac{\alpha^n\prod^n_{s=1}\frac{d^{\mu_s}}{d\lambda^{\mu_s}}F^{\pm}(\lambda^2)}{\left( 1+\alpha F^{\pm}(\lambda^2)\right)^{n+1}}\right|\leq CM^{2n}\alpha^{n}\lambda^{-n}\leq CM^{2m}(1+\alpha)^{m}\lambda^{-1},\quad \,\,\,\lambda>\frac{\lambda_0}{2},$$
where $C>0$ is an absolute constant independent of $\varphi$. In the inequality above,  we have used the fact that $M\geq\|\left\langle x\right\rangle^{-\delta}\|_{L^2}^{-1}$, which in turn follows from the assumption \eqref{eq1.2} and  $\|\varphi\|_{L^2}=1$.
Therefore, choose  $m=1, 2, \cdots, [\frac{d}{2}]+1$ in \eqref{eq.11.1.5} and notice that $1-\chi(\lambda)\in S^0_k\left((\frac{\lambda_0}{2}, \infty)\right)$ for any $k=0,1,2,\cdots$,  we obtain the following relation:
\begin{equation}\label{eq1.4.1}
M^{-2([\frac{d}{2}]+1)}(1+\alpha)^{-[\frac{d}{2}]-1}\frac{d^k}{d\lambda^k}F^{\alpha, h}_{\pm}(\lambda^2)\in S^0_1\left((\frac{\lambda_0}{2}, \infty)\right),\,\,\,\,\,\,\,\, 0\leq k\leq [\frac{d}{2}].
\end{equation}
Here, it deserves to mention that  in \eqref{eq1.4.1}, we view $M$ and $\alpha$ as parameters and it holds in the sense of \eqref{eq1.111}, i.e., the $L^{\infty}$ norm of the $k$-th derivative of the function in \eqref{eq1.4.1} is uniform with $M$ and $\alpha$. This will determine the upper bound of the constant in \eqref{eq3.6}, and we shall keep using this notation in what follows (see Lemma \ref{lm3.3}).

Due to the different behaviors of the resolvent in different dimensions, we further divide the high energy part into the following four cases.

$\bullet$ {\emph{case 1: d=1 \& d=3}}. First we consider the simplest case $d=1$, it follows from \eqref{eq2.1} that
\begin{equation}\label{eq3.7}
  U_1^{\pm, h}(t, x, y)=\frac{\alpha}{4\pi i}\int_{\mathbb{R}}\int_{\mathbb{R}}{I_1^{\pm, h}(t, |x-x_1|, |x_2-y|)\varphi(x_1)\varphi(x_2)\,dx_1dx_2},
\end{equation}
where
\begin{equation*}\label{eq3.8}
  I_1^{\pm, h}(t, |x-x_1|, |x_2-y|):=\int_0^{\infty}{e^{-it\lambda^2\pm i\lambda(|x-x_1|+ |x_2-y|)}F^{\alpha, h}_{\pm}(\lambda^2)\frac{d\lambda}{\lambda}}.
\end{equation*}
Hence by \eqref{eq1.4.1} we have
\begin{equation*}
\frac{1}{M^{2}(1+\alpha)\lambda}\cdot F^{\alpha, h}_{\pm}(\lambda^2)\in S^0_1\left((\frac{\lambda_0}{2}, \infty)\right).
\end{equation*}
Then we apply the oscillatory integral lemma \ref{lm2.6} with $\Omega=(\lambda_0/2, \infty)$, $b=0$ and $K=1$ to derive
\begin{align*}\label{eq3.11}
\left|I_1^{\pm, h}(t, |x-x_1|, |x_2-y|)\right|\leq C(1+\alpha)M^2\cdot t^{-\frac12}.
\end{align*}
This, along with \eqref{eq3.7}, yields that
\begin{align*}
\sup_{x, y\in\mathbb{R}}| U_1^{\pm, h}(t, x, y)|&\leq   C\alpha(1+\alpha)M^2\cdot t^{-\frac12}\int_{\mathbb{R}}\int_{\mathbb{R}}{|\varphi(x_1)||\varphi(x_2)|\,dx_1dx_2} \\
&\leq   C\alpha(1+\alpha)M^4\cdot t^{-\frac12},
\end{align*}
where the last inequality follows from the decay assumption \eqref{eq1.2}.
Therefore we prove \eqref{eq3.6} for $d=1$.

When $d=3$, we use the expression of  the free resolvent  \eqref{eq2.3} to write
\begin{equation}\label{eq3.12}
  U_3^{\pm, h}(t, x, y)=\frac{-\alpha}{16\pi^3 i}\int_{\mathbb{R}^3}\int_{\mathbb{R}^3}{I_3^{\pm, h}(t, |x-x_1|, |x_2-y|)\frac{\varphi(x_1)}{|x-x_1|}\frac{\varphi(x_2)}{|x_2-y|}\,dx_1dx_2},
\end{equation}
where
\begin{equation*}\label{eq3.13}
  I_3^{\pm, h}(t, |x-x_1|, |x_2-y|):=\int_0^{\infty}{e^{-it\lambda^2\pm i\lambda(|x-x_1|+ |x_2-y|)}F^{\alpha, h}_{\pm}(\lambda^2)\lambda\cdot d\lambda}.
\end{equation*}
We apply  (\romannumeral2) of Corollary \ref{cor2.6} with $\psi_1=M^{-4}(1+\alpha)^{-2}F^{\alpha, h}_{\pm}(\lambda^2)$ and $\psi_2=\lambda$, then one obtains that
\begin{equation}\label{eq3.15}
|I_3^{\pm, h}(t, |x-x_1|, |x_2-y|)|\leq C(1+\alpha)^2M^4(1+|x-x_1|+ |x_2-y|)\cdot t^{-\frac32}.
\end{equation}
This, together with \eqref{eq3.12}, implies that
\begin{align*}
\sup_{x, y\in\mathbb{R}^3}|U_3^{\pm, h}(t, x, y)| &\leq C\alpha(1+\alpha)^2M^4 t^{-\frac32}\cdot \int_{\mathbb{R}^3}\int_{\mathbb{R}^3}{\frac{1+|x-x_1|+ |x_2-y|}{|x-x_1||x_2-y|}\cdot|\varphi(x_1)||\varphi(x_2)|\,dx_1dx_2}\nonumber\\
&\leq C\alpha(1+\alpha)^2M^6\cdot t^{-\frac32},
\end{align*}
where the last inequality follows from the decay assumption \eqref{eq1.2} and Lemma \ref{lm2.8}.
Therefore we prove \eqref{eq3.6} for $d=3$.

$\bullet$ {\emph{case 2: $d>3$, odd.}} By \eqref{eq2.3} and \eqref{eq3.4}, the integral kernel $U_d^{+, h}(t, x, y)$ can be written as a linear combination of
\begin{align}\label{eq3.16}
  \frac{\alpha}{\pi i}\int_{\mathbb{R}^d}\int_{\mathbb{R}^d}\int_0^{\infty}{e^{-it\lambda^2+ i\lambda(|x-x_1|+ |x_2-y|)}F^{\alpha, h}_{+}(\lambda^2)\lambda^{k_1+k_2+1} d\lambda}\nonumber\\
  \cdot \frac{\varphi(x_1)}{|x-x_1|^{d-2-k_1}}\frac{\varphi(x_2)}{|x_2-y|^{d-2-k_2}}\,dx_1dx_2,
\end{align}
where $1\le k_1,\,k_2\le \frac{d-3}{2}$.
We still want to use oscillatory integral estimate \eqref{eq2.23} to obtain the upper bound $t^{-\frac d2}$. But when
$k_1+k_2+1$ is larger than $\frac{d-1}{2}$,  it no longer satisfies the conditions in Corollary \ref{cor2.6}. 
In order to overcome this difficulty, we borrow ideas from  \cite{EG10}. Since the function $\varphi$ admits certain smoothness when $d>3$, then one can use integration by parts arguments (note that there is an oscillating term $e^{i\lambda(|x-x_1|+ |x_2-y|)}$), and this is very different from the case $d\le 3$. To be more precise, let us define two operators
\begin{equation}\label{eq3.16.1}
  L_{x_1}:=\frac{1}{i\lambda}\frac{x_1-x}{|x_1-x|}\cdot\nabla_{x_1},\qquad\qquad L_{x_2}:=\frac{1}{i\lambda}\frac{x_2-y}{|x_2-y|}\cdot\nabla_{x_2},
\end{equation}
and their dual operators
\begin{equation}\label{eq3.16.2}
  L^*_{x_1}:=\frac{i}{\lambda}\nabla_{x_1}(\frac{x_1-x}{|x-x_1|}),\qquad\qquad L^*_{x_2}:=\frac{i }{\lambda}\nabla_{x_2}(\frac{x_2-y}{|x_2-y|}).
\end{equation}
Set
\begin{equation}\label{eq3.16.3}
 n_i=\max\left\{[k_i-\frac{d-3}{4}]+1,\, 0\right\},\,\,\,\,\,i=1,2.
\end{equation}
Then we have  $n_i\le [\frac{d-3}{4}]+1$ for $i=1, 2$. Now we apply $L_{x_1}$ $n_1$ times to the exponential $e^{i\lambda|x-x_1|}$,  and  apply $L_{x_2}$  $n_2$ times to the exponential $e^{i\lambda|x_2-y|}$. This yields that
\begin{align}\label{eq3.17}
& \int_{\mathbb{R}^d}\int_{\mathbb{R}^d}{e^{i\lambda(|x-x_1|+ |x_2-y|)}\frac{\varphi(x_1)}{|x-x_1|^{d-2-k_1}}\frac{\varphi(x_2)}{|x_2-y|^{d-2-k_2}}\,dx_1dx_2} \nonumber\\
&= \int_{\mathbb{R}^d}\int_{\mathbb{R}^d}{e^{i\lambda(|x-x_1|+ |x_2-y|)}(L^*_{x_1})^{n_1}\frac{\varphi(x_1)}{|x-x_1|^{d-2-k_1}}(L^*_{x_2})^{n_2}\frac{\varphi(x_2)}{|x_2-y|^{d-2-k_2}}\,dx_1dx_2}.
\end{align}
 We remark that, strictly speaking, the integration by parts should be performed with smooth cut off functions around the singularities. We omit it here for simplicity and refer to Remark \ref{rmk3.2} for the justification.
 Then the integral in \eqref{eq3.16} can be further written as
\begin{equation}\label{eq3.18.1}
  C\alpha\int_{\mathbb{R}^{2d}}{I_d^{+, h}(t, |x-x_1|, |x_2-y|)G_1(x, x_1)G_2(x_2, y)\,dx_1dx_2},
\end{equation}
where
\begin{equation}\label{eq3.19}
 I_d^{+, h}(t, |x-x_1|, |x_2-y|)=\int_0^{\infty}{e^{-it\lambda^2+ i\lambda(|x-x_1|+ |x_2-y|)}F^{\alpha, h}_{+}(\lambda^2)\lambda^{k_1+k_2+1-n_1-n_2} d\lambda}.
\end{equation}
We first estimate the functions $G_1, G_2$. By induction one has
\begin{align}\label{eq3.19.1}
&G_1(x, x_1)\nonumber\\
&=\sum_{i_1\cdots i_{n_1}}\sum_{\beta}{C_{\beta }\partial_{x_1}^{\beta_1}(\frac{x-x_1}{|x-x_1|}\cdot e_{i_1})\cdots \partial_{x_1}^{\beta_{n_1}}(\frac{x-x_1}{|x-x_1|}\cdot e_{i_{n_1}})\partial_{x_1}^{\beta_{n_1+1}}(\frac{1}{|x-x_1|^{d-2-k_1}})\partial_{x_1}^{\beta_{n_1+2}}\varphi},
\end{align}
where $1\le i_j\leq d$, and $e_{i_j}$ denotes the $i_j$-th unit vector. The multi-index $\beta_j\in \mathbb{N}_0^d$ satisfies $\beta_1+\cdots+\beta_{n_1+2}= \sum_{j=1}^{n_1}e_{i_j}$. Observe that
\begin{equation*}\label{eq3.19.2}
  |\partial_{x_1}^{\beta}(\frac{x-x_1}{|x-x_1|}\cdot e_{i_j})|\leq C|x-x_1|^{-|\beta|},
\end{equation*}
\begin{equation*}\label{eq3.20}
 |\partial_{x_1}^{\beta}(\frac{1}{|x-x_1|^{d-2-k_1}})|\leq\frac{C}{|x-x_1|^{d-2-k_1-|\beta|}}.
\end{equation*}
Notice that $\sum_{j=1}^{n_1+2}{|\beta_j|}=n_1$, we consider two extreme cases: The first is that when $|\beta_{n_1+2}|=n_1$, thus $\beta_1=\cdots=\beta_{n_1+1}=0$, and
\begin{equation*}\label{eq3.22}
|\partial_{x_1}^{\beta_1}(\frac{x-x_1}{|x-x_1|}\cdot e_{i_1})\cdots\partial_{x_1}^{\beta_{n_1+1}}(\frac{1}{|x-x_1|^{d-2-k_1}})|\leq \frac{C}{|x-x_1|^{d-2-k_1}};
\end{equation*}
the second is that when $|\beta_{n_1+2}|=0$, then
\begin{equation*}\label{eq3.23}
|\partial_{x_1}^{\beta_1}(\frac{x-x_1}{|x-x_1|}\cdot e_{i_1})\cdots\partial_{x_1}^{\beta_{n_1+1}}(\frac{1}{|x-x_1|^{d-2-k_1}})|\leq \frac{C}{|x-x_1|^{d-2-k_1+n_1}}.
\end{equation*}
Observe that $\frac{d-1}{2}\le d-2-k_1, d-2-k_1+n_1<d-1$.
Combining these with \eqref{eq3.19.1} and our decay assumption \eqref{eq1.3}, we obtain that
\begin{equation}\label{eq3.24}
  |G_1(x, x_1)|\leq C\cdot M\cdot \left(|x-x_1|^{-d+1}+|x-x_1|^{-\frac{d-1}{2}}\right)\cdot\langle x_1 \rangle^{-\delta},\,\,\,\,\delta>d+\frac{3}{2}.
\end{equation}
 $G_2$ is defined in the same way as in \eqref{eq3.19.1}, similarly, we have
\begin{equation}\label{eq3.25}
  |G_2(x_2, y)|\leq C\cdot M\cdot\left(|x_2-y|^{-d+1}+|x_2-y|^{-\frac{d-1}{2}}\right)\cdot\langle x_2 \rangle^{-\delta},\,\,\,\,\delta>d+\frac{3}{2}.
\end{equation}

Next we estimate the oscillatory integral \eqref{eq3.19}. By the definitions of  $n_1$ and $n_2$ in \eqref{eq3.16.3}, one has
 $$ k_1+k_2+1-n_1-n_2\leq \frac{d-1}{2}.$$
Applying  (\romannumeral2) of Corollary \ref{cor2.6} with $\psi_1=M^{-d-1}(1+\alpha)^{-\frac{d+1}{2}}F^{\alpha, h}_{\pm}(\lambda^2)$ and $\psi_2=\lambda^{k_1+k_2+1-n_1-n_2}$,   then we are able to control the integral \eqref{eq3.18.1} by
\begin{align}\label{eq3.31}
	& C\alpha \left|\int_{\mathbb{R}^{2d}}{I_d^{+, h}(t, |x-x_1|, |x_2-y|)G_1(x, x_1)G_2(x_2, y)\,dx_1dx_2}\right|\nonumber\\
	& \leq C\alpha(\alpha+1)^{\frac{d+1}{2}} M^{d+1} t^{-\frac{d}{2}}\int_{\mathbb{R}^{2d}}{(1+|x-x_1|+|x_2-y|)^{\frac{d-1}{2}}|G_1(x, x_1)||G_2(x_2, y)|\,dx_1dx_2}.
\end{align}
By our estimates on $G_1$ and $G_2$ in \eqref{eq3.24} and \eqref{eq3.25}, we obtain from Lemma \ref{lm2.8} that
\begin{equation}\label{eq3.32}
 	\sup_{x, y\in\mathbb{R}^d}\int_{\mathbb{R}^{2d}}{(1+|x-x_1|+|x_2-y|)^{\frac{d-1}{2}}|G_1(x, x_1)||G_2(x_2, y)|\,dx_1dx_2}\leq C\cdot M^2<\infty.
 \end{equation}
Therefore the desired estimate for $U_d^{+, h}(t, x, y)$  follows  by  \eqref{eq3.16},  \eqref{eq3.18.1}, \eqref{eq3.31} and \eqref{eq3.32}. The same argument also works for  $U_d^{-, h}(t, x, y)$ and   the proof of \eqref{eq3.6} for odd dimensions $d>3$ is complete.

$\bullet$ {\emph{case 3: $d=2$}.}
By \eqref{eq3.50} and \eqref{eq3.4} we write
\begin{equation}\label{eq3.26.1}
  U_2^{\pm, h}(t, x, y)=\frac{-\alpha}{\pi i}\int_{\mathbb{R}^2}\int_{\mathbb{R}^2}{I_2^{\pm, h}(t, |x-x_1|, |x_2-y|)\varphi(x_1)\varphi(x_2)\,dx_1dx_2},
\end{equation}
where
\begin{align}\label{eq3.27.1}
& I_2^{\pm, h}(t, |x-x_1|, |x_2-y|)\nonumber \\
&=\int_0^{\infty}{e^{-it\lambda^2\pm i\lambda(|x-x_1|+ |x_2-y|)}F^{\alpha, h}_{\pm}(\lambda^2)w_{\pm, >}(\lambda|x-x_1|)w_{\pm, >}(\lambda|x_2-y|)\cdot\lambda d\lambda}\nonumber \\
&+ \int_0^{\infty}{e^{-it\lambda^2\pm i\lambda|x-x_1|}F^{\alpha, h}_{\pm}(\lambda^2)w_{\pm, >}(\lambda|x-x_1|)w_{\pm, <}(\lambda|x_2-y|)\cdot\lambda d\lambda}\nonumber \\
&+ \int_0^{\infty}{e^{-it\lambda^2\pm i\lambda|x_2-y|}F^{\alpha, h}_{\pm}(\lambda^2)w_{\pm, <}(\lambda|x-x_1|)w_{\pm, >}(\lambda|x_2-y|)\cdot\lambda d\lambda}\nonumber \\
&+\int_0^{\infty}{e^{-it\lambda^2}F^{\alpha, h}_{\pm}(\lambda^2)w_{\pm, <}(\lambda|x-x_1|)w_{\pm, <}(\lambda|x_2-y|)\cdot\lambda d\lambda}\nonumber \\
&:=I_{2, 1}^{\pm, h}+I_{2, 2}^{\pm, h}+I_{2, 3}^{\pm, h}+I_{2, 4}^{\pm, h}.
\end{align}
We first deal with the term $I_{2, 1}^{\pm, h}$. Define
\begin{equation*}\label{eq3.28.1}
 \psi_{2, 1}(\lambda):=\lambda\cdot |x-x_1|^{\frac12}\cdot w_{\pm, >}(\lambda|x-x_1|)\cdot |x_2-y|^{\frac12}\cdot w_{\pm, >}(\lambda|x_2-y|).
\end{equation*}
Observe that by \eqref{eq3.51}, one has
\begin{equation*}
 |x-x_1|^{\frac12}\cdot w_{\pm, >}(\lambda|x-x_1|),\,\,\,\,\, |x_2-y|^{\frac12}\cdot w_{\pm, >}(\lambda|x_2-y|)\in S^{-\frac12}_2\left((\frac{\lambda_0}{2}, \infty)\right).
\end{equation*}
This yields  that
\begin{equation}\label{eq3.29.1}
\psi_{2, 1}(\lambda)\in S^{0}_2\left((\frac{\lambda_0}{2}, \infty)\right).
\end{equation}
Here, we view $|x-x_1|$ and $|x_2-y|$ as parameters (similar to \eqref{eq1.4.1}) and \eqref{eq3.29.1} holds in the sense of \eqref{eq1.111}, and we shall keep using this notation in the rest of the paper.
Therefore applying (\romannumeral2) of Corollary \ref{cor2.6} with  $\psi_1=M^{-4}(1+\alpha)^{-2}F^{\alpha, h}_{\pm}(\lambda^2)$ and  $\psi_2= \psi_{2, 1}(\lambda)$ we obtain that
\begin{align*}\label{eq3.30.1}
&\sup_{x,y\in\mathbb{R}^2}\left|\frac{\alpha}{\pi i}\int_{\mathbb{R}^2}\int_{\mathbb{R}^2}{I_{2,1}^{\pm, h}(t, |x-x_1|, |x_2-y|)\varphi(x_1)\varphi(x_2)\,dx_1dx_2}\right| \nonumber\\
&\leq \frac{C\alpha(1+\alpha)^2M^4}{t}\int_{\mathbb{R}^2}\int_{\mathbb{R}^2}{(1+|x-x_1|+ |x_2-y|)^{\frac12}|x-x_1|^{-\frac12}|x_2-y|^{-\frac12}|\varphi(x_1)||\varphi(x_2)|\,dx_1dx_2} \nonumber\\
&\leq \frac{C\alpha(1+\alpha)^2M^6}{t}.
\end{align*}
The methods to deal with $I_{2, 2}^{\pm, h}$, $I_{2, 3}^{\pm, h}$ and $I_{2, 4}^{\pm, h}$ are the same. Let us only take $I_{2, 2}^{\pm, h}$ for example. In this case, we define
\begin{equation*}
 \psi_{2, 2}(\lambda):=\lambda\cdot|x-x_1|^{\frac12}\cdot w_{\pm, >}(\lambda|x-x_1|)\cdot |x_2-y|^{\frac12}\cdot w_{\pm, <}(\lambda|x_2-y|).
\end{equation*}
Now it follows from  \eqref{eq3.51} and  \eqref{eq3.25.11}   that
\begin{equation*}
\psi_{2, 2}(\lambda)\in S^{0}_2\left((\frac{\lambda_0}{2}, \infty)\right).
\end{equation*}
Applying  (\romannumeral2) of Corollary \ref{cor2.6} with  $\psi_1=M^{-4}(1+\alpha)^{-2}F^{\alpha, h}_{\pm}(\lambda^2)$ and  $\psi_2= \psi_{2, 2}(\lambda)$ we have
\begin{align*}
&\sup_{x,y\in\mathbb{R}^2}\left|\frac{\alpha}{\pi i}\int_{\mathbb{R}^2}\int_{\mathbb{R}^2}{I_{2,1}^{\pm, h}(t, |x-x_1|, |x_2-y|)\varphi(x_1)\varphi(x_2)\,dx_1dx_2}\right| \nonumber\\
&\leq \frac{C\alpha(1+\alpha)^2M^4}{t}\int_{\mathbb{R}^2}\int_{\mathbb{R}^2}{(1+|x-x_1|)^{\frac12}|x-x_1|^{-\frac12}|x_2-y|^{-\frac12}|\varphi(x_1)||\varphi(x_2)|\,dx_1dx_2} \nonumber\\
&\leq \frac{C\alpha(1+\alpha)^2M^6}{t}.
\end{align*}
Therefore we prove the case for $d=2$.

$\bullet$ {\emph{case 4: $d>2$, even.}}
 In this case,
by \eqref{eq3.50} and \eqref{eq3.4} we have
\begin{equation}\label{eq3.53}
U_d^{\pm, h}(t, x, y)= U_{d, 1}^{\pm, h}(t, x, y)+U_{d, 2}^{\pm, h}(t, x, y)+U_{d, 3}^{\pm, h}(t, x, y)+U_{d, 4}^{\pm, h}(t, x, y),
\end{equation}
where
\begin{align}\label{eq3.54.11}
U_{d, 1}^{\pm, h}(t, x, y)&:=\frac{-\alpha}{\pi i}\int_{\mathbb{R}^{2d}}\int_0^{\infty}e^{-it\lambda^2\pm i\lambda(|x-x_1|+ |x_2-y|)}F^{\alpha, h}_{\pm}(\lambda^2)\lambda^{d-1}\nonumber \\
&\cdot \frac{w_{\pm, >}(\lambda|x-x_1|)}{|x-x_1|^{\frac{d}{2}-1}}\frac{w_{\pm, >}(\lambda|x_2-y|)}{|x_2-y|^{\frac{d}{2}-1}}\varphi(x_1)\varphi(x_2)\, d\lambda dx_1dx_2, \tag{$U^h_1$}
\end{align}
\begin{align}\label{eq3.55}
U_{d, 2}^{\pm, h}(t, x, y)&:=\frac{-\alpha}{\pi i}\int_{\mathbb{R}^{2d}}\int_0^{\infty}e^{-it\lambda^2\pm i\lambda|x-x_1|}F^{\alpha, h}_{\pm}(\lambda^2)\lambda^{\frac{d}{2}}\nonumber \\
&\cdot \frac{w_{\pm, >}(\lambda|x-x_1|)}{|x-x_1|^{\frac{d}{2}-1}}\frac{w_{\pm, <}(\lambda|x_2-y|)}{|x_2-y|^{d-2}\cdot}\varphi(x_1)\varphi(x_2)\, d\lambda dx_1dx_2, \tag{$U^h_2$}
\end{align}
\begin{align}\label{eq3.56}
U_{d, 3}^{\pm, h}(t, x, y)&:=\frac{-\alpha}{\pi i}\int_{\mathbb{R}^{2d}}\int_0^{\infty}e^{-it\lambda^2\pm i\lambda|x_2-y|}F^{\alpha, h}_{\pm}(\lambda^2)\cdot\lambda^{\frac{d}{2}} \nonumber \\
&\frac{w_{\pm, <}(\lambda|x-x_1|)}{|x-x_1|^{d-2}}\frac{w_{\pm, >}(\lambda|x_2-y|)}{|x_2-y|^{\frac{d}{2}-1}}\varphi(x_1)\varphi(x_2)\,d\lambda dx_1dx_2, \tag{$U^h_3$}
\end{align}
and
\begin{align}\label{eq3.57}
U_{d, 4}^{\pm, h}(t, x, y)&:=\frac{-\alpha}{\pi i}\int_{\mathbb{R}^{2d}}\int_0^{\infty}e^{-it\lambda^2}F^{\alpha, h}_{\pm}(\lambda^2)\cdot\lambda \nonumber \\
&\frac{w_{\pm, <}(\lambda|x-x_1|)}{|x-x_1|^{d-2}}\frac{w_{\pm, <}(\lambda|x_2-y|)}{|x_2-y|^{d-2}} \varphi(x_1)\varphi(x_2)\,d\lambda dx_1dx_2. \tag{$U^h_4$}
\end{align}
 For the integral $U_{d, 1}^{\pm, h}(t, x, y)$, notice the power of $\lambda$ is larger than $\frac{d}{2}$, then  the same issue occurs as in higher odd dimensions, and we shall use the same trick as before. Indeed, let $L_{x_1}$, $L_{x_2}$ and their dual operators be given by \eqref{eq3.16.1} and  \eqref{eq3.16.2}. Now set
\begin{equation}\label{eq3.58.11}
 n=[\frac{d-3}{4}]+1.
\end{equation}
Then we rewrite  $U_{d, 1}^{\pm, h}(t, x, y)$ as
\begin{align*}\label{eq3.59}
&U_{d, 1}^{\pm, h}(t, x, y)=\frac{-\alpha}{\pi i}\int_{\mathbb{R}^{2d}}\int_0^{\infty}e^{-it\lambda^2\pm i\lambda(|x-x_1|+ |x_2-y|)}F^{\alpha, h}_{\pm}(\lambda^2)\lambda^{d-1} d\lambda\nonumber\\
& \cdot(L^*_{x_1})^{n}\left(\frac{w_{\pm, >}(\lambda|x-x_1|)\varphi(x_1)}{|x-x_1|^{\frac{d}{2}-1}}\right)(L^*_{x_2})^{n}\left(\frac{w_{\pm, >}(\lambda|x_2-y|)\varphi(x_2)}{|x_2-y|^{\frac{d}{2}-1}}\right)\,dx_1dx_2,
\end{align*}
which can be further written as a combination of the terms
\begin{equation}\label{eq3.18}
C\alpha\int_{\mathbb{R}^{2d}}{I_{d, 1, \beta_1, \beta_2}^{\pm, h}(t, |x-x_1|, |x_2-y|)\tilde{G}_1(x, x_1)\tilde{G}_2(x_2, y)\,dx_1dx_2},
\end{equation}
where $0\le |\beta_1|, |\beta_2|\leq n$,
\begin{align}\label{eq3.19.3}
I_{d, 1, \beta_1, \beta_2}^{\pm, h}(t, |x-x_1|, |x_2-y|)&=\int_0^{\infty}e^{-it\lambda^2\pm i\lambda(|x-x_1|+ |x_2-y|)}F^{\alpha, h}_{\pm}(\lambda^2)\lambda^{d-1-2n}\nonumber\\
&\cdot (\lambda|x-x_1|)^{|\beta_1|}w^{(|\beta_1|)}_{\pm, >}(\lambda|x-x_1|)\cdot (\lambda|x_2-y|)^{|\beta_2|}w^{(|\beta_2|)}_{\pm, >}(\lambda|x_2-y|) d\lambda,
\end{align}
and the estimates for $\tilde{G}_1(x, x_1)$ and $\tilde{G}_2(x_2, y)$ are similar to \eqref{eq3.24} and \eqref{eq3.25}, moreover, they satisfy
\begin{equation}\label{eq3.24.1}
  |\tilde{G}_1(x, x_1)|\leq C\cdot M\cdot \left(|x-x_1|^{-d+1}+|x-x_1|^{-\frac{d}{2}+1}\right)\cdot\langle x_1 \rangle^{-\delta},\,\,\,\,\delta>d+\frac{3}{2},
\end{equation}
and
\begin{equation}\label{eq3.25.1}
  |\tilde{G}_2(x_2, y)|\leq C\cdot M\cdot \left(|x_2-y|^{-d+1}+|x_2-y|^{-\frac{d}{2}+1}\right)\cdot\langle x_2 \rangle^{-\delta},\,\,\,\,\delta>d+\frac{3}{2}.
\end{equation}
Now we set
$$
\psi_2:=\lambda^{|\beta_1|+|\beta_2|+d-1-2n}\cdot |x-x_1|^{\frac12+|\beta_1|}|x_2-y|^{\frac12+|\beta_2|}\cdot w^{(|\beta_1|)}_{\pm, >}(\lambda|x-x_1|)w^{(|\beta_1|)}_{\pm, >}(\lambda|x_2-y|).
$$
Observe that by \eqref{eq3.51}, we have
\begin{align*}
  \lambda^{|\beta_1|}|x-x_1|^{\frac12+|\beta_1|}\cdot w^{(|\beta_1|)}_{\pm, >}(\lambda|x-x_1|)&\in S^{-\frac12}_{\frac d2+1}\left((\frac{\lambda_0}{2}, \infty)\right), \\
  \lambda^{|\beta_2|}|x_2-y|^{\frac12+|\beta_2|}\cdot w^{(|\beta_2|)}_{\pm, >}(\lambda|x_2-y|)&\in S^{-\frac12}_{\frac d2+1}\left((\frac{\lambda_0}{2}, \infty)\right).
\end{align*}
Meanwhile, by the definition of $n$ in \eqref{eq3.58.11}, we have $n\geq \frac{d-2}{4}$ when $d$ is even. Then
$$d-1-2n\leq d-1-\frac{d-2}{2}= \frac d2.$$
Putting these together, we have
\begin{equation}\label{eq3.25.2}
\psi_2(\lambda) \in S^{\frac d2-1}_{\frac d2+1}\left((\frac{\lambda_0}{2}, \infty)\right).
\end{equation}
Based on \eqref{eq3.25.2} and \eqref{eq1.4.1},  we  are able to apply (\romannumeral2) of Corollary \ref{cor2.6} with
$$
\psi_1(\lambda):= M^{-2([\frac{d}{2}]+1)}(1+\alpha)^{-[\frac{d}{2}]-1}F^{\alpha, h}_{\pm}(\lambda^2)
$$
and $\psi_2$ above to the oscillatory integral $I_{d, 1, \beta_1, \beta_2}^{+, h}(t, |x-x_1|, |x_2-y|)$.  Then the integral \eqref{eq3.18} is controlled by
\begin{align}\label{equ3.25.3}
&|\eqref{eq3.18}|\nonumber\\
&\leq C\alpha(\alpha+1)^{\frac{d}{2}+1} M^{d+2} t^{-\frac{d}{2}}\int_{\mathbb{R}^{2d}}{(1+|x-x_1|+|x_2-y|)^{\frac{d-1}{2}}\frac{|\tilde{G}_1(x, x_1)|}{|x-x_1|^{\frac12}}\frac{|\tilde{G}_2(x_2, y)|}{|x_2-y|^{\frac12}}\,dx_1dx_2}\nonumber\\
& \leq C\alpha(1+\alpha)^{\frac{d}{2}+1} M^{d+4} t^{-\frac{d}{2}},
\end{align}
where in the last inequality we use Lemma \ref{lm2.8}.

Finally, We notice that for terms  $U_{d, 2}^{\pm, h}(t, x, y)$, $U_{d, 3}^{\pm, h}(t, x, y)$ and $U_{d, 4}^{\pm, h}(t, x, y)$
are the same form as \eqref{eq3.18} above. Further, by \eqref{eq3.51}  and \eqref{eq3.52}, one has
\begin{align*}
\lambda^{\frac{d}{2}}|x-x_1|^{\frac12}w_{\pm, >}(\lambda|x-x_1|)|x_2-y|^{\frac12}w_{\pm, <}(\lambda|x_2-y|)\in S^{\frac{d}{2}-1}_{\frac{d}{2}+1}\left((\frac{\lambda_0}{2}, \infty)\right),
\end{align*}
\begin{align*}
\lambda^{\frac{d}{2}}|x-x_1|^{\frac12}w_{\pm, <}(\lambda|x-x_1|)|x_2-y|^{\frac12}w_{\pm, >}(\lambda|x_2-y|)\in S^{\frac{d}{2}-1}_{\frac{d}{2}+1}\left((\frac{\lambda_0}{2}, \infty)\right),
\end{align*}
and
\begin{align*}
\lambda w_{\pm, <}(\lambda|x-x_1|)w_{\pm, <}(\lambda|x_2-y|)\in S^{\frac{d}{2}-1}_{\frac{d}{2}+1}\left((\frac{\lambda_0}{2}, \infty)\right).
\end{align*}
They play the same role as the function in \eqref{eq3.25.2}. Therefore the estimate \eqref{equ3.25.3} also holds for $U_{d, j}^{\pm, h}(t, x, y)$, $j=2, 3, 4$. This proves the case of even dimensions $d>2$.

Summing up, we finish the proof of Theorem \ref{thm3.1}.\end{proof}
\begin{remark}[Justification of integration by parts] \label{rmk3.2}
 When performing integration by parts with respect to $x_1$ and $x_2$, we should use smooth cutoff functions around the singularities to eliminate the boundary terms. We follow the same treatment in \cite[Sect.3.3]{EG10}. More precisely, let $\rho(t)$ ($t\in\mathbb{R}$) be a smooth cutoff function around $0$, i.e., $\rho(t)=1$ when $|t|>1$ and $\rho(t)=0$ when
 $|t|<1/2$. Note that $\sup_{\epsilon>0}|\frac{d^k}{dt^k}\rho(t/\epsilon)|\leq C_k|t|^{-k}$, then
 $$\sup_{\epsilon>0}|\partial_{x_1}^{\beta}\left( (\frac{x-x_1}{|x-x_1|}\cdot e_{i_1})\rho(|x-x_1|/\epsilon)\right) |\leq C|x-x_1|^{-|\beta|},$$
 which has the same size to $|\partial_{x_1}^{\beta}(\frac{x-x_1}{|x-x_1|})|$. Therefore in the proof Theorem \ref{thm3.1} and the following sections, we can omit the cutoff function in integration by parts with respect to spacial variables.
\end{remark}

\begin{remark}\label{rmk3.1}
We mention that during the proof in the high energy part, the only property that we need for $F^{\alpha, h}_{\pm}(\lambda^2)$ is \eqref{eq1.4.1}. This will be used in the analysis of finite rank perturbations.
\end{remark}

\subsection{The low-energy contribution}\label{sec3.2}

\begin{theorem}\label{thm3.2}
Assume $d\ge 1$. Let $\varphi$ satisfy the \textbf{Condition $H_{\alpha}$}. Let  $\lambda_0$ be given in Lemma \ref{lm2.2}, then,
\begin{equation}\label{eq.l1}
\|U_d^{+, l}-U_d^{-, l}\|_{L^1-L^{\infty}}\leq C(\varphi)\cdot t^{-\frac{d}{2}}, \,\,\,\text{for}\,\,\, t>0,
\end{equation}
where $C(\varphi)$ is a constant depending on $d, \varphi$  but uniform with $\alpha$. Further,  if $0<\alpha<1$ and $\lambda_0=\frac12$, then in the case $d\ge 3$ or $d=1,2$, $\int_{\mathbb{R}^d}\varphi(x)\,dx=0$, one has
\begin{equation}\label{eq.11.11}
\|U_d^{+, l}-U_d^{-, l}\|_{L^1-L^{\infty}}\leq C\cdot M^{2([\frac{d}{2}]+3)}\cdot\alpha\cdot t^{-\frac{d}{2}}, \,\,\,\text{for}\,\,\, t>0,
\end{equation}
where the absolute constant $C>0$ independent of $\alpha$ and $\varphi$.
\end{theorem}
From Theorem \ref{thm3.2}, we see that when $\alpha$ is large, the upper bound in \eqref{eq.l1} is better than that in \eqref{eq.11.11}, while $\alpha$ is small, the upper bound in \eqref{eq.11.11} is better, and it will be used in the proof of Theorem \ref{thm1.6}. The reason that causes this difference lies in the following Lemma \ref{lm3.2} and \ref{lm3.3}. To state the result,
we first  set
\begin{equation}\label{eq.11.1}
F^{\alpha, l}_{\pm}(\lambda^2):=\frac{\chi(\lambda)}{1+\alpha F^{\pm}(\lambda^2)}.
\end{equation}
The following two lemmas consist of all the properties that we need for $F^{\alpha, l}_{\pm}$ in the low energy part.
\begin{lemma}\label{lm3.2}
Let $\lambda_0>0$ be given in Lemma \ref{lm2.2}.

(\romannumeral1)  If $d\ge 3$ or $d=1, 2$, $\int_{\mathbb{R}^d}{\varphi\,dx}=0$, then one has
\begin{align}\label{equ6.2.3}
  	\alpha F^{\alpha, l}_{\pm}(\lambda^2)&\in S_{[\frac{d}{2}]+1}^0((0,\lambda_0)).
\end{align}

(\romannumeral2)  If $d\ge 2$, then one has
\begin{align}\label{equ6.2.3.3}
 \alpha(F^{\alpha, l}_{+}(\lambda^2)-F^{\alpha, l}_{-}(\lambda^2))&\in S_{[\frac{d}{2}]+1}^{d-2}((0,\lambda_0)).
\end{align}

(\romannumeral3) Further, if $d=2$ and $\int_{\mathbb{R}^2}{\varphi\,dx}\ne 0$, then we have a gain of\,    $\log{\frac{1}{\lambda}}$ in the sense that
 \begin{align}
 \alpha\log{\frac{1}{\lambda}}\cdot F^{\alpha, l}_{\pm}(\lambda^2)&\in S_2^0\left( (0,\lambda_0)\right) ,\label{equ6.2.1}\\
 \alpha(\log{\frac{1}{\lambda}})^2\cdot (F^{\alpha, l}_{+}(\lambda^2)-F^{\alpha, l}_{-}(\lambda^2))&\in S_2^0\left( (0,\lambda_0)\right);\label{equ6.2.1.1}
 \end{align}

(\romannumeral4) if $d=1$ and $\int_{\mathbb{R}}{\varphi\,dx}\ne 0$, then  we have a gain of \, $\lambda^{-1}$ in the sense that
\begin{align}\label{eq.13.3.1}
 \frac{\alpha}{\lambda}\cdot F^{\alpha, l}_{\pm}(\lambda^2)\in S^{0}_{1}\left((0, \lambda_0)\right).
 \end{align}
\end{lemma}
\begin{proof}
First, by Lemma \ref{lm2.2}, one has for all $d\ge 1$,
\begin{equation}\label{eq.11.1.2}
 |1+\alpha F^{\pm}(\lambda^2)|\ge \alpha C_1,\,\,\,\,\,\,0<\lambda<\lambda_0.
\end{equation}
In particular, when $d=1, 2$ and $\int_{\mathbb{R}^{d}}{\varphi\,dx}\ne 0$,
\begin{equation}\label{eq.11.1.3}
   |1+\alpha F^{\pm}(\lambda^2)|\ge \frac{\alpha C_1}{\lambda},\,\,\,\,\,\,0<\lambda<\lambda_0,\,\,\,d=1,
\end{equation}
and
\begin{equation}\label{eq.11.1.4}
   |1+\alpha F^{\pm}(\lambda^2)|\ge \frac{\alpha C_1}{|\log{\lambda}|},\,\,\,\,\,\,0<\lambda<\lambda_0,\,\,\,d=2.
\end{equation}

If $d\ge 3$ or $d=1, 2$, $\int_{\mathbb{R}^d}{\varphi\,dx}=0$, then by Lemma \ref{lm2.3} we obtain that
$$
F^{\pm}(\lambda^2)\in S_{[\frac{d}{2}]+1}^0((0,\lambda_0)).
$$
Therefore, \eqref{equ6.2.3} follows from the lower bound \eqref{eq.11.1.2} and the expression \eqref{eq.11.1.5}.
Since
$$F^{\alpha, l}_{+}(\lambda^2)-F^{\alpha, l}_{-}(\lambda^2)=\frac{\alpha\chi(\lambda)(F_{-}(\lambda^2)-F_{+}(\lambda^2))}{(1+\alpha F_{+}(\lambda^2))(1+\alpha F_{-}(\lambda^2))},
$$
then \eqref{equ6.2.3.3} follows from \eqref{equ6.2.3}  and  \eqref{eq2.18} in Lemma \ref{lm2.4}.

We are left to prove the special case that $d=1, 2$ and $\int_{\mathbb{R}^d}{\varphi\,dx}\ne 0$. By \eqref{eq.11.1.5},
\eqref{equ6.2.1} follows from \eqref{eq.11.1.4} and  \eqref{eq2.17.1} in Lemma \ref{lm2.3};
\eqref{equ6.2.1.1} follows from \eqref{eq.11.1.4}, \eqref{eq2.17.1} in Lemma \ref{lm2.3} and Lemma \ref{lm2.4};
finally, \eqref{eq.13.3.1} follows from \eqref{eq.11.1.3}  and \eqref{eq2.17}  in Lemma \ref{lm2.3}.
\end{proof}

When viewing $\alpha$ as a very small parameter in \eqref{equ6.2.3} and \eqref{equ6.2.3.3}, we have the following better properties concerning  $\alpha$:
\begin{lemma}\label{lm3.3}
 If $0<\alpha<1$ and we choose $\lambda_0=\frac12$ in the cut-off function \eqref{eq3.3.1}, then in the case $d\ge 3$ or $d=1, 2$, $\int_{\mathbb{R}^d}\varphi(x)\,dx=0$, one has
 \begin{align}\label{eq.11.1.2'}
  	M^{-2([\frac{d}{2}]+1)} F^{\alpha, l}_{\pm}(\lambda^2)&\in S_{[\frac{d}{2}]+1}^0\left((0, \frac12)\right) \tag{\ref{equ6.2.3}'}
\end{align}
and
\begin{align}\label{eq.11.1.3'}
 	M^{-2([\frac{d}{2}]+2)} (F^{\alpha, l}_{+}(\lambda^2)-F^{\alpha, l}_{-}(\lambda^2))&\in S_{[\frac{d}{2}]+1}^{d-2}\left((0, \frac12)\right)\tag{\ref{equ6.2.3.3}'}
\end{align}
in the sense of \eqref{eq.11.1.6} and \eqref{eq.11.1.7} with $C_k$ independent of $\varphi$.
\end{lemma}
\begin{proof}
The proof of \eqref{eq.11.1.2'} and \eqref{eq.11.1.3'} are the same as \eqref{equ6.2.3} and \eqref{equ6.2.3.3}
except that: (\romannumeral1) when $0<\lambda<\frac12$ we replace \eqref{eq.11.1.2} by the spectral assumption \eqref{eq1.4}; (\romannumeral2) in \eqref{eq.11.1.5}, we chose $m=[\frac{d}{2}]+1$  and use Lemma \ref{lm2.3} and \ref{lm2.4} to obtain explicit upper bound on $M$. Note that  the  constants $C_0, C_1,\ldots$ are independent of $\varphi$. Therefore we conclude that
\begin{equation}\label{eq.11.1.6}
|\frac{d^k}{d\lambda^k}F^{\alpha, l}_{\pm}(\lambda^2)|\leq C_k\cdot M^{2([\frac{d}{2}]+1)}\cdot\lambda^{-k},\,\,\,\,\,\qquad\qquad\qquad\qquad~~~~~k=0, 1,\ldots, [\frac{d}{2}]+1,
\end{equation}
and
\begin{equation}\label{eq.11.1.7}
|\frac{d^k}{d\lambda^k}(F^{\alpha, l}_{+}(\lambda^2)-F^{\alpha, l}_{-}(\lambda^2))|\leq C_k\cdot M^{2([\frac{d}{2}]+2)}\cdot\lambda^{d-2-k},\,\,\quad k=0, 1,\ldots, [\frac{d}{2}]+1,
\end{equation}
where each $C_k$ is independent of $\varphi$.
\end{proof}
\begin{remark}\label{rmk3.3}
Actually, the proof above shows that \eqref{eq.11.1.2'} and \eqref{eq.11.1.3'} also hold when $d=1, 2$, $\int_{\mathbb{R}^d}\varphi(x)\,dx\ne 0$. We point out that, in this case, \eqref{equ6.2.1}--\eqref{eq.13.3.1} are the key properties to the dispersive estimate. However, \eqref{eq.11.1.2'} and \eqref{eq.11.1.3'} are weaker with respect to $\lambda$, and we cannot use them to obtain the decay $t^{-\frac{d}{2}}$. This is the reason that in \eqref{eq.11.11} (Theorem \ref{thm3.2}) and \eqref{eq1.5.0} (Theorem \ref{thm-1.1}), we exclude this situation.

\end{remark}
Now we are ready to prove Theorem \ref{thm3.2}. We remark that Lemma \ref{lm3.2} plays a key role when proving that the constant $C(\varphi)$ is uniformly with $\alpha$, and when $0<\alpha<1$, we use Lemma \ref{lm3.3} to obtain the better upper bound \eqref{eq.11.11}.

\emph{Proof of Theorem \ref{thm3.2}.}
First we prove \eqref{eq.l1}. Write
\begin{align}\label{eq.l7}
U_d^{+, l}-U_d^{-, l}&=\frac{-\alpha}{\pi i}\int_0^{\infty}e^{-it\lambda^2}\left(F^{\alpha, l}_{+}(\lambda^2)-F^{\alpha, l}_{-}(\lambda^2)\right)R^{+}_{0}(\lambda^2)\varphi\langle R^{+}_{0}(\lambda^2)\, \cdot, \varphi\rangle\lambda\,d\lambda \nonumber\\
&-\frac{\alpha}{\pi i}\int_0^{\infty}e^{-it\lambda^2}F^{\alpha, l}_{-}(\lambda^2)\left(R^{+}_{0}(\lambda^2)-R^{-}_{0}(\lambda^2)\right)\varphi\langle R^{+}_{0}(\lambda^2)\, \cdot, \varphi\rangle\lambda\,d\lambda\nonumber\\
&-\frac{\alpha}{\pi i}\int_0^{\infty}e^{-it\lambda^2}F^{\alpha, l}_{-}(\lambda^2)R^{-}_{0}(\lambda^2)\varphi\left\langle \left(R^{+}_{0}(\lambda^2)-R^{-}_{0}(\lambda^2)\right)\, \cdot, \varphi\right\rangle\lambda\,d\lambda\nonumber\\
&:=U_{d, 1}^{ l}+U_{d, 2}^{ l}+U_{d, 3}^{ l},
\end{align}
where we have used the algebraic identity:
\begin{equation}\label{eq.17.10}
  \prod_{k=1}^3{A_k^+}-\prod_{k=1}^3{A_k^-}=(A_1^+-A_1^-)A_2^+A_3^++A_1^-(A_2^+-A_2^-)A_3^++A_1^-A_2^-(A_3^+-A_3^-).
\end{equation}
%

Now we divide the proof into three cases.

$\bullet$ {\emph{case 1: $d=1$.}}
In this case, instead of using \eqref{eq.l7}, we estimate $U_1^{\pm, l}(t, x, y)$ directly. Since in $d=1$, $R^{+}_{0}(\lambda^2)-R^{-}_{0}(\lambda^2)$ and $R^{\pm}_{0}(\lambda^2)$ have the same singularity at $\lambda=0$ by \eqref{eq2.1}, then the estimates for $U_1^{+, l}-U_1^{-, l}$ and $U_1^{\pm, l}$ are the same. To proceed, we write
\begin{equation}\label{eq.l2}
  U_1^{\pm, l}(t, x, y)=\frac{1}{4\pi i}\int_{\mathbb{R}}\int_{\mathbb{R}}{I_1^{\pm, l}(t, |x-x_1|, |x_2-y|)\varphi(x_1)\varphi(x_2)\,dx_1dx_2},
\end{equation}
where
\begin{equation*}\label{eq.l3}
  I_1^{\pm, l}(t, |x-x_1|, |x_2-y|):=\int_0^{\infty}{e^{-it\lambda^2\pm i\lambda(|x-x_1|+ |x_2-y|)}\frac{\alpha}{\lambda}\cdot F^{\alpha, l}_{\pm}(\lambda^2)d\lambda}.
\end{equation*}
Observe that by Lemma \ref{lm3.2}, the behavior of $F^{\alpha, l}_{\pm}(\lambda^2)$ is closely related to whether or not $\int_{\mathbb{R}}{\varphi(x)\,dx}=0$, therefore we furhter break it into two subcases:

(\romannumeral1) When $\int_{\mathbb{R}}{\varphi(x)\,dx}\ne 0$, then it follows directly from  \eqref{eq.13.3.1} that
\begin{equation}\label{eq.13.3}
 \frac{\alpha}{\lambda}\cdot F^{\alpha, l}_{\pm}(\lambda^2)\in S^{0}_{1}\left((0, \lambda_0)\right).
\end{equation}
To estimate $I_1^{\pm, l}(t, |x-x_1|, |x_2-y|)$, we  apply the oscillatory integral estimate \eqref{eq2.22} in (\romannumeral1) of Corollary \ref{cor2.6} with  $\psi= \frac{\alpha}{\lambda}\cdot F^{\alpha, l}_{\pm}(\lambda^2)$. This, together with \eqref{eq.l2}, yields that
\begin{equation}\label{eq.l6}
\sup_{x, y\in\mathbb{R}}|U_1^{\pm, l}(t, x, y)|\leq C(\varphi)\cdot t^{-\frac12}.
\end{equation}

(\romannumeral2)  When $\int_{\mathbb{R}}{\varphi(x)\,dx}=0$, note that \eqref{eq.13.3.1} is no longer valid, instead we only have the weaker property \eqref{equ6.2.3}  (there is a loss of $1/\lambda$). We shall proceed differently and exploit the cancellation property of $\varphi$. In fact, observe that if  $\int_{\mathbb{R}}{\varphi(x_1)\,dx_1}=0$, then by \eqref{eq2.1},
\begin{align*}
  R_0^{\pm}(\lambda^2)\varphi&=\frac{\pm ie^{\pm i\lambda|x|}}{2\lambda}\int_{\mathbb{R}}{[e^{\pm i\lambda(|x-x_1|-|x|)}-1]\varphi(x_1)\,dx_1}\\
  &=\frac{-e^{\pm i\lambda|x|}}{2}\cdot\int_{\mathbb{R}}{(|x-x_1|-|x|) \cdot r_1(\lambda)\varphi(x_1)\,dx_1},
\end{align*}
where $r_1(\lambda):=\frac{1}{\lambda}\int_0^{\lambda}e^{\pm is\cdot(|x-x_1|-|x|)}\,ds$. Similarly,
 \begin{align*}
 \int_{}\varphi(x_2) R_0^{\pm}(\lambda^2)(x_2, y)\,dx_2&=\frac{\pm ie^{\pm i\lambda|y|}}{2\lambda}\int_{\mathbb{R}}{[e^{\pm i\lambda(|x_2-y|-|y|)}-1]\varphi(x_2)\,dx_2}\\
  &=\frac{-e^{\pm i\lambda|y|}}{2}\cdot \int_{\mathbb{R}}{(|x_2-y|-|y|) \cdot r'_1(\lambda)\varphi(x_2)\,dx_2},
\end{align*}
where $r_1'(\lambda):=\frac{1}{\lambda}\int_0^{\lambda}e^{\pm is\cdot(|x_2-y|-|y|)}\,ds$. By \eqref{eq3.5} we rewrite the kernel of $U_1^{\pm, l}$ as:
\begin{equation}\label{eq.l6.11}
  U_1^{\pm, l}(t, x, y)=\int_{\mathbb{R}}\int_{\mathbb{R}}{\tilde{I}_1^{\pm, l}(t, |x|, |y|)(|x-x_1|-|x|)(|x_2-y|-|y|)\varphi(x_1)\varphi(x_2)\,dx_1dx_2},
\end{equation}
where
\begin{equation*}\label{eq.16.12}
  \tilde{I}_1^{\pm, l}(t, |x|, |y|):=\frac{-1}{4\pi i}\int_0^{\infty}{e^{-it\lambda^2\pm i\lambda(|x|+ |y|)}\alpha F^{\alpha, l}_{\pm}(\lambda^2)\cdot r_1(\lambda)\cdot r_1'(\lambda)\cdot \lambda d\lambda}.
\end{equation*}
A direct computation shows that $r_1(\lambda),\, r'_1(\lambda)\in S^{0}_{1}\left((0, \lambda_0)\right)$. Combining this with  \eqref{equ6.2.3} implies that
$$\alpha F^{\alpha, l}_{\pm}(\lambda^2)\cdot r_1(\lambda)\cdot r_1'(\lambda)\cdot \lambda\in S^{1}_{1}\left((0, \lambda_0)\right)\subset S^{0}_{1}\left((0, \lambda_0)\right). $$
Again  by the oscillatory integral estimate \eqref{eq2.22} with  $b=0, K=1$, one has
$$
\sup_{x, y\in\mathbb{R}}|\tilde{I}_1^{\pm, h}(t, |x|, |y|)|\leq C(\varphi)\cdot t^{-\frac12}.
$$
This, together with \eqref{eq.l6.11}, yields that
\begin{align}\label{eq.l6.13}
 \sup_{x, y\in\mathbb{R}}|U_1^{\pm, l}(t, x, y)|&\leq C(\varphi)\cdot t^{-\frac12}\int_{\mathbb{R}}\int_{\mathbb{R}}{(|x-x_1|-|x|)(|x_2-y|-|y|)\varphi(x_1)\varphi(x_2)\,dx_1dx_2}\nonumber\\
 &\leq C(\varphi)\cdot t^{-\frac12},
\end{align}
where we have used triangle inequality $||x-x_1|-|x||\leq |x_1|$,  $||x_2-y|-|y||\leq |x_2|$ and the decay assumption \eqref{eq1.2}.  Combining \eqref{eq.l6} and \eqref{eq.l6.13}, we prove \eqref{eq.l1} for $d=1$.

$\bullet$ {\emph{case 2: $d\ge 3$, odd.}}

First, we deal with the term $U_{d, 1}^{ l}$. From the explicit expression of the free resolvent \eqref{eq2.3}, its integral kernel can be written as a linear combination of
\begin{align}\label{eq.l8}
\frac{\alpha}{\pi i}\int_{\mathbb{R}^d}\int_{\mathbb{R}^d}\int_0^{\infty}{e^{-it\lambda^2+ i\lambda(|x-x_1|+ |x_2-y|)}(F^{\alpha, l}_{+}(\lambda^2)-F^{\alpha, l}_{-}(\lambda^2))\lambda^{k_1+k_2+1}} \nonumber\\
\cdot\frac{\varphi(x_1)}{|x-x_1|^{d-2-k_1}}\frac{\varphi(x_2)}{|x_2-y|^{d-2-k_2}}d\lambda\,dx_1dx_2,
\end{align}
where $0\leq k_1, k_2\leq \frac{d-3}{2}$. We rewrite \eqref{eq.l8} as the form
\begin{equation}\label{eq.19.1}
\int_{\mathbb{R}^d}\int_{\mathbb{R}^d}{I_d^{l}(t, |x-x_1|, |x_2-y|)\frac{\varphi(x_1)}{|x-x_1|^{d-2-k_1}}\frac{\varphi(x_2)}{|x_2-y|^{d-2-k_2}}\,dx_1dx_2},\tag{\ref{eq.l8}'}
\end{equation}
where
\begin{align*}
I_d^{l}(t, |x-x_1|, |x_2-y|)&:=\frac{\alpha}{\pi i}\int_0^{\infty}{e^{-it\lambda^2+ i\lambda(|x-x_1|+ |x_2-y|)}(F^{\alpha, l}_{+}(\lambda^2)-F^{\alpha, l}_{-}(\lambda^2))\lambda^{k_1+k_2+1} d\lambda}.
\end{align*}
Then it follows from  \eqref{equ6.2.3.3} in Lemma \ref{lm3.2}  that for any $0\le k_1, k_2\le \frac{d-3}{2}$,
\begin{equation*}\label{eq.l9}
\alpha\cdot\left(F^{\alpha, l}_{+}(\lambda^2)-F^{\alpha, l}_{-}(\lambda^2)\right)\cdot\lambda^{k_1+k_2+1}\in  S^{d-1}_{\frac{d+1}{2}}\left((0, \lambda_0)\right).
\end{equation*}
Now we use oscillatory integral estimates \eqref{eq2.22} (with $b=d-1$, $K=\frac{d+1}{2}$ ) in Corollary \ref{cor2.6} to yield that
\begin{equation*}\label{eq.20}
|I_d^{l}(t, |x-x_1|, |x_2-y|)|\leq C(\varphi)(1+|x-x_1|+ |x_2-y|)^{\frac{d-1}{2}}\cdot t^{-\frac{d}{2}}.
\end{equation*}
Therefore by \eqref{eq.l8}, \eqref{eq.19.1} we have
\begin{align*}\label{eq.21}
&\sup_{x, y\in\mathbb{R}^d}|U_{d, 1}^{l}(t, x, y)|\nonumber\\
&\leq t^{-\frac{d}{2}}\sum_{k_1, k_2=0}^{\frac{d-3}{2}}C_{k_1, k_2, \varphi}\int_{\mathbb{R}^{2d}}{(1+|x-x_1|+ |x_2-y|)^{\frac{d-1}{2}}\frac{|\varphi(x_1)|}{|x-x_1|^{d-2-k_1}}\frac{|\varphi(x_2)|}{|x_2-y|^{d-2-k_2}}\,dx_1dx_2}\nonumber\\
&\leq C(\varphi)\cdot t^{-\frac{d}{2}},
\end{align*}
where the last inequality follows by \eqref{eq2.20} in Lemma \ref{lm2.8} and the decay assumption \eqref{eq1.2}.

Next, we handle the term $U_{d, 2}^{l}$. By \eqref{eq2.3} again, the kernel of $U_{d, 2}^{l}$ can be written as
\begin{equation}\label{eq.22}
U_{d, 2}^{l}(t, x, y)=\sum_{k=0}^{\frac{d-3}{2}}{C_{d,\,k}\cdot U_{d, 2,k}^{l}(t, x, y)},
\end{equation}
where
\begin{align}\label{eq.23}
U_{d, 2,k}^{l}(t, x, y)&=\frac{\alpha}{\pi i}\int_{\mathbb{R}^{2d}}\int_0^{\infty}e^{-it\lambda^2+ i\lambda\cdot|x_2-y|}F^{\alpha,l}_{-}(\lambda^2)\lambda^{k+1}\nonumber\\
&\cdot \left(R^{+}_{0}(\lambda^2,x,x_1)-R^{-}_{0}(\lambda^2,x,x_1)\right)\frac{\varphi(x_1)\varphi(x_2)}{|x_2-y|^{d-2-k}}d\lambda\,dx_1dx_2.
\end{align}
Now the key point is to decompose the difference $R^{+}_{0}(\lambda^2)-R^{-}_{0}(\lambda^2)$ in a suitable form. Using \eqref{eq2.3} and applying Taylor expansion for $e^{\pm i\lambda |x-x_1|}$ we can write (similar to \eqref{eq2.8})
\begin{equation*}\label{eq.24}
R^{\pm}_0(\lambda^2, x, x_1)=C_d\sum^{d-3}_{l=0}{d_l(\pm i\lambda)^l|x-x_1|^{l+2-d}}+\sum^{\frac{d-3}{2}}_{j=0}{c^{\pm}_j\lambda^j\int_0^{\lambda}e^{\pm is |x-x_1|}(\lambda-s)^{d-3-j}\,ds},
\end{equation*}
where $d_l$ satisfies (see \cite[Lemma 3.3]{J80}):
$$
d_l=0,\,\,\,\text{for}\,\,l=1, 3,\ldots, d-4.
$$
This implies that
\begin{equation*}
 R^{+}_{0}(\lambda^2, x, x_1)-R^{-}_{0}(\lambda^2, x, x_1)=\sum^{\frac{d-3}{2}}_{j=0}{\lambda^j\int_0^{\lambda}(c^{+}_j e^{+is |x-x_1|}-c^{-}_j e^{-is |x-x_1|})(\lambda-s)^{d-3-j}\,ds}.
\end{equation*}
Using this decomposition we rewrite $U_{d, 2,k}^{l}(t, x, y)$ in \eqref{eq.23} as
\begin{align}\label{eq.24.1}
U_{d, 2,k}^{l}(t, x, y)&=\sum^{\frac{d-3}{2}}_{j=0}\int_{\mathbb{R}^{2d}}(I_{d, k, j}^{l,+}(t, x, y)-I_{d, k, j}^{l,-}(t, x, y)) \cdot\frac{\varphi(x_1)\varphi(x_2)}{|x_2-y|^{d-2-k}}d\lambda\,dx_1dx_2\nonumber\\
&:=\sum^{\frac{d-3}{2}}_{j=0}U_{d, 2,k,j}^{l}(t, x, y),\,\,\tag{\ref{eq.23}'}
\end{align}
where
\begin{equation*}\label{eq.25}
 I_{d, k, j}^{l,\pm}(t, |x-x_1|, |x_2-y|):=\frac{\alpha}{\pi i}\int_0^{\infty}{e^{-it\lambda^2+ i\lambda|x_2-y|}F^{\alpha, l}_{-}(\lambda^2)g^{\pm}_{d, k, j}(\lambda, |x-x_1|)\,d\lambda},
\end{equation*}
and
\begin{equation}\label{eq.26}
g^{\pm}_{d, k, j}(\lambda, |x-x_1|)=c^{\pm}_j \lambda^{j+k+1}\cdot\int_0^{\lambda}e^{\pm is |x-x_1|}(\lambda-s)^{d-3-j}\,ds.
\end{equation}

We observe that for each $0\leq k\leq \frac{d-3}{2}$, $0\leq j\leq \frac{d-5}{2}$, there exists an absolute constant $C_{jk}<\infty$ such that for $0<\lambda<1$,
\begin{equation}\label{eq.26.1}
\left|\frac{d^s}{d\lambda^s}g^{\pm}_{d, k, j}(\lambda, |x-x_1|)\right|\leq C_{jk}\lambda^{d-1-s},\,\,\,0\leq s\leq \frac{d+1}{2}.
\end{equation}
This, together with \eqref{equ6.2.3}  in Lemma \ref{lm3.2}, indicates that
$$
\alpha F^{\alpha, l}_{-}(\lambda^2)g^{\pm}_{d, k, j}(\lambda, |x-x_1|)\in S^{d-1}_{\frac{d+1}{2}}\left((0, \lambda_0)\right).
$$
Applying \eqref{eq2.22} (with $b=d-1$, $K=\frac{d+1}{2}$ ) in Corollary \ref{cor2.6} shows that
\begin{equation*}\label{eq.27}
  |I_{d, k, j}^{l,\pm}(t, |x-x_1|, |x_2-y|)|\leq C(\varphi)t^{-\frac{d}{2}}(1+|x_2-y|)^{\frac{d-1}{2}}.
\end{equation*}
This, together with \eqref{eq.23}' and Lemma \ref{lm2.8}, implies that
\begin{align}\label{eq.28}
\sum^{\frac{d-5}{2}}_{j=0}\left|U_{d, 2,k,j}^{l}(t, x, y)\right|&\leq C(\varphi)t^{-\frac{d}{2}} \int_{\mathbb{R}^{2d}}(1+|x_2-y|)^{\frac{d-1}{2}}\frac{|\varphi(x_1)||\varphi(x_2)|}{|x_2-y|^{d-2-k}}d\lambda\,dx_1dx_2 \nonumber \\
&\leq C(\varphi)\cdot t^{-\frac{d}{2}}.
\end{align}

We are left to deal with the case $j=\frac{d-3}{2}$ in $U_{d, 2,k,j}$. In this case, the trouble is that \eqref{eq.26.1} doesn't hold for $s=\frac{d+1}{2}$, and this is due to
\begin{equation*}
  \frac{d^{\frac{d+1}{2}}}{d\lambda^{\frac{d+1}{2}}}\int_0^{\lambda}e^{\pm is |x-x_1|}(\lambda-s)^{\frac{d-3}{2}}\,ds=C_d|x-x_1|\cdot e^{\pm i\lambda |x-x_1|}.
\end{equation*}
There is a growth $|x-x_1|$ in the estimate above, and we cannot afford it in \eqref{eq.24.1} by using Lemma \ref{lm2.6} directly. Thus we proceed differently. Note that
$$\frac{-1}{2i\lambda t}\cdot\frac{d}{d\lambda}e^{-it\lambda^2}=e^{-it\lambda^2},$$
then we perform integration by parts $\frac{d-1}{2}$ times with this identity to derive the following
\begin{align}\label{eq.29}
 &I_{d, k, \frac{d-3}{2}}^{l,\pm}(t, |x-x_1|, |x_2-y|)= \sum_{s_1, s_2, s_3,s_4}C^{\pm}_{s_1 s_2 s_3 s_4}\alpha \cdot t^{-\frac{d-1}{2}}|x_2-y|^{s_1} \nonumber\\
 &\cdot \int_{0}^{\infty}{e^{-it\lambda^2+ i\lambda|x_2-y|}\frac{d^{s_2}}{d\lambda^{s_2}}F^{\alpha, h}_{-}(\lambda^2)\frac{d^{s_3}}{d\lambda^{s_3}}\int_0^{\lambda}e^{\pm is |x-x_1|}(\lambda-s)^{\frac{d-3}{2}}\,ds \cdot \lambda^{k-s_4} }d\lambda,
\end{align}
where $s_1, s_2, s_3, s_4$ are nonnegative integers satisfying
\begin{equation}\label{eq.29.6}
s_1+s_2+s_3+ s_4=\frac{d-1}{2}.
\end{equation}
(\romannumeral1)  If $s_3<\frac{d-1}{2}$, then note that   and a direct computation shows that
\begin{equation*}\label{eq.29.1}
\frac{d^{s_3}}{d\lambda^{s_3}}\int_0^{\lambda}e^{\pm is |x-x_1|}(\lambda-s)^{\frac{d-3}{2}}\,ds \cdot \lambda^{k-s_4} \in S_1^{\frac{d-1}{2}+k-s_3-s_4}\left((0,\lambda_0)\right).
\end{equation*}
Further, note that by \eqref{eq.29.6}, $\frac{d-1}{2}+k-s_2-s_3-s_4\ge 0$. Thus  it follows from \eqref{equ6.2.3}  in Lemma \ref{lm3.2} that
\begin{equation*}
\alpha\cdot \frac{d^{s_2}}{d\lambda^{s_2}}F^{\alpha, l}_{-}(\lambda^2)\cdot \frac{d^{s_3}}{d\lambda^{s_3}} \int_0^{\lambda}e^{\pm is |x-x_1|}(\lambda-s)^{\frac{d-3}{2}}\,ds \cdot \lambda^{k-s_4}\in  S_1^0\left((0,\lambda_0)\right).
\end{equation*}
Then we apply oscillatory integral estimates \eqref{eq2.22} (with $b=0$, $K=1$) in Corollary \ref{cor2.6} to obtain that
\begin{align*}\label{eq.29.2}
\alpha\cdot\int_{0}^{\infty}{e^{-it\lambda^2+ i\lambda|x_2-y|}\frac{d^{s_2}}{d\lambda^{s_2}}F^{\alpha, l}_{-}(\lambda^2) \cdot \frac{d^{s_3}}{d\lambda^{s_3}} \int_0^{\lambda}e^{\pm is |x-x_1|}(\lambda-s)^{\frac{d-3}{2}}\,ds \cdot \lambda^{k-s_4}d\lambda}\leq C(\varphi)\cdot t^{-\frac{1}{2}},
\end{align*}
where the constant $C(\varphi)$ is independent of $\alpha, x, x_1, x_2, y$.

(\romannumeral2) If $s_3=\frac{d-1}{2}$ (which implies $s_1=s_2=s_4=0$), observe that in this case we have
$$
\frac{d^{\frac{d-1}{2}}}{d\lambda^{\frac{d-1}{2}}}\int_0^{\lambda}e^{\pm is |x-x_1|}(\lambda-s)^{\frac{d-3}{2}}\,ds=Ce^{\pm i\lambda |x-x_1|},
$$
then the oscillatory integral in \eqref{eq.29} becomes
\begin{equation*}
\alpha\cdot \int_{0}^{\infty}{e^{-it\lambda^2+ i\lambda(|x_2-y|\pm |x-x_1|)}F^{\alpha, h}_{-}(\lambda^2)\lambda^{k}\,d\lambda}.
\end{equation*}
Now apply \eqref{equ6.2.3}   and \eqref{eq2.22} (with $b=0$, $K=1$) in Corollary \ref{cor2.6}, we derive that
\begin{equation*}\label{eq.29.3}
\alpha\cdot \int_{0}^{\infty}{e^{-it\lambda^2+ i\lambda(|x_2-y|\pm |x-x_1|)}F^{\alpha, h}_{-}(\lambda^2)\lambda^{k}\,d\lambda}
\leq C(\varphi)\cdot t^{-\frac{1}{2}}.
\end{equation*}
Combining (\romannumeral1)  and (\romannumeral2), and plugging them into \eqref{eq.29}, then we obtain from \eqref{eq.24.1} and Lemma \ref{lm2.8} that for any $0\le k\le \frac{d-3}{2}$,
\begin{align}\label{eq.29.4}
\left|U_{d, 2,k, \frac{d-3}{2}}^{l}(t, x, y)\right|&\leq  C(\varphi)\cdot t^{-\frac{d}{2}} \int_{\mathbb{R}^{2d}}(1+|x_2-y|)^{\frac{d-1}{2}}\frac{|\varphi(x_1)||\varphi(x_2)|}{|x_2-y|^{d-2-k}}d\lambda\,dx_1dx_2 \nonumber \\
&\leq  C(\varphi)\cdot t^{-\frac{d}{2}}.
\end{align}
Therefore it follows from \eqref{eq.24.1}, \eqref{eq.28} and \eqref{eq.29.4} that
\begin{equation}\label{eq.29.5}
\sup_{x, y\in\mathbb{R}^d}|U_{d, 2,k}^{l}(t, x, y)|\leq  C(\varphi)\cdot t^{-\frac{d}{2}}.
\end{equation}
Finally in view of \eqref{eq.22} and \eqref{eq.29.5}, we prove the desired estimates for  $U_{d, 2}^{l}$. Since the proof for  $U_{d, 3}^{l}$ is the same as  $U_{d, 2}^{l}$, we complete the proof of \eqref{eq.l1} in  higher odd dimensions.

$\bullet$ {\emph{case 3: $d\ge 2$}, even.}

In even dimensions, the expressions of $U_{d, 1}^{l}, U_{d, 2}^{l}, U_{d, 3}^{l}$ are more complicated than that of odd dimensions. Indeed, based on the expression of the free resolvent \eqref{eq3.50}, we first write the integral kernel of $U_{d, 1}^{l}$ as
$$
U_{d, 1}^{l}(t, x, y)=\sum_{j=1}^4{U_{d, 1, j}^{l}(t, x, y)},
$$
where
\begin{align}\label{eq3.54.5}
U_{d, 1, 1}^{l}(t, x, y)&:=\frac{-\alpha}{\pi i}\int_{\mathbb{R}^{2d}}\int_0^{\infty}e^{-it\lambda^2+ i\lambda(|x-x_1|+ |x_2-y|)}(F^{\alpha, l}_{+}(\lambda^2)-F^{\alpha, l}_{-}(\lambda^2))\cdot\lambda^{d-1}\nonumber \\
&\cdot \frac{w_{+, >}(\lambda|x-x_1|)}{|x-x_1|^{\frac{d}{2}-1}}\frac{w_{+, >}(\lambda|x_2-y|)}{|x_2-y|^{\frac{d}{2}-1}} \varphi(x_1)\varphi(x_2)\,d\lambda dx_1dx_2, \tag{$U^l_1$}
\end{align}
\begin{align}\label{eq3.55.5}
U_{d, 1, 2}^{l}(t, x, y)&:=\frac{-\alpha}{\pi i}\int_{\mathbb{R}^{2d}}\int_0^{\infty}e^{-it\lambda^2+ i\lambda|x-x_1|}(F^{\alpha, l}_{+}(\lambda^2)-F^{\alpha, l}_{-}(\lambda^2))\cdot\lambda^{\frac{d}{2}} \nonumber \\
&\cdot \frac{w_{+, >}(\lambda|x-x_1|)}{|x-x_1|^{\frac{d}{2}-1}}\frac{w_{+, <}(\lambda|x_2-y|)}{|x_2-y|^{d-2}\cdot}\varphi(x_1)\varphi(x_2)\,d\lambda dx_1dx_2, \tag{$U^l_2$}
\end{align}
\begin{align}\label{eq3.56.5}
U_{d, 1, 3}^{l}(t, x, y)&:=\frac{-\alpha}{\pi i}\int_{\mathbb{R}^{2d}}\int_0^{\infty}e^{-it\lambda^2+ i\lambda|x_2-y|}(F^{\alpha, l}_{+}(\lambda^2)-F^{\alpha, l}_{-}(\lambda^2))\cdot\lambda^{\frac{d}{2}} \nonumber \\
&\cdot\frac{w_{+, <}(\lambda|x-x_1|)}{|x-x_1|^{d-2}}\frac{w_{+, >}(\lambda|x_2-y|)}{|x_2-y|^{\frac{d}{2}-1}}\varphi(x_1)\varphi(x_2)\,d\lambda dx_1dx_2,\tag{$U^l_3$}
\end{align}
and
\begin{align}\label{eq3.57.2}
U_{d, 1, 4}^{l}(t, x, y):=&\frac{-\alpha}{\pi i}\int_{\mathbb{R}^{2d}}\int_0^{\infty}e^{-it\lambda^2}(F^{\alpha, l}_{+}(\lambda^2)-F^{\alpha, l}_{-}(\lambda^2))\cdot\lambda\nonumber \\
&\cdot \frac{w_{+, <}(\lambda|x-x_1|)}{|x-x_1|^{d-2}}\frac{w_{+, <}(\lambda|x_2-y|)}{|x_2-y|^{d-2}}\varphi(x_1)\varphi(x_2)\,d\lambda dx_1dx_2.\tag{$U^l_4$}
\end{align}

Next, in order to decompose $U_{d, 2}^{l}$ into a similar form, we use \eqref{equ6.2.8} to write
$$
U_{d, 2}^{l}(t, x, y)=\sum_{j=1}^4\sum_{\pm}{U^{\pm, l}_{d, 2, j}(t, x, y)},
$$
where
\begin{align}\label{eq3.54.1}
U_{d, 2, 1}^{\pm, l}(t, x, y)&:=\frac{-\alpha}{\pi i}\int_{\mathbb{R}^{2d}}\int_0^{\infty}e^{-it\lambda^2+ i\lambda(\pm|x-x_1|+ |x_2-y|)}F^{\alpha, l}_{-}(\lambda^2)\cdot\lambda^{d-1}\nonumber \\
&\cdot \frac{J_{\pm, >}(\lambda|x-x_1|)}{|x-x_1|^{\frac{d}{2}-1}}\frac{w_{+, >}(\lambda|x_2-y|)}{|x_2-y|^{\frac{d}{2}-1}} \varphi(x_1)\varphi(x_2)\,d\lambda dx_1dx_2,\tag{$U^l_5$}
\end{align}
\begin{align}\label{eq3.55.1}
U_{d, 2, 2}^{\pm, l}(t, x, y)&:=\frac{-\alpha}{\pi i}\int_{\mathbb{R}^{2d}}\int_0^{\infty}e^{-it\lambda^2\pm i\lambda|x-x_1|}F^{\alpha, l}_{-}(\lambda^2)\cdot\lambda^{\frac{d}{2}} \nonumber \\
&\cdot \frac{J_{\pm, >}(\lambda|x-x_1|)}{|x-x_1|^{\frac{d}{2}-1}}\frac{w_{+, <}(\lambda|x_2-y|)}{|x_2-y|^{d-2}\cdot}\varphi(x_1)\varphi(x_2)\,d\lambda dx_1dx_2,\tag{$U^l_6$}
\end{align}
\begin{align}\label{eq3.56.1}
U_{d, 2, 3}^{\pm, l}(t, x, y)&:=\frac{-\alpha}{\pi i}\int_{\mathbb{R}^{2d}}\int_0^{\infty}e^{-it\lambda^2+ i\lambda(\pm|x-x_1|+ |x_2-y|)}F^{\alpha, l}_{-}(\lambda^2)\cdot\lambda^{d-1}\nonumber \\
&\cdot\frac{J_{\pm, <}(\lambda|x-x_1|)}{|x-x_1|^{\frac{d}{2}-1}}\frac{w_{+, >}(\lambda|x_2-y|)}{|x_2-y|^{\frac{d}{2}-1}}\varphi(x_1)\varphi(x_2)\, d\lambda dx_1dx_2,\tag{$U^l_7$}
\end{align}
and
\begin{align}\label{eq3.57.1}
U_{d, 2, 4}^{\pm, l}(t, x, y)&:=\frac{\alpha}{\pi i}\int_{\mathbb{R}^{2d}}\int_0^{\infty}e^{-it\lambda^2\pm i\lambda|x-x_1|}F^{\alpha, l}_{-}(\lambda^2)\cdot\lambda^{\frac{d}{2}}\nonumber \\
&\cdot\frac{J_{\pm, <}(\lambda|x-x_1|)}{|x-x_1|^{\frac{d}{2}-1}}\frac{w_{+, <}(\lambda|x_2-y|)}{|x_2-y|^{d-2}}\varphi(x_1) \varphi(x_2)\,d\lambda dx_1dx_2,\tag{$U^l_8$}
\end{align}

$U_{d, 3}^{l}$ can be decomposed exactly the same way as $U_{d, 2}^{l}$ (just exchange the role between $|x_2-y|$ and $|x-x_1|$), thus we omit it.

Before proceeding, we need the following facts concerning  $w_{\pm, \gtrless}$ and $J_{\pm, \gtrless}$, which will be used later in estimating oscillatory integrals.

When $d\ge 2$ is even, it follows from \eqref{eq3.51} that
\begin{align}\label{equ6.2.5}
  w_{\pm, >}(\lambda|x-x_1|)&\in  S_{\frac{d}{2}+1}^0\left((0,\lambda_0)\right),\\
   |x-x_1|^{\frac12}w_{\pm, >}(\lambda|x-x_1|)&\in  S_{\frac{d}{2}+1}^{-\frac12}\left((0,\lambda_0)\right).\label{equ6.2.5.1}
\end{align}
And by \eqref{eq3.25.11} and \eqref{eq3.52} one has
\begin{align}\label{equ6.2.6}
(1+|x-x_1|^{-\frac12})^{-1}w_{\pm, <}(\lambda|x-x_1|)/\log{\lambda }&\in  S_{\frac{d}{2}+1}^0\left((0,\lambda_0)\right),\qquad \text{when}\,\,\,d=2,\\
w_{\pm, <}(\lambda|x-x_1|)&\in  S_{\frac{d}{2}+1}^{0}\left((0,\lambda_0)\right),\qquad \text{when}\,\,\,d>2, even,\label{equ6.2.7}
\end{align}
where in \eqref{equ6.2.6}, we have used the expansion \eqref{eq2.6}, the definition of $w_{\pm, <}$ (see \eqref{eq3.50} and below) as well as the following inequality
$$
|\log{\lambda|x-x_1|}|\leq C|\log{\lambda}|\cdot (1+|x-x_1|^{-\frac12}),\,\,\,\text{when}\,\,\, \lambda|x-x_1|\leq 1\,\,\,\text{and}\,\,\,0<\lambda<\lambda_0<1.
$$
Meanwhile, by \eqref{equ6.2.9} and the fact that $\text{supp}\,J_{\pm, >}(z)\subset[\frac12, \infty]$, $J_{\pm, >}(\cdot)$ satisfies
\begin{align}\label{equ6.2.5.2}
  |x-x_1|^{\frac12}J_{\pm, >}(\lambda|x-x_1|)&\in  S_{\frac{d}{2}+1}^{-\frac12}\left((0,\lambda_0)\right),\\
   |x-x_1|^{1-\frac d2}J_{\pm, >}(\lambda|x-x_1|)&\in  S_{\frac{d}{2}+1}^{\frac d2 -1}\left((0,\lambda_0)\right).\label{equ6.2.5.3}
\end{align}
By \eqref{equ6.2.10} and the fact that $\text{supp}\,J_{\pm, <}(z)\subset[0, 1]$,  we have
\begin{align}\label{equ6.2.7.1}
|x-x_1|^{\frac 12}J_{\pm, <}(\lambda|x-x_1|)&\in  S_{\frac{d}{2}+1}^{-\frac12}\left((0,\lambda_0)\right), \\
|x-x_1|^{1-\frac d2}J_{\pm, <}(\lambda|x-x_1|)&\in S_{\frac{d}{2}+1}^{\frac d2 -1}\left((0,\lambda_0)\right). \label{equ6.2.7.2}
\end{align}
We remark that, compared \eqref{equ6.2.7.2} with \eqref{equ6.2.6} and \eqref{equ6.2.7}, we find that $J_{\pm, <}(\cdot)$ belongs to a better symbol class  than $w_{\pm, <}(\cdot)$ for all $d\ge 2$. This plays a key role in the estimate of $U_{d, 2}^{l}$ and  $U_{d, 3}^{l}$.

We shall use properties \eqref{equ6.2.3}--\eqref{equ6.2.1.1}, as well as
 \eqref{equ6.2.5}--\eqref{equ6.2.7.2} to estimate integrals \eqref{eq3.54.5}--\eqref{eq3.57.1}.
We mention that in two dimension,  whether or not the integral $\int_{\mathbb{R}^2}{\varphi(x)\,dx}$ equals to zero may affect the behavior of $F^{\alpha, l}_{\pm}(\lambda^2)$ (see Lemma \ref{lm3.2}).
Therefore similar to $d=1$, we further divide the analysis into two subcases:

$\bullet$ {\emph{Subcase 1: $d=2$ and $\int_{\mathbb{R}^2}{\varphi(x)\,dx}\ne0$, or $d\ge 4$}, even.}

\emph{Step 1.} In this step we consider the term $U_{d, 1}^{ l}$. Since its kernel is the sum of
\eqref{eq3.54.5}, \eqref{eq3.55.5}, \eqref{eq3.56.5} and \eqref{eq3.57.2}, we present the proof in order.

 For the term \eqref{eq3.54.5}, the oscillatory integral we face is
  \begin{align}\label{eq3.58}
I_{d, 1, 1}^{l}(t, x, y)&:=\frac{\alpha}{\pi i}\int_0^{\infty}e^{-it\lambda^2+ i\lambda(|x-x_1|+ |x_2-y|)}(F^{\alpha, l}_{+}(\lambda^2)-F^{\alpha, l}_{-}(\lambda^2))\nonumber \\
&\cdot w_{+, >}(\lambda|x-x_1|)w_{+, >}(\lambda|x_2-y|)\lambda^{d-1}\,d\lambda.
\end{align}
We mention that in order to prove uniform bounds with respect to $x, y$, we shall use Lemma \ref{lm2.6} (which is finer than Corollary \ref{cor2.6}) to control $I_{d, 1, 1}^{l}$. This leads to the following discussion:
(\romannumeral1) In the region $|x-x_1|+ |x_2-y|\leq t^{\frac12}$, note that by \eqref{equ6.2.3.3}, \eqref{equ6.2.5}, we have
$$
\alpha(F^{\alpha, l}_{+}(\lambda^2)-F^{\alpha, l}_{-}(\lambda^2))w_{+, >}(\lambda|x-x_1|)w_{+, >}(\lambda|x_2-y|)\lambda^{d-1}\in S_{\frac{d}{2}+1}^{d-1}\left((0, \lambda_0)\right)
$$
(in fact, it belongs to the smaller class $S_{\frac{d}{2}+1}^{2d-3}\left((0, \lambda_0)\right)$ when $d\ge 4$, which is better than we need). Then apply Lemma \ref{lm2.6}  in the above region  with $b=d-1$ and $K=\frac{d}{2}+1$,  we have
$$
|I_{d, 1, 1}^{l}(t, x, y)|\leq C(\varphi)\cdot t^{-\frac{d}{2}}.
$$
Plugging this into \eqref{eq3.54.5}, it follows from Lemma \ref{lm2.8} and the decay assumption \eqref{eq1.2} that
\begin{align}\label{eq.59}
 \left|U_{d, 1, 1}^{l}(t, x, y)\right|&\leq C(\varphi)\cdot t^{-\frac{d}{2}} \int_{\mathbb{R}^{2d}}\frac{|\varphi(x_1)|}{|x-x_1|^{\frac{d}{2}-1}}\cdot\frac{|\varphi(x_2)|}{|x_2-y|^{\frac{d}{2}-1}} \,dx_1dx_2\nonumber \\
&\leq C(\varphi)\cdot t^{-\frac{d}{2}}.
\end{align}
(\romannumeral2)  In the region $t^{\frac12}<|x-x_1|+ |x_2-y|$, we notice that by \eqref{equ6.2.3.3} and \eqref{equ6.2.5.1}, we have
 $$
\alpha(F^{\alpha, l}_{+}(\lambda^2)-F^{\alpha, l}_{-}(\lambda^2))(|x-x_1|+ |x_2-y|)^{\frac12} w_{+, >}(\lambda|x-x_1|)w_{+, >}(\lambda|x_2-y|)\lambda^{d-1}\in S_{\frac{d}{2}+1}^{\frac{d-1}{2}}\left((0,\lambda_0)\right)
$$
(in fact, it belongs to the smaller class $S_{\frac{d}{2}+1}^{2d-7/2}\left((0, \lambda_0)\right)$ when $d\ge 4$). Now we apply Lemma \ref{lm2.6}  in the above region  with $b=\frac{d-1}{2}$ and $K=\frac{d}{2}+1$, then we have
$$
|I_{d, 1, 1}^{l}(t, x, y)|\leq C(\varphi)\cdot t^{-\frac{d}{2}}(|x-x_1|+ |x_2-y|)^{\frac{d}{2}-1}.
$$
Then it follows again  from Lemma \ref{lm2.8} and the decay assumption \eqref{eq1.2} that
\begin{align}\label{eq.60}
\left|U_{d, 1, 1}^{l}(t, x, y)\right|&\leq C(\varphi)\cdot t^{-\frac{d}{2}} \int_{\mathbb{R}^{2d}}\frac{(|x-x_1|+ |x_2-y|)^{\frac{d}{2}-1}}{|x-x_1|^{\frac{d}{2}-1}|x_2-y|^{\frac{d}{2}-1}}|\varphi(x_1)||\varphi(x_2)| \,dx_1dx_2 \nonumber \\
&\leq C(\varphi)\cdot t^{-\frac{d}{2}}.
\end{align}
Therefore, combining \eqref{eq.59} and \eqref{eq.60} we have
\begin{equation}\label{eq.61}
  \sup_{x, y}|U_{d, 1, 1}^{l}(t, x, y)|\leq C(\varphi)\cdot  t^{-\frac{d}{2}}.
\end{equation}

For the term \eqref{eq3.55.5},  the oscillatory integral we face becomes
\begin{align}\label{eq3.62}
I_{d, 1, 2}^{l}(t, x, y)&:=\frac{\alpha}{\pi i}\int_0^{\infty}e^{-it\lambda^2+i\lambda|x-x_1|}(F^{\alpha, l}_{+}(\lambda^2)-F^{\alpha, l}_{-}(\lambda^2))\nonumber \\
&\cdot w_{+, >}(\lambda|x-x_1|)w_{+, <}(\lambda|x_2-y|)\lambda^{\frac{d}{2}}\,d\lambda.
\end{align}
(\romannumeral1) In the region $|x-x_1|\leq t^{\frac12}$. When  $d\ge 4$, note that by \eqref{equ6.2.3.3}, \eqref{equ6.2.5} and \eqref{equ6.2.7} we have
$$
\alpha(F^{\alpha, l}_{+}(\lambda^2)-F^{\alpha, l}_{-}(\lambda^2))w_{+, >}(\lambda|x-x_1|)w_{+, <}(\lambda|x_2-y|)\cdot\lambda^{\frac{d}{2}}\in S_{\frac{d}{2}+1}^{d-1}\left((0,\lambda_0)\right)
$$
(in fact, it belongs to $S_{\frac{d}{2}+1}^{d-2+d/2}\left((0,\lambda_0)\right)$, which is better than we need when $d\ge 4$).  Now we apply Lemma \ref{lm2.6} in the above region with $b=d-1$ and $K=\frac{d}{2}+1$ to obtain that
\begin{align}\label{eq.63}
\left|U_{d, 1, 2}^{l}(t, x, y)\right|&\leq  C(\varphi)\cdot t^{-\frac{d}{2}} \int_{\mathbb{R}^{2d}}\frac{|\varphi(x_1)|}{|x-x_1|^{\frac{d}{2}-1}}\cdot\frac{|\varphi(x_2)|}{|x_2-y|^{d-2}} \,dx_1dx_2 \nonumber \\
&\leq  C(\varphi)\cdot t^{-\frac{d}{2}}.
\end{align}
When $d=2$, by \eqref{equ6.2.1.1}, \eqref{equ6.2.5}, \eqref{equ6.2.6} we see that
\begin{equation}\label{eq3.90.1}
  \alpha(F^{\alpha, l}_{+}(\lambda^2)-F^{\alpha, l}_{-}(\lambda^2))w_{+, >}(\lambda|x-x_1|)\cdot (1+|x_2-y|^{-\frac12})^{-1}w_{+, <}(\lambda|x_2-y|)\cdot\lambda\in S_{2}^{1}\left((0,\lambda_0)\right).
\end{equation}
Again we apply Lemma \ref{lm2.6} in the above region with $b=1$ and $K=2$ to derive that
\begin{align}\label{eq.63.1}
\left|U_{d, 1, 2}^{l}(t, x, y)\right|&\leq C(\varphi)\cdot t^{-1} \int_{\mathbb{R}^{2d}}|\varphi(x_1)|\cdot|\varphi(x_2)|(1+|x_2-y|^{-\frac12}) \,dx_1dx_2 \nonumber \\
&\leq C(\varphi)\cdot t^{-1}.
\end{align}
(\romannumeral2)  In the region  $t^{\frac12}<|x-x_1|$. When $d\ge 4$, note that  by \eqref{equ6.2.3.3}, \eqref{equ6.2.5.1} and \eqref{equ6.2.7}, then
$$
\alpha(F^{\alpha, l}_{+}(\lambda^2)-F^{\alpha, l}_{-}(\lambda^2))|x-x_1|^{\frac12} w_{+, >}(\lambda|x-x_1|)w_{+, <}(\lambda|x_2-y|)\lambda^{\frac{d}{2}}\in S_{\frac{d}{2}+1}^{\frac{d-1}{2}}\left((0,\lambda_0)\right)
$$
(in fact, it belongs to the smaller class $S_{\frac{d}{2}+1}^{(3d-5)/2}\left((0,\lambda_0)\right)$).
Now we apply Lemma \ref{lm2.6} in the above region with $b=\frac{d-1}{2}$ and $K=\frac{d}{2}+1$ to obtain that
\begin{align}\label{eq.64}
\left|U_{d, 1, 2}^{l}(t, x, y)\right|&\leq C(\varphi)\cdot t^{-\frac{d}{2}} \int_{\mathbb{R}^{2d}}\frac{|x-x_1|^{\frac{d}{2}-1}}{|x-x_1|^{\frac{d}{2}-1}|x_2-y|^{\frac{d}{2}-1}}|\varphi(x_1)||\varphi(x_2) | \,dx_1dx_2 \nonumber \\
&\leq C(\varphi)\cdot t^{-\frac{d}{2}}.
\end{align}
When $d=2$, by \eqref{equ6.2.1} \eqref{equ6.2.5.1}, and \eqref{equ6.2.6}, we have
\begin{equation}\label{eq3.90.2}
\alpha(F^{\alpha, l}_{+}(\lambda^2)-F^{\alpha, l}_{-}(\lambda^2)|x-x_1|^{\frac12}w_{+, >}(\lambda|x-x_1|)(1+|x_2-y|^{-\frac12})^{-1}w_{+, <}(\lambda|x_2-y|)\lambda\in S_{2}^{\frac12}\left((0,\lambda_0)\right).
\end{equation}
Again we apply Lemma \ref{lm2.6} in the above region with $b=\frac12$ and $K=2$ to derive that
\begin{align}\label{eq.63.1.1}
\left|U_{d, 1, 2}^{l}(t, x, y)\right|&\leq C(\varphi)\cdot t^{-1} \int_{\mathbb{R}^{2}}|\varphi(x_1)|\cdot|\varphi(x_2)|(1+|x_2-y|^{-\frac12})|x-x_1|^{\frac12} |x-x_1|^{-\frac12}  \,dx_1dx_2 \nonumber \\
&\leq C(\varphi)\cdot t^{-1}.
\end{align}
Therefore combining (\romannumeral1)  and (\romannumeral2)  we conclude that
\begin{equation}\label{eq.65}
  \sup_{x, y}|U_{d, 1, 2}^{l}(t, x, y)|\leq  C(\varphi)\cdot t^{-\frac{d}{2}}.
\end{equation}

The treatment for the term \eqref{eq3.56.5} is nearly the same as \eqref{eq3.55.5}, thus we omit.

For the term \eqref{eq3.57.2}, we shall deal with the following oscillatory integral
\begin{align}\label{eq3.66}
I_{d, 1, 4}^{l}(t, x, y)&:=\int_0^{\infty}e^{-it\lambda^2}(F^{\alpha, l}_{+}(\lambda^2)-F^{\alpha, l}_{-}(\lambda^2))
w_{+, <}(\lambda|x-x_1|)w_{+, <}(\lambda|x_2-y|)\lambda\,d\lambda.
\end{align}
Now it follows from \eqref{equ6.2.3.3} and \eqref{equ6.2.7} that when $d\ge 4$,
\begin{equation*}
\alpha\cdot(F^{\alpha, l}_{+}(\lambda^2)-F^{\alpha, l}_{-}(\lambda^2))w_{+, <}(\lambda|x-x_1|)w_{+, <}(\lambda|x_2-y|)\lambda\in S_{\frac{d}{2}+1}^{d-1}((0,\lambda_0)).
\end{equation*}
(In this case, it's exactly what we need when applying Corollary \ref{cor2.6}.)
And by \eqref{equ6.2.1.1}, \eqref{equ6.2.6}, it follows that when $d=2$,
\begin{align}\label{eq3.90.3}
&\alpha(F^{\alpha, l}_{+}(\lambda^2)-F^{\alpha, l}_{-}(\lambda^2))\nonumber \\
 & \cdot(1+|x-x_1|^{-\frac12})^{-1}w_{+, <}(\lambda|x-x_1|)\cdot(1+|x_2-y|^{-\frac12})^{-1}w_{+, <}(\lambda|x_2-y|)\cdot\lambda\in S_{2}^{1}((0,\lambda_0)).
\end{align}
By oscillatory integral estimates \eqref{eq2.22} in Corollary \ref{cor2.6} with $b=d-1, K=\frac d2 +1$, we  have when $d\ge 4$,
\begin{align}\label{eq.63.2}
\left|U_{d, 1, 4}^{l}(t, x, y)\right|&\leq  C(\varphi)\cdot t^{-\frac{d}{2}} \int_{\mathbb{R}^{2d}}\frac{|\varphi(x_1)|}{|x-x_1|^{d-2}}\cdot\frac{|\varphi(x_2)|}{|x_2-y|^{d-2}} \,dx_1dx_2 \nonumber \\
&\leq  C(\varphi)\cdot t^{-\frac{d}{2}},
\end{align}
and when $d=2$,
\begin{align}\label{eq.63.3}
\left|U_{d, 1, 2}^{l}(t, x, y)\right|&\leq Ct^{-1} \int_{\mathbb{R}^{2d}}|\varphi(x_1)|\cdot|\varphi(x_2)|(1+|x_2-y|^{-\frac12})(1+|x-x_1|^{-\frac12}) \,dx_1dx_2 \nonumber \\
&\leq  C(\varphi)\cdot t^{-1}.
\end{align}
Combining \eqref{eq.63.2} and  \eqref{eq.63.3} above we conclude that
\begin{equation}\label{eq.67}
\sup_{x, y}|U_{d, 1, 4}^{l}(t, x, y)|\leq  C(\varphi)\cdot t^{-\frac{d}{2}}.
\end{equation}
Therefore by \eqref{eq.61}, \eqref{eq.65} and \eqref{eq.67}, we finish the estimate for $U_{d, 1}^{ l}$ in \emph{subcase 1}.

\emph{Step 2.}  In this step we consider the term $U_{d, 2}^{ l}$ and prove estimates for \eqref{eq3.54.1}--\eqref{eq3.57.1}. We point out that the procedure is similar to Step 1 with minor modifications. Indeed, in the oscillatory integrals \eqref{eq3.58}, \eqref{eq3.62}, and \eqref{eq3.66},
$F^{\alpha, l}_{+}(\lambda^2)-F^{\alpha, l}_{-}(\lambda^2)$ is replaced by $F^{\alpha, l}_{-}(\lambda^2)$;
$w_{+, >}(\lambda|x-x_1|)$ is replaced by $J_{\pm, >}(\lambda|x-x_1|)$;
$w_{+, <}(\lambda|x-x_1|)$ is replaced by $J_{\pm, <}(\lambda|x-x_1|)$. More precisely,   in \eqref{eq3.54.1} and \eqref{eq3.56.1},  the oscillatory integrals we face are
\begin{align}\label{eq3.69}
&I_{d, 2, 1}^{\pm, l}(t, x, y):=\nonumber \\
&\frac{\alpha}{\pi i}\int_0^{\infty}e^{-it\lambda^2+i\lambda(\pm |x-x_1|+|x_2-y|)}F^{\alpha, l}_{-}(\lambda^2)
\cdot J_{\pm, >}(\lambda|x-x_1|)w_{+, >}(\lambda|x_2-y|)\lambda^{d-1}\,d\lambda,
\end{align}
and
\begin{align}\label{eq3.71}
	&I_{d, 2, 3}^{\pm, l}(t, x, y):= \nonumber \\
	&\frac{\alpha}{\pi i}\int_0^{\infty}e^{-it\lambda^2+i\lambda(\pm |x-x_1|+|x_2-y|)}F^{\alpha, l}_{-}(\lambda^2)\cdot J_{\pm, <}(\lambda|x-x_1|)w_{+, >}(\lambda|x_2-y|)\lambda^{d-1}\,d\lambda.
\end{align}

(\romannumeral1)  In the region $\left||x_2-y|\pm |x-x_1|\right|\leq t^{\frac12}$, we use the fact
$$
\alpha F^{\alpha, l}_{-}(\lambda^2)|x-x_1|^{1-\frac d2}J_{\pm, \gtrless}(\lambda|x-x_1|)w_{+, >}(\lambda|x_2-y|)\lambda^{d-1}\in S_{\frac{d}{2}+1}^{d-1}\left( (0,\lambda_0)\right) ,
$$
which follows from \eqref{equ6.2.3}, \eqref{equ6.2.5}, \eqref{equ6.2.5.3} and \eqref{equ6.2.7.2}.

(\romannumeral2) In the region $||x_2-y|\pm |x-x_1||> t^{\frac12}$, we use the fact
$$
\alpha F^{\alpha, l}_{-}(\lambda^2)(|x-x_1|+ |x_2-y|)^{\frac12} J_{\pm, \gtrless}(\lambda|x-x_1|)w_{+, >}(\lambda|x_2-y|)\lambda^{d-1}\in S_{\frac{d}{2}+1}^{\frac{d-1}{2}}\left( (0,\lambda _0)\right) ,
$$
which follows  from \eqref{equ6.2.3}, \eqref{equ6.2.5.1}, \eqref{equ6.2.5.2} and \eqref{equ6.2.7.1}.
Then by the same arguments in \eqref{eq3.58}--\eqref{eq.60}, we find that
$$
  \sup_{x, y}\left( |U_{d, 2, 1}^{\pm, l}(t, x, y)|+|U_{d, 2, 3}^{\pm, l}(t, x, y)|\right) \leq C(\varphi)\cdot t^{-\frac{d}{2}}.
$$

For the terms \eqref{eq3.55.1} and \eqref{eq3.57.1}, the arguments are similar as before. In particular, the oscillatory integrals become
\begin{align}\label{eq3.70}
I_{d, 2, 2}^{\pm, l}(t, x, y):=\frac{\alpha}{\pi i}\int_0^{\infty}e^{-it\lambda^2\pm i\lambda|x-x_1|}F^{\alpha, l}_{-}(\lambda^2)J_{\pm, >}(\lambda|x-x_1|)w_{+, <}(\lambda|x_2-y|)\lambda^{\frac{d}{2}}\,d\lambda,
\end{align}
and
\begin{align}\label{eq3.72}
	I_{d, 2, 4}^{\pm, l}(t, x, y):=\frac{\alpha}{\pi i}\int_0^{\infty}e^{-it\lambda^2\pm i\lambda|x-x_1|}F^{\alpha, l}_{-}(\lambda^2)J_{\pm, <}(\lambda|x-x_1|)w_{+, <}(\lambda|x_2-y|)\lambda^{\frac{d}{2}}\,d\lambda.
\end{align}

(\romannumeral1)  In the region $|x-x_1|> t^{\frac12}$, observe that by \eqref{equ6.2.3}, \eqref{equ6.2.7}, \eqref{equ6.2.5.3} and \eqref{equ6.2.7.2},
$$
\alpha F^{\alpha, l}_{-}(\lambda^2)|x-x_1|^{1-\frac{d}{2}}J_{\pm,\gtrless}(\lambda|x-x_1|)w_{+, <}(\lambda|x_2-y|)\lambda^{\frac{d}{2}}\in S_{\frac{d}{2}+1}^{d-1}\left( (0,\lambda_0)\right)
$$
holds for  $d\ge4$. And by \eqref{equ6.2.1}, \eqref{equ6.2.6}, \eqref{equ6.2.5.3} and \eqref{equ6.2.7.2}
\begin{equation}\label{eq3.90.4}
\alpha F^{\alpha, l}_{-}(\lambda^2) \cdot J_{\pm, \gtrless}(\lambda|x-x_1|)\cdot(1+|x_2-y|^{-\frac12})^{-1}w_{+, <}(\lambda|x_2-y|)\lambda\in S_{2}^{1}\left( (0,\lambda_0)\right)
\end{equation}
holds for $d=2$.

(\romannumeral2)  In the region $|x-x_1|\leq t^{\frac12}$, one observes that by \eqref{equ6.2.3}, \eqref{equ6.2.7}, \eqref{equ6.2.5.2} and \eqref{equ6.2.7.1},
$$
\alpha F^{\alpha, l}_{-}(\lambda^2)\cdot|x-x_1|^{\frac12}J_{\pm,\gtrless}(\lambda|x-x_1|)\cdot w_{+, <}(\lambda|x_2-y|)\lambda^{\frac{d}{2}}\in S_{\frac{d}{2}+1}^{\frac{d-1}{2}}\left( (0,\lambda_0)\right)
$$
holds for  $d\ge4$. And by\eqref{equ6.2.1}, \eqref{equ6.2.6}, \eqref{equ6.2.5.2} and \eqref{equ6.2.7.1}
\begin{equation}\label{eq3.90.5}
\alpha F^{\alpha, l}_{-}(\lambda^2)\cdot |x-x_1|^{\frac12} J_{\pm, \gtrless}(\lambda|x-x_1|)\cdot(1+|x_2-y|^{-\frac12})^{-1}w_{+, <}(\lambda|x_2-y|)\lambda\in S_{2}^{\frac 12}\left( (0,\lambda_0)\right) 	
\end{equation}
holds for $d=2$.
Then by the same arguments in \eqref{eq3.62}--\eqref{eq.63.1.1}, we find that
$$
  \sup_{x, y\in\mathbb{R}^d}\left(\left|U_{d, 2, 2}^{l}(t, x, y)\right|+\left|U_{d, 2, 4}^{l}(t, x, y)\right|\right)\leq Ct^{-\frac{d}{2}}.
$$
Therefore  we finish the estimate for $U_{d, 2}^{ l}$ in \emph{subcase 1}.

\emph{Step 3.} In this step we consider the term $U_{d, 3}^{l}$. In view of the definition \eqref{eq.l7}, the proof is the same as $U_{d, 2}^{l}$ and we omit the details.

By \emph{Step 1}--\emph{Step 3}, we complete the proof of  \emph{Subcase 1}.

$\bullet$ {\emph{Subcase 2: $d=2$ and $\int_{\mathbb{R}^2}{\varphi(x)\,dx}=0$.}}

We first point out that in this case, the lower bound \eqref{eq.11.1.4} is no longer valid, instead we only have the weaker estimate \eqref{eq.11.1.2} (there is a loss of $\log{\lambda}$  in  \eqref{eq.11.1.2}).
However, we can borrow ideas in the case $d=1$ and  use cancellation property of $\varphi$ to eliminate the $\log{\lambda}$ singularity appeared in $R_0^{\pm}(\lambda^2)\varphi$. More precisely,  by \eqref{eq3.50} and the  assumption  $\int_{\mathbb{R}^2}{\varphi(x)\,dx}=0$, we write
\begin{align}\label{eq3.3.69}
R_0^{\pm}(\lambda^2)\varphi&=\int_{\mathbb{R}^2}{e^{\pm i\lambda|x-x_1|}w_{\pm, >}(\lambda|x-x_1|)\varphi(x_1)\,dx_1}+\int_{\mathbb{R}^2}{w_{\pm, <}(\lambda|x-x_1|)\varphi(x_1)\,dx_1}\nonumber\\
  &=\int_{\mathbb{R}^2}{e^{\pm i\lambda|x-x_1|}w_{\pm, >}(\lambda|x-x_1|)\varphi(x_1)\,dx_1}\nonumber\\
  &+\int_{\mathbb{R}^2}{\left(w_{\pm, <}(\lambda|x-x_1|)+\frac{\log{\lambda\langle x\rangle}}{2\pi}\omega(\lambda\langle x\rangle)\right)\varphi(x_1)\,dx_1}.
\end{align}
We denote
\begin{equation}\label{eq3.3.70}
  w'_{\pm, <}(\lambda, x, x_1):=w_{\pm, <}(\lambda|x-x_1|)+\frac{\log{\lambda\langle x\rangle}}{2\pi}\omega(\lambda\langle x\rangle),
\end{equation}
where $\omega(\lambda)$ is the same  bump function used in \eqref{eq3.50}, i.e., $\omega(\lambda)\in C_0^{\infty}$, $\omega(\lambda)=1$, if $|\lambda|<\frac12$ and  $\omega(\lambda)=0$, if $|\lambda|>1$.  We claim that
\begin{equation}\label{eq3.3.71}
  |w'_{\pm, <}(\lambda, x, x_1)|\leq C\left(1+\langle x_1\rangle+\frac{\langle x_1\rangle}{|x-x_1|}\right)^{\frac12},\,\,\,\,\,0<\lambda<\lambda_0,
\end{equation}
and
\begin{equation}\label{eq3.3.72}
  |\frac{d^k}{d\lambda^k} w'_{\pm, <}(\lambda, x, x_1)|\leq C\lambda^{-k},\,\,\,\,\,0<\lambda<\lambda_0,\,\,\,\,k=1,2.
\end{equation}
In fact, \eqref{eq3.3.72} follows directly from the definition \eqref{eq3.3.70} and \eqref{eq3.25.11}. Then it suffices to prove \eqref{eq3.3.71}, we recall that $w_{\pm, <}(\lambda|x-x_1|)=\frac{\pm i}{4}H^{\pm}_0(\lambda|x-x_1|)\omega(\lambda|x-x_1|)$  (see \eqref{eq3.50}).  Now we rewrite
\begin{align*}
 w'_{\pm, <}(\lambda, x, x_1)&=\left(\frac{\pm i}{4}H^{\pm}_0(\lambda|x-x_1|)+\frac{\log{\lambda|x-x_1|}}{2\pi}\right)\omega(\lambda|x-x_1|)\\
& +\left(\frac{\log{\lambda\langle x\rangle}}{2\pi}\omega(\lambda\langle x\rangle)-\frac{\log{\lambda|x-x_1|}}{2\pi}\omega(\lambda|x-x_1|)\right)\\
&:=\uppercase\expandafter{\romannumeral1}+\uppercase\expandafter{\romannumeral2}.
\end{align*}
It's enough to show that both \uppercase\expandafter{\romannumeral1} and \uppercase\expandafter{\romannumeral2} satisfy the estimate \eqref{eq3.3.71}.
We first note that the expansion of $H^{\pm}_0(\lambda|x-x_1|)$ (see \cite{AS}) is as follows:
$$
H^{\pm}_0(\lambda|x-x_1|)=1\pm i\frac{2\gamma}{\pi}\pm i\frac{2}{\pi}\log{\frac{\lambda|x-x_1|}{2}}+O(\log{\lambda|x-x_1|}\cdot(\lambda|x-x_1|)^2), \,\,\text{for} \,\, \lambda|x-x_1|<1.$$
Then a direct computation shows that
$$
\left|\uppercase\expandafter{\romannumeral1}\right|\leq C.
$$
Second, we further break \uppercase\expandafter{\romannumeral2} into three terms:
\begin{align*}
&\frac{\log{\lambda\langle x\rangle}}{2\pi}\omega(\lambda\langle x\rangle)-\frac{\log{\lambda|x-x_1|}}{2\pi}\omega(\lambda|x-x_1|)
=\frac{\log{\lambda\langle x\rangle}}{2\pi}\omega(\lambda\langle x\rangle)(1-\omega(\lambda|x-x_1|))\\
  &+\log{\frac{\langle x\rangle}{|x-x_1|}}\omega(\lambda\langle x\rangle)\omega(\lambda|x-x_1|)
-\frac{\log{\lambda|x-x_1|}}{2\pi}(1-\omega(\lambda\langle x\rangle))\omega(\lambda|x-x_1|).
\end{align*}
To estimate the first term on the right hand side,  we note that
\begin{align*}
 |\frac{\log{\lambda\langle x\rangle}}{2\pi}\omega(\lambda\langle x\rangle)(1-\omega(\lambda|x-x_1|))| &\leq C(\lambda\langle x\rangle)^{-\frac12}|(1-\omega(\lambda|x-x_1|))|\\
&\leq C\left(\frac{|x-x_1|}{\langle x\rangle}\right)^{\frac12},
\end{align*}
where in the first inequality, we use the fact that
\begin{equation*}
\left|(\lambda\langle x\rangle)^{\frac12}\cdot\log{\lambda\langle x\rangle}\cdot\omega(\lambda\langle x\rangle)\right|\leq C,
\end{equation*}
and in the second inequality, we use the fact that
\begin{equation*}
 \left|1-\omega(\lambda|x-x_1|)\right|\leq C\cdot(\lambda|x-x_1|)^{\frac12}.
\end{equation*}
Using the same trick, we can also estimate the third term  on the right hand side and obtain that
$$
\left|\frac{\log{\lambda|x-x_1|}}{2\pi}(1-\omega(\lambda\langle x\rangle))\cdot\omega(\lambda|x-x_1|)\right|\leq C\cdot\left(\frac{\langle x\rangle}{|x-x_1|}\right)^{\frac12}.
$$
Therefore \eqref{eq3.3.71} follows from the following elementary inequalities:
$$
\left(\frac{|x-x_1|}{\langle x\rangle}\right)^{\frac12}+\left(\frac{\langle x\rangle}{|x-x_1|}\right)^{\frac12}\leq C\left(1+\langle x_1\rangle+\frac{\langle x_1\rangle}{|x-x_1|}\right)^{\frac12},
$$
$$
\left|\log{\frac{\langle x\rangle}{|x-x_1|}}\right|\leq C\left(1+\langle x_1\rangle+\frac{\langle x_1\rangle}{|x-x_1|}\right)^{\frac12}.
$$
\eqref{eq3.3.71} and \eqref{eq3.3.72} imply that
\begin{equation}\label{eq3.67}
\left(1+\langle x_1\rangle+\frac{\langle x_1\rangle}{|x-x_1|}\right)^{-\frac12}w'_{\pm, <}(\lambda, x, x_1)\in S_{2}^0\left((0,\lambda_0)\right).
\end{equation}
We remark that comparing this with \eqref{equ6.2.6}, there is a gain of $\log{\lambda}$  in \eqref{eq3.67}.
Thanks to \eqref{eq3.3.69} and \eqref{eq3.3.70}, we can replace $w_{\pm, <}(\lambda, x, x_1)$ and $w_{\pm, <}(\lambda, x_2,y)$ by $w'_{\pm, <}(\lambda, x, x_1)$ and $w'_{\pm, <}(\lambda, x_2, y)$ in \eqref{eq3.54.5}--\eqref{eq3.57.1} when $d=2$. This allows us to use the same arguments in \emph{Subcase 1}. Indeed, it suffices to prove that the relations \eqref{eq3.90.1}, \eqref{eq3.90.2}, \eqref{eq3.90.3}, \eqref{eq3.90.4} and
\eqref{eq3.90.5} still hold when
\begin{equation*}
(1+|x-x_1|^{-\frac12})^{-1}w_{\pm, <}(\lambda|x-x_1|)\,\,\, \text{and} \,\,\,(1+|x_2-y|^{-\frac12})^{-1}w_{\pm, <}(\lambda|x_2-y|)
\end{equation*}
are replaced by
\begin{equation*}
\left(1+\langle x_1\rangle+\frac{\langle x_1\rangle}{|x-x_1|}\right)^{-\frac12}w'_{\pm, <}(\lambda, x, x_1)\,\,\, \text{and} \,\,\, \left(1+\langle x_2\rangle+\frac{\langle x_2\rangle}{|x_2-y|}\right)^{-\frac12}w'_{\pm, <}(\lambda, x_2, y)
\end{equation*}
respectively.
But this follows from \eqref{eq3.67}, together with  \eqref{equ6.2.3}, \eqref{equ6.2.3.3}.

Therefore the proof of \eqref{eq.l1} is finished for all $d\ge 1$.

We are left with the proof of the estimate \eqref{eq.11.11}. In this case,  only a small modification is needed. In fact, in \eqref{eq.l7}, we first fix $\lambda_0=\frac12$, then we replace
 $$
 \alpha F^{\alpha, l}_{\pm}(\lambda^2) \qquad \text{and} \qquad  \alpha(F^{\alpha, l}_{+}(\lambda^2)-F^{\alpha, l}_{-}(\lambda^2))
 $$ by
 $$
 M^{-2([\frac{d}{2}]+1)} F^{\alpha, l}_{\pm}(\lambda^2) \qquad \text{and} \qquad  M^{-2([\frac{d}{2}]+2)} (F^{\alpha, l}_{+}(\lambda^2)-F^{\alpha, l}_{-}(\lambda^2))
 $$
respectively. During the proof, we  use \eqref{eq.11.1.2'} and \eqref{eq.11.1.3'} instead of \eqref{equ6.2.3} and \eqref{equ6.2.3.3}, therefore by following the same arguments as before we obtain \eqref{eq.11.11}. Moreover, the  following supplementary remarks are in order:

(\romannumeral1) Since we use Lemma \ref{lm3.3}, we require that $d\ge 3$ or $d=1, 2$, $\int_{\mathbb{R}^d}\varphi(x)\,dx=0$.

(\romannumeral2) Observe that there is an additional $M^2$ coming from  two $\varphi$'s in \eqref{eq.l7}. Hence the upper bound in  \eqref{eq.l1}  becomes $C\cdot M^{2([\frac{d}{2}]+3)}\cdot\alpha$.

(\romannumeral3) Finally note that by Lemma \ref{lm3.3}, the constants $C_k$ in \eqref{eq.11.1.6} and \eqref{eq.11.1.7} don't dependent on $\varphi$ and  we have fixed $\lambda_0=\frac12$, therefore by comparing the proof of \eqref{eq.l1}, we show that the constant $C$ in  \eqref{eq.11.11} does not dependent on $\varphi$.

Summing up, we complete the proof of Theorem \ref{thm3.2}. \qed

\begin{remark}\label{rmk3.4}
We mention that during the proof of the low energy estimate \eqref{eq.l1}, all properties that we need for  $F^{\alpha, l}_{\pm}(\lambda^2)$ are contained in Lemma \ref{lm3.2}. This will be used in the analysis of finite rank perturbations.

\end{remark}

\emph{Proof of Theorem \ref{thm-1.1}}
The estimate \eqref{eq1.5} follows from Theorem \ref{thm3.1} and \eqref{eq.l1} in Theorem \ref{thm3.2} with  $\lambda_0$  given in Lemma \ref{lm2.2}.
And the estimate \eqref{eq1.5.0} follows from Theorem \ref{thm3.1} and \eqref{eq.11.11} in Theorem \ref{thm3.2} with  $\lambda_0=\frac12$.

\section{$L^1-L^{\infty}$ estimates for finite rank perturbations}\label{sec4}
\subsection{Aronszajn-Krein formula for finite rank perturbations}\label{sec4.1}
Now we consider the Laplacian with  finite rank perturbations
\begin{equation}\label{eq4.1}
H=H_0+\sum_{j=1}^NP_j, \,\,\,\,\, H_0=-\Delta, \,\,\, P_j=\langle\cdot,\varphi_j\rangle \varphi_j.
\end{equation}
As we have seen in the case of rank one perturbation, the Aronszajn-Krein formula \eqref{eq3.2} plays a critical role during the proof. Then it's natural to ask whether there is an analogue of \eqref{eq3.2} for finite rank perturbations.
We first confirm this by  establishing the following  result, which may be interesting in its own right. We define $R(z):=(H-z)^{-1}$ and $R_0(z)$ is the same as before.
\begin{proposition}\label{prop4.1}
Let $\varphi_j\in L^2(\mathbb{R}^d)$, $1\leq j\leq N$. Assume that
\begin{equation*}
  \langle\varphi_i, \varphi_j\rangle=\delta_{ij},  \,\,\,\,\, 1\leq i, j\leq N.
\end{equation*}
Then for any $z\in \mathbb{C}\setminus[0, \infty)$,
\begin{equation}\label{eq4.2}
R(z)=R_0(z)-\sum_{i,j=1}^N{g_{ij}(z)}R_0(z)\varphi_i\left\langle R_0(z)\cdot,\varphi_j\right\rangle,
\end{equation}
where $G=(g_{ij}(z))_{N\times N}$ is the inverse matrix of $A=(a_{ij}(z))_{N\times N}$, with
\begin{equation}\label{eq4.3}
 a_{ij}(z)=\delta_{ij}+\langle R_0(z)\varphi_j,\varphi_i\rangle, \,\,\,\,\, 1\leq i, j\leq N.
\end{equation}
\end{proposition}
\begin{remark}\label{rmk4.1}
Note that when $N=1$, \eqref{eq4.2} coincides with \eqref{eq3.2} ($\alpha=1$).  Therefore we call \eqref{eq4.2} Aronszajn-Krein formula for finite rank perturbations. We mention that using the same arguments, one can also obtain similar result for $H=H_0+\sum_{j=1}^N\alpha_j P_j$, $\forall$ $\alpha_j>0$.
\end{remark}
\begin{proof}
Let $P=\sum_{j=1}^NP_j$.
We start by the strong resolvent identity formula
\begin{equation}\label{eq4.4}
 R(z)=R_0(z)-R(z)PR_0(z).
\end{equation}
Acting $P$ on both sides and using $P^2=P$, we get
\begin{equation}\label{eq4.4.1}
 R(z)P(I+PR_0(z)P)=R_0(z)P.
\end{equation}
First we prove the following

\textbf{Claim 1.} \,  $I+PR_0(z)P$ is invertible in $PL^2(\mathbb{R}^d)=\text{Span}\{\varphi_j, 1\leq j\leq N\}$ for any $z\in \mathbb{C}\setminus[0, \infty)$.

\emph{Proof of Claim 1.}  Since $PL^2(\mathbb{R}^d)$ is of dimension $N$, it suffices to show that  $I+PR_0(z)P$  is injective. Assume that
\begin{equation}\label{eq4.5}
 (I+PR_0(z)P)u=0, \,\,\,\,\,u\in PL^2(\mathbb{R}^d).
\end{equation}
Denote by $f=R_0(z)u\in L^2$, then it follows from \eqref{eq4.5} that
\begin{equation*}
  u=-PR_0(z)u=-Pf.
\end{equation*}
Since $P$ is self-adjoint, we have
\begin{equation}\label{eq4.6}
 0\leq \langle u, u\rangle=-\langle Pf, u\rangle=-\langle f, Pu\rangle =-\langle R_0(z)u, u\rangle,
\end{equation}
which shows that $\langle R_0(z)u, u\rangle$ is real valued. If $\Im z\ne 0$, then
\begin{equation*}
0=\Im\left\langle R_0(z)u, u\right\rangle=\int_{\mathbb{R}^d}{\frac{\Im z \cdot |\hat{u}(\xi)|^2}{(|\xi|^2-\Re z)^2+(\Im z)^2}},
\end{equation*}
which implies that $\hat{u}(\xi)\equiv 0$, hence $u\equiv0$.
If $z<0$, then $(R_0(z)u, u)\ge 0$ by the positivity of the operator $-\Delta-z$. This, together with \eqref{eq4.6}, implies $u=0$.
Therefore the proof of the claim is complete.

Now we are able to rewrite \eqref{eq4.4.1} as
\begin{equation*}
  R(z)P=R_0(z)P(I+PR_0(z)P)^{-1}P.
\end{equation*}
Put this formula back to the resolvent identity, we obtain that
\begin{equation}\label{eq4.7}
R(z)=R_0(z)-R_0(z)P(I+PR_0(z)P)^{-1}PR_0(z).
\end{equation}

To proceed, we also need the following observation.

\textbf{Claim 2.} $I+PR_0(z)P$ is invertible in $PL^2(\mathbb{R}^d)$ $\Longleftrightarrow$ the matrix $I_{N\times N}+\left(f_{ij}(z)\right)_{N\times N}$ is nonsingular, where $z\in \mathbb{C}\setminus[0, \infty)$,  and
$$
f_{ij}(z):=\langle R_0(z)\varphi_j, \varphi_i\rangle.
$$
\emph{Proof of Claim 2.} We set $\mathbb{A}(z):=C A(z) B$, where
\begin{align}\label{eq4.9}
B&:\, PL^2\rightarrow C^N:\,\, Bu=(x_1,\cdots, x_N)^T, \qquad  u=\sum_{j=1}^N{x_j\varphi_j}, \nonumber\\
A(z)&:\, C^N\rightarrow C^N:\,\,   A=I_{N\times N}+\left(f_{ij}(z)\right)_{N\times N},
\end{align}
and
$$
C:\, C^N\rightarrow PL^2:\,\, Cx=\sum_{j=1}^N{x_j\varphi_j}, \qquad x=(x_1,\cdots, x_N)^T.
$$
Thus $\mathbb{A}(z):\, PL^2\rightarrow PL^2$. And a direct computation shows that
$$
\mathbb{A}(z)\varphi_j=\sum_{i}{(\delta_{ij}+f_{ij}(z))\varphi_i}=(I+PR_0(z)P)\varphi_j,\,\,\,\,\,1\le j\le N.
$$
Equivalently, this means
$$
I+PR_0(z)P=C A(z) B  \quad\text{on}\,\,\,PL^2.
$$
Note that both $B$ and $C$ are bijective, therefore the claim follows. Furthermore, by \textbf{Claim 1},  one has for any $z\in \mathbb{C}\setminus[0, \infty)$,
\begin{equation}\label{eq4.8}
(I+PR_0(z)P)^{-1}=BA(z)^{-1}C.
\end{equation}
Let $G(z)=(g_{ij}(z))_{N\times N}$ be the inverse matrix of $A(z)$, then
it follows from \eqref{eq4.7} and \eqref{eq4.8} that
\begin{equation*}
R(z)=R_0(z)-\sum_{i,j=1}^N{g_{ij}(z)}R_0(z)\varphi_i\langle R_0(z)\cdot,\varphi_j\rangle,
\end{equation*}
i.e., \eqref{eq4.2} is valid.
\end{proof}

Next, we shall consider the boundary values of $R_0(z)$ and  extend \eqref{eq4.2} to the positive real axis. Assume that $\varphi_j\in L^2_{\sigma}$, $\sigma>\frac12$, $j=1,2,\ldots, N$. Thanks to the limiting absorption principle, one can define
$$
F_{N\times N}^{\pm}(\lambda^2)=\left(f_{ij}^{\pm}(\lambda^2)\right)_{N\times N},\,\,\,\text{where}   \,\,f_{ij}^{\pm}(\lambda^2):=\langle R_0^{\pm}(\lambda^2)\varphi_j, \varphi_i\rangle,\,\,\,\,\,1\leq i, j\leq N,
$$
moreover, each $f_{ij}^{\pm}(\lambda^2)$ is continuous on $(0, \infty)$. Denote the boundary value of the matrix $A(z)$ in \eqref{eq4.3} by
\begin{equation}\label{eq4.11}
  A^{\pm}(\lambda^2):=A(\lambda^2\pm i0)=I_{N\times N}+F_{N\times N}^{\pm}(\lambda^2),\qquad\,\,\,\lambda>0.
\end{equation}
The following result shows that one can extend \eqref{eq4.2} to the boundary provided that the matrix $A^{\pm}(\lambda^2)$ is invertible. More precisely, we have
\begin{theorem}\label{thm4.2}
Let $\lambda>0$, and assume that $\varphi_j\in L^2_{\sigma}$, $\sigma>1$, $j=1,2,\ldots, N$.

(\romannumeral1) If $\lambda^2\in \sigma_c(H)$, then $\det( A^{\pm}(\lambda^2))\ne 0$.

(\romannumeral2) Conversely, if $\det( A^{\pm}(\lambda^2))\ne 0$, then $\lambda^2\in \sigma_{ac}(H)$. In this case, we have the Aronszajn-Krein formula:
\begin{equation}\label{eq4.13}
 R^{\pm}(\lambda^2)= R^{\pm}_{0}(\lambda^2)-\sum_{i,j=1}^N{g^{\pm}_{ij}(\lambda^2)}\cdot R^{\pm}_{0}(\lambda^2)\varphi_i\langle R^{\pm}_{0}(\lambda^2)\cdot,  \varphi_j\rangle, \tag{\ref{eq4.2}'}
\end{equation}
where
\begin{equation}\label{eq4.13.1}
G^{\pm}(\lambda^2):=\left(g_{ij}^{\pm}(\lambda^2)\right)_{N\times N}=A^{\pm}(\lambda^2)^{-1}.
\end{equation}
\end{theorem}
\begin{proof}
First we prove (\romannumeral1) by contradiction. Assume that $\det( A^{\pm}(\lambda^2))=0$. Note that \textbf{ Claim 2} in the proof of Proposition \ref{prop4.1} still holds when $z=\lambda^2\pm i0$, then $I+PR_0^{\pm}(\lambda^2)P$ is not invertible on the finite dimensional space $PL^2(\mathbb{R}^d)$. Thus there exists a
$u\in PL^2(\mathbb{R}^d)$, such that $(I+PR^{\pm}_0(\lambda^2)P)u=0$. Denote by $f=R^{\pm}_0(\lambda^2)u$, by limiting absorption principle we have $f\in L^{2}_{-\sigma}$. Similar to \eqref{eq4.6}, we have $\langle R^{\pm}_0(\lambda^2)u, u\rangle$ is real valued, then
$$
0=\lim_{\epsilon\rightarrow 0}\int_{\mathbb{R}^d}{\frac{\pm \epsilon \cdot |\hat{u}(\xi)|^2}{(|\xi|^2-\lambda)^2+\epsilon^2}}=c\lambda^{d-2}\int_{S^{d-1}}{|\hat{u}(\lambda \xi|)|^2\,\sigma(d\xi)}
$$
with some constant $c\ne 0$. Hence $\hat{u}=0$ on $\lambda S^{d-1}$ in the $L^2$ trace sense. Then  Agmon's classical results for the decay of eigenfunctions of Schr\"{o}dinger operators \cite{Ag} imply that
$$
(1+|x|)^{\sigma-1}R^{\pm}_0(\lambda^2)u\in L^2.
$$
Our assumption $\sigma>1$ shows that actually we have $f=R^{\pm}_0(\lambda^2)u\in L^2$. In addition, since
 $(I+PR^{\pm}_0(\lambda^2)P)u=0$, then $f$ satisfies
$$
(-\Delta+P-\lambda^2)f=0.
$$
Thus $f$ is an eigenfunction of $H=-\Delta+P$ with  eigenvalue $\lambda^2$, but this contradicts to the assumption $\lambda^2\in \sigma_c(H)$.

 Next we prove (\romannumeral2). If $\det( A^{\pm}(\lambda^2))\ne 0$, then by the continuity of $f_{ij}^{\pm}(\lambda^2$) and the definition of $g_{ij}^{\pm}(\lambda^2)$ in \eqref{eq4.13.1},
 it follows that $g_{ij}^{\pm}(\lambda^2)$ is also continuous. Furthermore, for any $f, g\in L^2_{\sigma}$, $\sigma>\frac12$,
\begin{equation*}
	\langle R(\lambda^2\pm i\epsilon)f, g\rangle=
	\langle R_0(\lambda^2\pm i\epsilon)f, g\rangle-\sum_{i,j=1}^N{g_{ij}(\lambda^2\pm i\epsilon)}\cdot \langle R_0(\lambda^2\pm i\epsilon)\varphi_i, g\rangle\langle R_0(\lambda^2\pm i\epsilon)f,\varphi_j\rangle,
\end{equation*}
Then, the limiting absorption principle, together with the continuity of $g_{ij}(z)$,  shows that the the boundary value of $\langle R(\lambda^2+i\epsilon)f, g\rangle$ exists, and
\begin{equation*}
	\langle R^{\pm}(\lambda^2)f, g\rangle=\langle R^{\pm}_{0}(\lambda^2)f, g\rangle-\sum_{i,j=1}^N{g^{\pm}_{ij}(\lambda^2)}\cdot \langle R^{\pm}_{0}(\lambda^2)\varphi_i, g\rangle\langle R^{\pm}_{0}(\lambda^2)f, \varphi_j\rangle.
\end{equation*}
Therefore $\lambda^2\in \sigma_{ac}(H)$ and \eqref{eq4.13} holds.
\end{proof}

Now we are ready to present an integral formula for the evolution $e^{-itH}$. Note that by Lemma \ref{lmA.1}, the spectrum of $H$ is purely absolutely continuous. Then by the Stone's formula and the resolvent identity \eqref{eq4.7}, we have formally,
\begin{align}\label{eq4.13.0}
e^{-itH}-e^{-itH_{0}}&=-\frac{1}{\pi i}\int_0^{\infty}e^{-it\lambda^2}R^{+}_{0}(\lambda^2)P(1+PR^{+}_{0}(\lambda^2)P)^{-1}PR^{+}_{0}(\lambda^2)\cdot \lambda\,d\lambda\nonumber\\
&+ \frac{1}{\pi i}\int_0^{\infty}e^{-it\lambda^2}R^{-}_{0}(\lambda^2)P(1+PR^{-}_{0}(\lambda^2)P)^{-1}PR^{-}_{0}(\lambda^2)\cdot \lambda\,d\lambda
\end{align}
Using Aronszajn-Krein formula \eqref{eq4.13}, we cam rewrite \eqref{eq4.13.0} as
\begin{align}\label{eq4.13.2}
e^{-itH}-e^{-itH_{0}}&=-\frac{1}{\pi i}\sum_{i,j=1}^N\int_0^{\infty}e^{-it\lambda^2}g^+_{ij}(\lambda^2)R^{+}_{0}(\lambda^2)\varphi_i\langle R^{+}_{0}(\lambda^2)\, \cdot, \varphi_j\rangle \lambda\,d\lambda\nonumber\\
&+\frac{1}{\pi i}\sum_{i,j=1}^N\int_0^{\infty}e^{-it\lambda^2}g^-_{ij}(\lambda^2)R^{-}_{0}(\lambda^2)\varphi_i\langle R^{-}_{0}(\lambda^2)\, \cdot, \varphi_j\rangle \lambda\,d\lambda\nonumber\\
&:=\sum_{i,j=1}^N{\left(U^+_{d, ij}-U^-_{d, ij}\right)}.\tag{\ref{eq4.13.0}'}
\end{align}
This formula plays the same role as \eqref{eq3.3} in the  rank one case (with $\alpha=1$). Moreover, in the following subsections, we shall show how  to reduce  the proofs of Theorem \ref{thm1.2}--\ref{thm1.5} to Theorem \ref{thm-1.1}.

\subsection{Proof of Theorem  \ref{thm1.2}}
First of all, by comparing \eqref{eq4.13.2} with \eqref{eq3.3}, we find that the main difference is that in \eqref{eq4.13.2}, $\frac{1}{1+ F^{\pm}(\lambda^2)}$ is replaced by $g^{\pm}_{ij}(\lambda^2)$ (note that $\varphi_i$ and $\varphi$ satisfy the same smooth and decay condition). Thus the main task is to  prove that $g^{\pm}_{ij}(\lambda^2)$ has the same properties as $\frac{1}{1+ F^{\pm}(\lambda^2)}$.  We recall that in the rank one case and $d=1, 2$,  whether $\int_{\mathbb{R}^d}{\varphi\,dx}=0$ or not may affect the analysis in the low energy part. Such problem also occurs in the finite rank perturbations, since the arguments are more complicated, we will treat them separately.

As in the rank one case, we  first  choose a smooth cut-off function $\chi(\lambda)$ satisfying \eqref{eq3.3.1} with $\lambda_0$  a small constant to be fixed later. Then we split \eqref{eq4.13.2} into the  high and low energy part respectively:
\begin{align*}
  U^{\pm, h}_{d, ij}&:=-\frac{1}{\pi i}\int_0^{\infty}e^{-it\lambda^2}(1-\chi(\lambda))g^{\pm}_{ij}(\lambda^2)R^{\pm}_{0}(\lambda^2)\varphi_i\langle R^{\pm}_{0}(\lambda^2)\, \cdot, \varphi_j\rangle\lambda\,d\lambda,\\
 U^{\pm, l}_{d, ij}&:=-\frac{1}{\pi i}\int_0^{\infty}e^{-it\lambda^2}\chi(\lambda)g^{\pm}_{ij}(\lambda^2)R^{\pm}_{0}(\lambda^2)\varphi_i\langle R^{\pm}_{0}(\lambda^2)\, \cdot, \varphi_j\rangle\lambda\,d\lambda.
\end{align*}
We  divide our proof into the following four steps.

{\bf Step 1. Estimates for high energy part when $d\ge 1$.} By Remark \ref{rmk3.1}, it's enough to prove that for any $\lambda_0>0$, there exists an absolute constant $C=C(N, \lambda_0, \varphi_1,\ldots, \varphi_N)>0$ such that for all $\lambda>\lambda_0$, and $1\leq i, j\leq N$, one has
\begin{equation}\label{eq4.14}
\begin{split}
|\frac{d^k}{d\lambda^k}g_{ij}^{\pm}(\lambda^2)|\leq
\begin{cases}C,\,\,\,\,\,\,\,\quad\quad  \mathrm{if}\,\,\,k=0,\\
C\lambda^{-1}, \,\, \,\,\,\quad \mathrm{if}\,\,\,~k=1,\ldots,[\frac{d}{2}]+1.
\end{cases}
\end{split}
\end{equation}
Indeed, we  observe that similar to the proof of \eqref{eq2.11.2} in Lemma \ref{lm2.1}, it follows from the decay assumptions on $\varphi_j$ that for all $\lambda>\lambda_0$,
\begin{equation}\label{eq4.14.1}
  |\frac{d^k}{d\lambda^k}f^{\pm}_{i,j}(\lambda^2)|\leq C\lambda^{-1}, \,\,\,\quad \mathrm{if}\,\,\,~k=0. 1,\ldots,[\frac{d}{2}]+1.
\end{equation}
Recall that by \eqref{eq4.13.1},
\begin{equation}\label{eq4.14.2}
\left(g_{ij}^{\pm}(\lambda^2)\right)_{N\times N}=A^{\pm}(\lambda^2)^{-1}=\frac{1}{\det{(A^{\pm}(\lambda^2))}}\text {adj}(A^{\pm}(\lambda^2)),
\end{equation}
where $\text{adj}(A^{\pm}(\lambda^2))$ is the adjugate matrix of $A^{\pm}(\lambda^2)$ (given in \eqref{eq4.11}).
Note that by assumption \eqref{eq1.13.1} one has $|\det{A^{\pm}(\lambda^2)}|\ge c_0>0$, thus by \eqref{eq4.14.1} and \eqref{eq4.14.2},  there is some positive constant depending on $N, \lambda_0, \varphi_1,\ldots, \varphi_N$ such that
\begin{equation}\label{eq4.14.2.0}
|g_{ij}^{\pm}(\lambda^2)|\leq C(N, \lambda_0, \varphi_1,\ldots, \varphi_N),
\end{equation}
which proves \eqref{eq4.14} for $k=0$. Further, we notice that
\begin{equation}\label{eq4.14.3}
  \frac{d^k}{d\lambda^k}A^{\pm}(\lambda^2)^{-1}=A^{\pm}(\lambda^2)^{-1}\sum_{s=1}^{k}\sum_{l_{m,1}+\ldots+l_{m,s}=k}
  C(l_{m,1},\cdots,l_{m,s})\prod_{m=1}^s\left[\frac{d^{l_{m,s}}}{d\lambda^{l_{m,s}}}(A^{\pm}(\lambda^2)) \cdot A^{\pm}(\lambda^2)^{-1}\right].
\end{equation}
This, together with \eqref{eq4.14.1}, proves \eqref{eq4.14} for $k=1,\ldots,[\frac{d}{2}]+1$.

Now we turn to the low energy estimates. Before proceeding, we point out the following fact that will be used for all $d\ge 1:$

\textbf{Fact 1.} All the results in Lemma \ref{lm2.3} and Lemma \ref{lm2.4} are still valid  when the inner product $F^{\pm}(\lambda^2)=\langle R^{\pm}_{0}(\lambda^2)\varphi,\varphi\rangle$ is replaced by $f^{\pm}_{i,j}=\langle R^{\pm}_{0}(\lambda^2)\varphi_i, \varphi_j\rangle$.

In fact, this follows immediately  since each $\varphi_j, j=1,\cdots, N$, has the same decay assumption as $\varphi$ in the rank one case.

{\bf Step 2. Estimates for Low energy part  when $d\ge 3$.}
By Remark \ref{rmk3.4}, it suffices to prove that
\begin{equation}\label{eq4.15}
g^{\pm}_{i,j}(\lambda^2)\in S^{0}_{[\frac{d}{2}]+1}((0, \lambda_0)), \,\,\,\qquad\,\,\,\,\, \,\qquad\mathrm{if}\,\,\, d\ge 3,\\
\end{equation}
and
\begin{equation}\label{eq4.16}
g^{+}_{i,j}(\lambda^2)-g^-_{i,j}(\lambda^2)\in S^{d-2}_{[\frac{d}{2}]+1}((0, \lambda_0)), \,\,\,\quad \mathrm{if}\,\,\, d\ge 3.
\end{equation}
In order to prove these, we note that by Lemma \ref{lm2.3},  \ref{lm2.4} and  \textbf{Fact 1} above, we have
\begin{equation}\label{eq4.18.1}
f^{\pm}_{i,j}(\lambda^2)\in S^{0}_{[\frac{d}{2}]+1}\left((0, \lambda_0)\right),
\end{equation}
and
\begin{equation}\label{eq4.18.2}
f^{+}_{i,j}(\lambda^2)-f^{-}_{i,j}(\lambda^2)\in S^{d-2}_{[\frac{d}{2}]+1}\left((0, \lambda_0)\right).
\end{equation}
Therefore the property \eqref{eq4.18.1}, together with \eqref{eq4.14.2}, \eqref{eq4.14.3} and the assumption \eqref{eq1.13.1}, yields \eqref{eq4.15}.
And
\eqref{eq4.16} follows from \eqref{eq4.18.2} and the following identity
\begin{equation*}
  A^+(\lambda^2)^{-1}-  A^-(\lambda^2)^{-1}=  A^+(\lambda^2)^{-1}\left(A^{-}(\lambda^2)-A^{+}(\lambda^2)\right) A^-(\lambda^2)^{-1}.
\end{equation*}
This proves the low energy part for $d\ge3$.

Combining {\bf Step 1} and {\bf Step 2} with  $\lambda_0=\frac12$, we complete the proof of Theorem \ref{thm1.2} when $d\ge 3$. We are left with the low energy estimates when $d=1, 2$.

{\bf Step 3. Low energy estimates when $d=1$.}
First,  we assume the special case
$$
\int_{\mathbb{R}}{\varphi_i}(x)\,dx=0, \,\,\,\,\,\,\text{for all}\,\,\,1\le i\le N.
$$
Repeat the same arguments as in the proof of \eqref{eq4.15}, we have
 \begin{equation*}\label{eq4.16.11}
g_{ij}^{\pm}(\lambda^2)\in S^{0}_{1}\left((0, \lambda_0)\right).
 \end{equation*}
Now we are able to reduce to the rank one case in $d=1$. Precisely,  the method to treat the low energy part of $U^{\pm}_{1, ij}$ is the same to the estimate of $U_1^{\pm, l}(t, x, y)$ ((when $\int_{\mathbb{R}}{\varphi}(x)\,dx=0$)) in the one dimensional case of  Theorem \ref{thm3.2}, except that we replace $\varphi$ by $\varphi_i$ ($1\le i\le N$).

Next, without loss of generality, assume that there exists some $1\le k_0\le N$ such that
\begin{equation}\label{eq4.16.1}
\int_{\mathbb{R}}{\varphi_i(x)\,dx}\ne 0, \,\,\,\,\text{for}\,\,\, 1\le i\le k_0;\qquad  \int_{\mathbb{R}}{\varphi_i(x)\,dx}=0, \,\,\,\,\text{for}\,\,\, k_0<i\le N.
\end{equation}
The proof for this part is much more delicate since the methods in rank one case can't be applied directly.
In particular, in rank one case, when $\int{\varphi(x)\,dx}\ne 0$, the function $\frac{1}{1+ F^{\pm}(\lambda^2)}$ has better properties due to \eqref{eq.11.1.3}. But  this doesn't seem to be obvious  from the definition of $(g^{\pm}_{ij}(\lambda^2))_{N\times N}$ (see \eqref{eq4.13.1}) in the finite rank case. Instead,
inspired by the works for the Schr\"{o}dinger operator $H=-\Delta+V$ (see \cite{Sch}),
we shall write the inverse $(I+PR_0^{\pm}(\lambda^2)P)^{-1}$ in a suitable form and then exploit the cancelation property from \eqref{eq4.16.1}.
Thus we denote by
\begin{equation}\label{eq4.18.3}
  U_N^{\pm, l}=-\frac{1}{\pi i}\int_0^{\infty}e^{-it\lambda^2}\chi(\lambda)R^{\pm}_{0}(\lambda^2)P(1+PR^{\pm}_{0}(\lambda^2)P)^{-1}PR^{\pm}_{0}(\lambda^2)\cdot \lambda\,d\lambda. \\
\end{equation}
Clearly, in view of \eqref{eq4.13.0}, \eqref{eq4.13.2} and the definition of $U^{\pm, l}_{d, ij}$, we have
\begin{equation*}
 U_N^{\pm, l}=\sum_{i,j}U^{\pm, l}_{d, ij}.
\end{equation*}
To proceed,
we recall the following abstract lemma (see e.g. in \cite[Lemma 2.1]{JN} or \cite{ES,GW15,GW17,Sch}).
\begin{lemma}\label{lm4.1}
Let $\mathscr{A}$ be a closed operator on a Hilbert space $\mathcal{H}$ and $S$ be a projection. Suppose $\mathscr{A}+S$ has a bounded inverse. Then  $\mathscr{A}$ has a bounded inverse if and only if
\begin{equation}\label{eq4.1601}
  \mathscr{B}=S-S(\mathscr{A}+S)^{-1}S
\end{equation}
has a bounded inverse in $S\mathcal{H}$, and, in this case,
\begin{equation}\label{eq4.1602}
 \mathscr{A}^{-1}=(\mathscr{A}+S)^{-1}+(\mathscr{A}+S)^{-1}S\mathscr{B}^{-1}S(\mathscr{A}+S)^{-1}.
\end{equation}
\end{lemma}

In order to apply this lemma, we denote by
$$\mathscr{A}_{\pm}:=I+PR_0^{\pm}(\lambda^2)P, \qquad\text{and} \qquad \mathcal{H}:=PL^2=\text{span}\{\varphi_i\}_{i=1}^N.
$$
Note that by \eqref{eq2.4},
we have the following expansion
\begin{equation}\label{eq4.18}
R^{\pm}_0(\lambda^2)f=\frac{\pm i}{2\lambda}L_1f+L_2f+\mathscr{R}^{\pm}_1(\lambda)f,
\end{equation}
with
\begin{equation*}
  L_1f=\int_{\mathbb{R}}f(x)\,dx,\,\,\,\, L_2f=-\frac{1}{2}\int_{\mathbb{R}}|x-y|f(y)\,dy,\,\,\,\,\mathscr{R}^{\pm}_1(\lambda)f=\int_{\mathbb{R}}r^{\pm}_1(\lambda, |x-y|)f(y)\,dy.
\end{equation*}
Using \eqref{eq4.18} and the cancellation property \eqref{eq4.16.1}, we have
\begin{equation}\label{eq4.19}
\mathscr{A}_{\pm}=\frac{\pm i}{2\lambda}\sum_{i=1}^{k_0}P_i\cdot L_1\cdot\sum_{i=1}^{k_0}P_i+(I+PL_2P)+P\mathscr{R}^{\pm}_1(\lambda)P.
\end{equation}
Let $\mathscr{P}$ be the Riesz projection onto the null space of the operator $\sum_{i=1}^{k_0}P_i\cdot L_1\cdot\sum_{i=1}^{k_0}P_i$ on $\mathcal{H}$, and denote by $S_{\pm}=\frac{\pm i}{2\lambda}\mathscr{P}$. Then $\sum_{i=1}^{k_0}P_i\cdot L_1\cdot\sum_{i=1}^{k_0}P_i+\mathscr{P}$ is invertible, and we define
\begin{equation*}
B_0:=\left(\sum_{i=1}^{k_0}P_i\cdot L_1\cdot\sum_{i=1}^{k_0}P_i+\mathscr{P}\right)^{-1}.
\end{equation*}
Thus in view of \eqref{eq4.19}, we use the  Neumann series expansion to derive
\begin{equation}\label{eq4.20}
 (\mathscr{A}_{\pm}+S_{\pm})^{-1}=\mp2i\lambda B_0+4\lambda^2 B_0 (I+PL_2P) B_0+\mathscr{R}^{\pm}_2(\lambda),\qquad 0<\lambda\ll 1,
\end{equation}
where $\mathscr{R}^{\pm}_2(\lambda)=4\lambda^2B_0P\mathscr{R}^{\pm}_1(\lambda)PB_0+O(\lambda^3)$.
By \eqref{eq2.5} and the decay assumption on $\varphi_i$, it follows that $\mathscr{R}^{\pm}_2(\lambda)$ is  bounded on $\mathcal{H}$ and when $\lambda$ is small enough, one has
\begin{equation}\label{eq4.21}
\|\frac{d^k}{d\lambda^k}\mathscr{R}^{\pm}_2(\lambda)\|_{\mathcal{H}-\mathcal{H}}\leq C\lambda^{\frac52-k},\,\,\,\,\,\,k=0,\,1.
\end{equation}
Next, we shall compute the operator $\mathscr{B}$ defined in \eqref{eq4.1601}. Note that by the definition of  $\mathscr{P}$ and $B_0$, we have the relation
\begin{equation*}
\mathscr{P}B_0=B_0\mathscr{P}=\mathscr{P},\quad\,\,\,\,P\mathscr{P}=\mathscr{P}P=\mathscr{P}.
\end{equation*}
This, together with \eqref{eq4.20}, we have
\begin{align*}\label{eq4.22}
  \mathscr{B}_{\pm}& :=S_{\pm}-S_{\pm}(\mathscr{A}_{\pm}+S_{\pm})^{-1}S_{\pm}  \nonumber \\
& =  \frac{\pm i}{2\lambda}\mathscr{P}+\frac{1}{4\lambda^2}\mathscr{P}\left(\mp2i\lambda B_0+4\lambda^2 B_0 (I+PL_2P) B_0+\mathscr{R}^{\pm}_2(\lambda)\right)\mathscr{P}        \nonumber \\
& =  \mathscr{P}+\mathscr{P}L_2\mathscr{P}+\frac{1}{4\lambda^2}\mathscr{P}\mathscr{R}^{\pm}_2(\lambda)\mathscr{P}.
\end{align*}
We claim that
\begin{equation*}
  (\mathscr{P}+\mathscr{P}L_2\mathscr{P})^{-1}\,\, \text{exists\, on\, }\,\, \mathscr{P}\mathcal{H}.
\end{equation*}
 In fact, this follows that if $(\mathscr{P}+\mathscr{P}L_2\mathscr{P})f=0$, $f\in \mathscr{P}\mathcal{H}$, then by \eqref{eq4.19} we have
 \begin{align*}
  0=\langle(\mathscr{P}+\mathscr{P}L_2\mathscr{P})f, f\rangle=\lim_{\lambda\rightarrow 0^+}\langle(\mathscr{P}+\mathscr{P}R^{\pm}_0(\lambda^2)\mathscr{P})f, f\rangle
  =\langle(I+(-\Delta)^{-1})f, f\rangle,\qquad\,\,\,,
 \end{align*}
which implies that $f\equiv0$.

Therefore, using  the  Neumann series expansion and \eqref{eq4.21} we obtain
\begin{equation}\label{eq4.23}
 \mathscr{B}_{\pm}^{-1}=D_0+\mathscr{R}^{\pm}_3(\lambda),\qquad 0<\lambda\ll 1,
\end{equation}
where $\mathscr{R}^{\pm}_3=\frac{-1}{4\lambda^2}D_0\mathscr{P}\mathscr{R}^{\pm}_2(\lambda)\mathscr{P}D_0+O(\lambda)$, and $D_0:=(\mathscr{P}+\mathscr{P}L_2\mathscr{P})^{-1}$ is bounded on $\mathscr{P}\mathcal{H}$. By  \eqref{eq4.21}, one has for sufficiently small $\lambda>0$,
\begin{equation}\label{eq4.23.1}
\|\frac{d^k}{d\lambda^k}\mathscr{R}^{\pm}_3(\lambda)\|_{\mathscr{P}\mathcal{H}-\mathscr{P}\mathcal{H}}\leq C\lambda^{\frac12-k},\,\,\,\,\,\,k=0,\,1.
\end{equation}

Now we are in the position to write the expansion of $(I+PR_0^{\pm}(\lambda^2)P)^{-1}$ in $\mathcal{H}$. By \eqref{eq4.20}, \eqref{eq4.23} and \eqref{eq4.1602} in Lemma \ref{lm4.1}, one can choose some small $\lambda_0>0$ such that
\begin{equation}\label{eq4.24}
 (I+PR_0^{\pm}(\lambda^2)P)^{-1}=\mathscr{P}D_0\mathscr{P}+\mathscr{P}\mathscr{R}^{\pm}_3\mathscr{P}+\mathscr{R}^{\pm}_4(\lambda), \quad 0<\lambda<\lambda_0,
\end{equation}
where
$$
\mathscr{R}^{\pm}_4(\lambda)=\mp2i\lambda B_0\pm2i\lambda\mathscr{P}D_0\mathscr{P}(I+PL_2P) B_0\pm2i\lambda B_0(I+PL_2P)\mathscr{P}D_0\mathscr{P}+O(\lambda^{\frac32}).
 $$
It follows from \eqref{eq4.21} and \eqref{eq4.23.1} that
\begin{equation}\label{eq4.25}
  \|\frac{d^k}{d\lambda^k}\mathscr{R}^{\pm}_4(\lambda)\|_{\mathcal{H}-\mathcal{H}}\leq C\lambda^{1-k},\,\,\quad 0<\lambda<\lambda_0.
\end{equation}
Using \eqref{eq4.24} and  \eqref{eq4.25}, we are ready to reduce low-energy dispersive estimates to the rank one case. Indeed, let $\lambda_0$ be given in \eqref{eq4.24}, we rewrite $U_N^{\pm, l}$ (see \eqref{eq4.18.3}) as
\begin{align}\label{eq4.26}
   U_N^{\pm, l}&=\frac{-1}{\pi i}\int_0^{\infty}e^{-it\lambda^2}R^{\pm}_0(\lambda^2)\mathscr{P}\left( D_0+\mathscr{R}^{\pm}_3(\lambda)\right)\mathscr{P}R^{\pm}_0(\lambda^2)\lambda\cdot\chi(\lambda) d\lambda\nonumber\\
 &- \frac{1}{\pi i}\int_0^{\infty}e^{-it\lambda^2}R^{\pm}_0(\lambda^2)P\mathscr{R}^{\pm}_4(\lambda)PR^{\pm}_0(\lambda^2)\lambda\cdot\chi(\lambda) d\lambda\nonumber\\
 &:=  U_{N,1}^{\pm, l}+  U_{N,2}^{\pm, l}.
\end{align}

First, we deal with $U_{N,2}^{\pm, l}$. Note that from the explicit expression \eqref{eq2.1} the integral kernel of $U_{N,2}^{\pm, l}$ can be written as
\begin{equation*}
U_{N,2}^{\pm, l}(t, x, y)=\sum_{i,j=1}^{N}\int_{\mathbb{R}}\int_{\mathbb{R}}{I_{N,2,i,j}^{\pm, l}(t, |x-x_1|, |x_2-y|)\varphi_i(x_1)\varphi_j(x_2)\,dx_1dx_2},
\end{equation*}
where
\begin{equation*}
  I_{N,2,i,j}^{\pm, l}(t, |x-x_1|, |x_2-y|):=\frac{1}{4\pi i}\int_0^{\infty}{e^{-it\lambda^2\pm i\lambda(|x-x_1|+ |x_2-y|)}\frac{\langle \mathscr{R}^{\pm}_4(\lambda)\varphi_j, \varphi_i\rangle}{\lambda}\chi(\lambda)d\lambda}.
\end{equation*}
Since $\mathscr{R}_4(\lambda)$ satisfies \eqref{eq4.25}, thus we have
\begin{equation*}
 \frac{\langle \mathscr{R}^{\pm}_4(\lambda)\varphi_j, \varphi_i\rangle}{\lambda}\in S^0_1\left((0, \lambda_0)\right).
\end{equation*}
Hence by (\romannumeral1) of Corollary \ref{cor2.6}, we have
\begin{equation*}
 |I_{N,2,i,j}^{\pm, l}(t, |x-x_1|, |x_2-y|)|\leq C_ N\cdot t^{-\frac12}.
\end{equation*}
Therefore, it follows from the decay assumption of $\varphi_i$  that
\begin{equation}\label{eq4.26.1}
\sup_{x, y\in\mathbb{R}}|U_{N,2}^{\pm, l}(t, x, y)|\leq C(N, d, \varphi_1,\ldots, \varphi_N)\cdot  t^{-\frac12}.
\end{equation}

Next, we deal with the term $U_{N,1}^{\pm, l}$, in which  we shall exploit the orthogonality from the definition of $\mathscr{P}$. To be more precise, we write
\begin{align*}
e^{\mp i\lambda|x|} R^{\pm}_0(\lambda^2, x, x_1)&=\frac{\pm i}{2\lambda}e^{\pm i\lambda(|x-x_1|-|x|)} \\
&=\frac{\pm i}{2\lambda}-\frac{|x-x_1|-|x|}{2\lambda}\int_0^{\lambda}e^{\pm is\lambda(|x-x_1|-|x|)}\,ds.
\end{align*}
Insert this formula into the expression of $U_{N,1}^{\pm, l}$ in \eqref{eq4.26} and notice that
  $L_1\mathscr{P}=0$, then the integral kernel of  $U_{N,1}^{\pm, l}$  can be rewritten as
\begin{equation*}
U_{N,1}^{\pm, l}(t, x, y)=\sum_{i,j=1}^{N}\int_{\mathbb{R}}\int_{\mathbb{R}}{I_{N,1,i,j}^{\pm, l}(t, |x-x_1|, |x_2-y|)(|x-x_1|-|x|)\varphi_i(x_1)\varphi_j(x_2)\,dx_1dx_2},
\end{equation*}
where
\begin{equation*}
  I_{N,1,i,j}^{\pm, l}(t, |x-x_1|, |x_2-y|):=\frac{\pm 1}{4\pi}\int_0^{\infty}{e^{-it\lambda^2\pm i\lambda(|x|+ |x_2-y|)}\langle \mathscr{P}(D_0+\mathscr{R}^{\pm}_3)\mathscr{P}\varphi_j, \varphi_i\rangle g(\lambda)\chi(\lambda) d\lambda},
\end{equation*}
and
$$
 g(\lambda):=\frac{1}{\lambda}\int_0^{\lambda}e^{\pm is\lambda(|x-x_1|-|x|)}\,ds.
$$
A direct computation implies that
\begin{equation*}
  g(\lambda)\in S^0_1\left((0, \lambda_0)\right).
\end{equation*}
This, together with \eqref{eq4.23.1} and our decay assumption on $\varphi_j$ for $j=1,\cdots,N$, yields that
$$\langle \mathscr{P}(D_0+\mathscr{R}^{\pm}_3)\mathscr{P}\varphi_j, \varphi_i\rangle g(\lambda)\chi(\lambda)\in S^0_1\left((0, \lambda_0)\right)$$
Now we use (\romannumeral1) Corollary \ref{cor2.6} with $b=0, K=1$ to obtain
\begin{equation*}
 |I_{N,1}^{\pm, l}(t, |x-x_1|, |x_2-y|)|\leq C(N, d, \varphi_1,\ldots, \varphi_N)\cdot  t^{-\frac12}.
\end{equation*}
 Combining this and the decay assumption on $\varphi_j$, as well as the fact $||x-x_1|-|x||\leq |x_1|$, one has
\begin{equation}\label{eq4.27}
	\sup_{x, y\in\mathbb{R}}|U_{N,1}^{\pm, l}(t, x, y)|\leq C(N, d, \varphi_1,\ldots, \varphi_N)\cdot  t^{-\frac12}.
\end{equation}
Finally  by \eqref{eq4.26.1} and  \eqref{eq4.27}  we have
$$
\|U_N^{\pm, l}\|_{L^1-L^{\infty}}\leq C(N, d, \varphi_1,\ldots, \varphi_N)\cdot  t^{-\frac12}.
$$

{\bf Step 4. Low energy estimates when $d=2$.} The proof is almost the same as that of $d=1$, and we only give a sketch.
When
$$
\int_{\mathbb{R}^2}{\varphi_i}(x)\,dx=0, \,\,\,\,\,\,\text{for all}\,\,\,1\le i\le N.
$$
Similar to the proof of \eqref{eq4.15}, we have
 \begin{equation*}\label{eq4.16.11}
g_{ij}^{\pm}(\lambda^2)\in S^{0}_{2}\left((0, \lambda_0)\right).
 \end{equation*}
Thus the method to treat the low energy part of $U^{\pm}_{2, ij}$ is the same as the estimate of $U_2^{\pm, l}(t, x, y)$ in subcase 2  ($d=2$ and $\int_{\mathbb{R}^2}{\varphi}(x)\,dx=0$) of  Theorem \ref{thm3.2}, except that we replace $\varphi$ by $\varphi_i$ ($1\le i\le N$).

Next, assume without loss of generality that there exists some $1\le k_0\le N$ such that
\begin{equation}\label{eq4.16.15}
\int_{\mathbb{R}^2}{\varphi_i(x)\,dx}\ne 0, \,\,\,\,\text{for}\,\,\, 1\le i\le k_0;\qquad  \int_{\mathbb{R}^2}{\varphi_i(x)\,dx}=0, \,\,\,\,\text{for}\,\,\, k_0<i\le N.
\end{equation}
We follow the same strategy in {\bf Step 3}. By the resolvent expansion \eqref{eq2.6}, we have
\begin{equation}\label{eq4.32}
 (I+PR_0^{\pm}(\lambda^2)P)^{-1}=\mathscr{P}D_0\mathscr{P}+\frac{1}{\log{\lambda}}D^{\pm}_1+\mathscr{R}^{\pm}_4(\lambda), \qquad 0<\lambda<\lambda_0,
\end{equation}
where $\lambda_0$ is some small but fixed constant, $\mathscr{P}$ is the same as in  {\bf Step 3} and
\begin{align*}
  D_0:=( \mathscr{P}+\mathscr{P}L_2\mathscr{P})^{-1},\,\,\,\,\,L_2f:=\frac{-1}{2\pi}\int{log{|x-y|}f(y)\,dy}.
\end{align*}
Moreover, $\mathscr{P}D_0\mathscr{P}$, $D^{\pm}_1$, $\mathscr{R}^{\pm}_4(\lambda)$ are bounded operators on $\mathcal{H}$, and
\begin{equation}\label{eq4.32.1}
 \|\frac{d^k}{d\lambda^k}\mathscr{R}^{\pm}_4(\lambda)\|_{\mathcal{H}-\mathcal{H}}\leq
 	C \frac{\lambda^{-k}}{|\log{\lambda}|^2},\,\,\quad 0<\lambda<\lambda_0,\,\,\,0\le k\le 2.
\end{equation}
Similar to the low energy decomposition \eqref{eq.l7}, using \eqref{eq4.18.3} we have
\begin{align}\label{eq4.32.2}
U_N^{+, l}-U_N^{-, l}&=\frac{-1}{\pi i}\int_0^{\infty}e^{-it\lambda^2}R^{-}_{0}(\lambda^2)P\left((1+PR^{+}_{0}(\lambda^2)P)^{-1}-(1+PR^{-}_{0}(\lambda^2)P)^{-1}\right)PR^{+}_{0}(\lambda^2) \lambda\chi(\lambda)\,d\lambda \nonumber\\
&-\frac{1}{\pi i}\int_0^{\infty}e^{-it\lambda^2}\left(R^{+}_{0}(\lambda^2)-R^{-}_{0}(\lambda^2)\right)P(1+PR^{+}_{0}(\lambda^2)P)^{-1}PR^{+}_{0}(\lambda^2)\cdot \lambda\chi(\lambda)\,d\lambda \nonumber\\
&-\frac{1}{\pi i}\int_0^{\infty}e^{-it\lambda^2}R^{-}_{0}(\lambda^2)P(1+PR^{+}_{0}(\lambda^2)P)^{-1}P\left(R^{+}_{0}(\lambda^2)-R^{-}_{0}(\lambda^2)\right)\cdot \lambda\chi(\lambda)\,d\lambda \nonumber\\
&:=U_{N, 1}^{ l}+U_{N, 2}^{ l}+U_{N, 3}^{ l}.
\end{align}
Based on the expression \eqref{eq4.32} and the definition of $\mathscr{P}$, we are able to reduce the analysis of $U_{N, j}^{ l}$ ($1\le j\le 3$) to the rank one case with $\int_{\mathbb{R}^2}{\varphi}(x)\,dx=0$. We remark that the analysis of $U_{N, 1}^{ l}$ is the same as \eqref{eq3.54.5}--\eqref{eq3.57.2} and the analysis of $U_{N, 2}^{ l},U_{N, 3}^{ l}$ is the same as \eqref{eq3.54.1}-\eqref{eq3.57.1}. We only take the proof of $U_{N, 2}^{ l}$ as an example.

Instead of decomposing  $R_0^{\pm}(\lambda^2)\varphi$ as in \eqref{eq3.3.69},
using \eqref{eq4.32} and it's not hard to see that the kernel of $P(I+PR^{+}_{0}(\lambda^2)P)^{-1}PR^{+}_{0}(\lambda^2)$ can be written as
 \begin{align}\label{eq4.40.1}
&\left(P(I+PR^{+}_{0}(\lambda^2)P)^{-1}P(e^{\pm i\lambda|\cdot-y|}w_{\pm, >}(\lambda|\cdot-y|)+w_{\pm, <}(\lambda|\cdot-y|))\right)(x_1)\nonumber\\
&=\left(P(I+PR^{+}_{0}(\lambda^2)P)^{-1}Pe^{\pm i\lambda|\cdot-y|}w_{\pm, >}(\lambda|\cdot-y|)\right)(x_1)\nonumber\\
&+\left(\left(\mathscr{P}D_0\mathscr{P}+\frac{1}{\log{\lambda}}PD^{\pm}_1P+P\mathscr{R}^{\pm}_4(\lambda)P\right)w_{\pm, <}(\lambda|\cdot-y|)\right)(x_1)\nonumber\\
&=\sum_{i,j}^Ng_{ij}^{+}(\lambda^2)\varphi_i(x_1)\left\langle\, e^{\pm i\lambda|\cdot-y|}w_{\pm, >}(\lambda|\cdot-y|),\varphi_j(\cdot)\right\rangle\nonumber\\
&+\sum_{i,j}^N\left\langle\mathscr{P}D_0\mathscr{P}\varphi_j,\,\varphi_i\right\rangle\varphi_i(x_1)\left\langle w'_{\pm, <}(\lambda, y, \cdot),\, \varphi_j(\cdot)\right\rangle\nonumber\\
&+\sum_{i,j}^N\left\langle(\frac{1}{\log{\lambda}}D^{+}_1+\mathscr{R}^{+}_4(\lambda))\varphi_j,\,\varphi_i\right\rangle\varphi_i(x_1)\left\langle w_{\pm, <}(\lambda|\cdot-y|),\, \varphi_j(\cdot)\right\rangle\nonumber\\
&:=\uppercase\expandafter{\romannumeral1}+\uppercase\expandafter{\romannumeral2}+\uppercase\expandafter{\romannumeral3},
 \end{align}
where $w'_{\pm, <}(\lambda, y, x_2)$ is the same function defined in \eqref{eq3.3.70}. Indeed,  in the first equality above,  we use the expression of $R^{+}_{0}(\lambda^2)$; in the second equality, we take advantage of the orthogonality by using \eqref{eq4.32} and the fact that $\mathscr{P}L_1=0$, where $L_1f=\int_{\mathbb{R}^2}{f(x)\,dx}$ is the same as in Step 3.

Plugging \eqref{eq4.40.1} into $U_{N, 2}^{ l}$,  then the analysis of $U_{N, 2}^{ l}$ is very similar to the estimate of \eqref{eq3.54.1}-\eqref{eq3.57.1} in the rank one case, $d=2$.

Precisely, when dealing with the term \uppercase\expandafter{\romannumeral1}, we observe that the expressions coincide with  \eqref{eq3.54.1} and \eqref{eq3.56.1}, except that $\alpha F^{\alpha, l}_{-}$ is replaced by $g_{ij}^{+}(\lambda^2)\cdot\chi(\lambda)$  and $\varphi(x_1)$ ($\varphi(x_2)$) replaced by $\varphi_i(x_1)$ ($\varphi_j(x_2)$). Note that by \textbf{Fact 1},  both $\alpha F^{\alpha, l}_{-}$ and $g_{ij}^{+}(\lambda^2)\cdot\chi(\lambda)$ belong to $S_{2}^{0}\left( (0,\lambda_0)\right)$, then they satisfy the same estimates as \eqref{eq3.54.1} and \eqref{eq3.56.1}.

When dealing with the term $\uppercase\expandafter{\romannumeral2}$, we observe that the  expressions  coincide with \eqref{eq3.55.1} and  \eqref{eq3.57.1},  except that $\alpha F^{\alpha, l}_{-}w_{+, <}(\lambda|x_2-y|)$ is replaced by
   $$
\langle\mathscr{P}D_0\mathscr{P}\varphi_j,\,\varphi_i\rangle w'_{+, <}(\lambda, y, x_2)
 $$
 and $\varphi(x_1)$ ($\varphi(x_2)$) replaced by $\varphi_i(x_1)$ ($\varphi_j(x_2)$). Then by properties of $w'_{+, <}$ (see \eqref{eq3.67}),
\eqref{eq3.90.4}  and \eqref{eq3.90.5} remain valid when $\alpha F^{\alpha, l}_{-}$ is replaced by $\langle\mathscr{P}D_0\mathscr{P}\varphi_j,\,\varphi_i\rangle\cdot\chi(\lambda) $
 and  $(1+|x_2-y|^{-\frac12})^{-1}w_{+, <}(\lambda|x_2-y|)$	 is replaced by
$$\left(1+\langle x_2\rangle+\frac{\langle x_2\rangle}{|x_2-y|}\right)^{-\frac12}w'_{\pm, <}(\lambda, x_2, y).$$

When dealing with the term $\uppercase\expandafter{\romannumeral3}$, we observe that the  expressions  coincide with \eqref{eq3.55.1} and  \eqref{eq3.57.1},  except that $\alpha\cdot F^{\alpha, l}_{-}$ is replaced by
$\langle(\frac{1}{\log{\lambda}}D^{+}_1+\mathscr{R}^{+}_4(\lambda))\varphi_j,\,\varphi_i\rangle$
 and $\varphi(x_1)$ ($\varphi(x_2)$) replaced by $\varphi_i(x_1)$ ($\varphi_j(x_2)$). In this case,  by  \eqref{eq4.32.1}, we see that
 \eqref{eq3.90.4}  and \eqref{eq3.90.5} remain valid when $\alpha\cdot F^{\alpha, l}_{-}$ is replaced by
$\langle(\frac{1}{\log{\lambda}}D^{+}_1+\mathscr{R}^{+}_4(\lambda))\varphi_j,\,\varphi_i\rangle$.
 Therefore  they satisfy the same estimates as  \eqref{eq3.55.1} and  \eqref{eq3.57.1}.

 Combining {\bf Step 1} and {\bf Step 3} ({\bf Step 4}) with $\lambda_0$ given in \eqref{eq4.24} (\eqref{eq4.32}), we complete the proof of Theorem \ref{thm1.2} when $d=1$ ($d=2$).
 \qed

\begin{remark}\label{rmk3.5}
We point out that in the proof of Theorem \ref{thm1.2}, the constant $C(N, d, \varphi_1,\ldots, \varphi_N)$ in \eqref{eq1.13} may be extremely large with respect to $N$ (about $C^N\cdot N!$) if we use \eqref{eq4.14.2} to compute the inverse of the matrix $A_{ij}^{\pm}(\lambda^2)$. Indeed, note that each cofactor of $A_{ij}^{\pm}(\lambda^2)$ is the sum of  $(N-1)!$ terms, then
by \eqref{eq4.14.1} and \eqref{eq4.14.2}, it follows that
$$
\left| g_{ij}^{\pm}(\lambda^2)\right| \leq  (N-1)!\cdot (C\lambda^{-1})^{N-1},\qquad \text{for $\lambda>\lambda_0$ }.
$$
Furthermore, by using  \eqref{eq4.14.3}, one has similar estimates for $\frac{d^k}{d\lambda^k} g_{ij}^{\pm}(\lambda^2)$.
These facts show that the constant in \eqref{eq4.14} grows at least $C^{N-1} \cdot(N-1)!$, which in turn implies the same conclusion for the constant $C(N, d, \varphi_1,\ldots, \varphi_N)$ in \eqref{eq1.13}.

\end{remark}

\subsection{Proof of Theorem \ref{thm1.4}}\label{sec4.3}
 First of all, observe that by Plancherel's theorem and the assumption \eqref{eq0.19}, i.e., $\text{supp} \, \hat{\varphi_i}\cap \text{supp}\, \hat{\varphi_j}=\emptyset$, when  $i\ne j$,  one has
\begin{align*}
\langle R_0^{\pm}(\lambda^2)\varphi_j,\varphi_i\rangle=\lim_{\varepsilon\rightarrow 0}\int_{\mathbb{R}^d}{\frac{\hat{\varphi_j}(\xi)\cdot \bar{\hat{\varphi_i}}(\xi)}{|\xi|^2-\lambda^2\mp i\varepsilon}\,d\xi}
=0,\,\,\,\,\quad \text{if}\,\,\,\,i\ne j.
\end{align*}
This means that the matrix $I+F_{N\times N}^{\pm}(\lambda^2)$ is diagonal, i.e.,
\begin{equation}\label{eq4.32.3}
  I+F_{N\times N}^{\pm}(\lambda^2)=\begin{pmatrix}
        1+\langle R_0^{\pm}(\lambda^2)\varphi_1,\varphi_1\rangle&   & \\
        & & \ddots & \\
         &  &  & 1+\langle R_0^{\pm}(\lambda^2)\varphi_N,\varphi_N\rangle
    \end{pmatrix}.
\end{equation}
Second, since each $\varphi_i$ ($1\le i\le N$) is also assumed to satisfy the spectral assumption \eqref{eq1.4}, then it follows immediately that $I+F_{N\times N}^{\pm}(\lambda^2)$ is invertible, in particular, \eqref{eq1.13.1} is valid and the spectrum of $H$  is purely absolutely continuous (see Lemma \ref{lmA.1}).  Furthermore, recall that  $(g^{\pm}_{ij})_{N\times N}=(I+F_{N\times N}^{\pm})^{-1}$, and we have
\begin{equation}\label{eq4.34.1}
g^{\pm}_{ij}(\lambda^2)=
\begin{cases}
\frac{1}{1+\langle R_0^{\pm}(\lambda^2)\varphi_i,\varphi_i\rangle},\qquad \qquad\text{if}\,\,\,i=j,\\
 0,\qquad \qquad\qquad \quad\,\,\,\quad\text{if}\,\,\,i\ne j.
 \end{cases}
\end{equation}
Put this into \eqref{eq4.13.2}, we thus obtain that
\begin{align}\label{eq4.32.4}
e^{-itH}-e^{-itH_{0}}&=-\frac{1}{\pi i}\sum_{i=1}^N\int_0^{\infty}e^{-it\lambda^2}g^+_{ii}(\lambda^2)R^{+}_{0}(\lambda^2)\varphi_i\langle R^{+}_{0}(\lambda^2)\, \cdot, \varphi_i\rangle \lambda\,d\lambda\nonumber\\
&+\frac{1}{\pi i}\sum_{i=1}^N\int_0^{\infty}e^{-it\lambda^2}g^-_{ii}(\lambda^2)R^{-}_{0}(\lambda^2)\varphi_i\langle R^{-}_{0}(\lambda^2)\, \cdot, \varphi_i\rangle \lambda\,d\lambda.
\end{align}
Finally, note that by \eqref{eq4.34.1}, $g^{\pm}_{ii}(\lambda^2)$ coincides with $\frac{1}{1+ F^{\pm}(\lambda^2)}$
if we replace $\varphi_i$ by  $\varphi$.
Since  $\varphi_i$ and $\varphi$ satisfy the same assumptions, the integrals in \eqref{eq4.32.4} are the same as that in  \eqref{eq3.3}  with $\alpha=1$. Therefore the proof is complete by applying \eqref{eq1.5} in  Theorem \ref{thm-1.1}.  \qed

\subsection{Proof of Theorem \ref{thm1.5}}\label{sec4.4}
We  first establish the following Lemma, from which we know more about the structure of the matrix $I+F_{N\times N}^{\pm}(\lambda^2)$. This additional information plays a crucial role in our proof.

\begin{lemma}\label{lm4.2}
Assume that $d\ge 3$, $\lambda>0$, and let $\varphi_1,\ldots, \varphi_N$ be given by Theorem \ref{thm1.5}. Then there exists some $C_1=C(d, \varphi)>0$ such that
\begin{equation}\label{eq4.33}
|\langle R_0^{\pm}(\lambda^2)\varphi_j, \varphi_i\rangle|\leq
C_1\langle\tau_0\rangle^{-\frac{d-1}{2}},\,\,\,\quad \mathrm{for}\,\,\,i\ne j.\\
\end{equation}
\end{lemma}
\begin{proof}
 First of all, we show that the claim follows immediately when $d=3$. Indeed,  by \eqref{eq2.3}, we have
\begin{align}
|\langle R_0^{\pm}(\lambda^2)\varphi_j, \varphi_i\rangle|&=\left|\int_{\mathbb{R}^3}\int_{\mathbb{R}^3}\frac{e^{\pm i\lambda|x-y|}}{4\pi|x-y|}\varphi_j(y)\varphi_i(x)\,dxdy\right| \label{eq4.33.2} \\
&\leq C\int_{\mathbb{R}^3}\int_{\mathbb{R}^3}\langle x-(\tau_i-\tau_j)\rangle^{-\delta}|x-y|^{-1}\langle y\rangle^{-\delta}\,dxdy \nonumber  \\
&\leq C\langle \tau_0\rangle^{-1},\label{eq4.33.1}
\end{align}
where in the first inequality, we make the change of variables $x-\tau_j=x', y-\tau_j=y'$ and use the decay assumption  of $\varphi$, the second inequality follows from Lemma \ref{lm2.8} and the assumption \eqref{eq1.16.1}, i.e., $|\tau_i-\tau_j|\ge \tau_0$.

Next, when $d>3$ is odd, we use \eqref{eq2.3} to obtain
\begin{align}\label{eq4.34}
\langle R_0^{\pm}(\lambda^2)\varphi_j, \varphi_i\rangle=\sum_{k=0}^{\frac{d-3}{2}}C_k(\mp2i\lambda)^k
\cdot\int_{\mathbb{R}^d}\int_{\mathbb{R}^d}\frac{e^{\pm i\lambda|x-y|}}{|x-y|^{d-2-k}}\varphi_j(y)\varphi_i(x)\,dxddy.
\end{align}
For each term in the sum \eqref{eq4.34}, we shall use $L_{x_i}$ and $L^*_{x_i}$   defined in \eqref{eq3.16.1} and \eqref{eq3.16.2} to perform integration by parts arguments. Precisely, for any fixed $k=0, 1,\ldots, \frac{d-3}{2}$, we have
\begin{align}\label{eq4.35}
&\lambda^k \int_{\mathbb{R}^d}\int_{\mathbb{R}^d}\frac{e^{\pm i\lambda|x-y|}}{|x-y|^{d-2-k}}\varphi_j(y)\varphi_i(x)\,dxdy\nonumber\\
&=\lambda^k \int_{\mathbb{R}^d}\int_{\mathbb{R}^d}\frac{e^{\pm i\lambda|x-y|}}{|x-y|^{d-2-k}}\varphi(y)\varphi(x-(\tau_i-\tau_j))\,dxdy\nonumber\\
&=\lambda^k \int_{\mathbb{R}^d}\int_{\mathbb{R}^d}e^{\pm i\lambda|x-y|}(L^*_{x})^{n_1}(L^*_{y})^{n_2}\left[\frac{\varphi(y)\varphi(x-(\tau_i-\tau_j))}{|x-y|^{d-2-k}}\right]\,dxdy.
\end{align}
We choose some fixed $n_1, n_2\in \mathbb{N}_0$ such that
\begin{equation}\label{eq4.35.1}
n_1+n_2=k,\qquad\,\, 0\le n_1, n_2\le [\frac{d-3}{4}]+1.
\end{equation}
Similar to the proof of \eqref{eq3.24} and \eqref{eq3.25} in  Theorem \ref{thm3.1} for the case $d>3$, we have
\begin{align}\label{eq4.35.2}
G_k(x, y, \tau_i, \tau_j, \lambda)&:= (L^*_{x})^{n_1}(L^*_{y})^{n_2}\left[\frac{\varphi(y)\varphi(x-(\tau_i-\tau_j))}{|x-y|^{d-2-k}}\right]\nonumber\\
&\leq  C\lambda^{-k}\langle x-(\tau_i-\tau_j)\rangle^{-\delta}\langle y\rangle^{-\delta}\left(\frac{1}{|x-y|^{d-2-k}}+\frac{1}{|x-y|^{d-2}}\right).
\end{align}
 By the definition of $L_{x_i}$ and $L^*_{x_i}$, we see that $\lambda^{k}G_k(x, y, \tau_i, \tau_j, \lambda)$ doesn't depend on $\lambda$. Denote by
\begin{equation}\label{eq4.35.3}
  G^{\pm}(x, y, \tau_i, \tau_j):=\sum_{k=0}^{\frac{d-3}{2}}C_k(\mp2i\lambda)^k G_k(x, y, \tau_i, \tau_j, \lambda).
\end{equation}
It follows from \eqref{eq4.35.2} that
\begin{equation}\label{eq4.35.4}
 |G^{\pm}(x, y, \tau_i, \tau_j)|\leq C\langle x-(\tau_i-\tau_j)\rangle^{-\delta}\langle y\rangle^{-\delta}\left(\frac{1}{|x-y|^{\frac{d-1}{2}}}+\frac{1}{|x-y|^{d-2}}\right).
\end{equation}
 Therefore, it follows from \eqref{eq4.34}, \eqref{eq4.35.3}, \eqref{eq4.35.4} and Lemma \ref{lm2.8} that
\begin{align}
|\langle R_0^{\pm}(\lambda^2)\varphi_j, \varphi_i\rangle|&=|\int_{\mathbb{R}^{2d}}{e^{\pm i\lambda|x-y|}G^{\pm}(x, y, \tau_i, \tau_j)\,dxdy}|\label{eq4.36}\\
&\leq C\int_{\mathbb{R}^d}\int_{\mathbb{R}^d}\langle x-(\tau_i-\tau_j)\rangle^{-\delta}\langle y\rangle^{-\delta}\left(\frac{1}{|x-y|^{\frac{d-1}{2}}}+\frac{1}{|x-y|^{d-2}}\right) \,dxdy\nonumber\\
&\leq C\langle\tau_0\rangle^{-\frac{d-1}{2}},\label{eq4.36.1}
\end{align}
which finishes the proof of the claim for $d>3$, odd.

Third, when $d>3$ is even, we fix a smooth cut-off function $\chi(\lambda)$ such that
$\chi(\lambda)=1$ if $\lambda<\frac12$ and $\chi(\lambda)=0$ if $\lambda>1$.
Then by \eqref{eq3.50}, we write
\begin{align}\label{eq4.36.2}
\langle R_0^{\pm}(\lambda^2)\varphi_j, \varphi_i\rangle & =\lambda^{\frac{d}{2}-1}\int_{\mathbb{R}^{2d}} {e^{\pm i\lambda|x-y|}(1-\chi(\lambda))\frac{w_{\pm, >}(\lambda|x-y|)}{|x-y|^{\frac{d}{2}-1}}\varphi_j(y)\varphi_i(x)\,dxdy}\nonumber\\
& +\lambda^{\frac{d}{2}-1}\int_{\mathbb{R}^{2d}} {e^{\pm i\lambda|x-y|}\chi(\lambda)\frac{w_{\pm, >}(\lambda|x-y|)}{|x-y|^{\frac{d}{2}-1}}\varphi_j(y)\varphi_i(x)\,dxdy}\nonumber\\
& +\int_{\mathbb{R}^{2d}}{\frac{w_{\pm, <}(\lambda|x-y|)}{|x-y|^{d-2}}\varphi_j(y)\varphi_i(x)\,dxdy}\nonumber\\
&:=\uppercase\expandafter{\romannumeral1}+\uppercase\expandafter{\romannumeral2}+\uppercase\expandafter{\romannumeral3}.
\end{align}
For the term \uppercase\expandafter{\romannumeral1}, we use the same approach as in odd dimensions by applying $L_{x_i}$ and $L^*_{x_i}$. Then we have
\begin{equation}\label{eq4.36.3}
 \uppercase\expandafter{\romannumeral1}=\lambda^{\frac{d}{2}-1}\int_{\mathbb{R}^{2d}} {e^{\pm i\lambda|x-y|}(1-\chi(\lambda))(L^*_{x})^{n_1}(L^*_{y})^{n_2}\left[\frac{w_{\pm, >}(\lambda|x-y|)}{|x-y|^{\frac{d}{2}-1}}\varphi_j(y)\varphi_i(x)\right]\,dxdy},
\end{equation}
where $0\le n_1, n_2\le [\frac{d-3}{4}]+1$ are fixed integers satisfying $n_1+n_2=\frac{d}{2}-1$. It follows from properties of $w_{\pm, >}$ (see \eqref{eq3.51}) and definition of $\varphi_i$ that
\begin{align}\label{eq4.36.4}
 & \left|\frac{d^k}{d\lambda^k}\left\{\lambda^{\frac{d}{2}-1}(1-\chi(\lambda))(L^*_{x})^{n_1}(L^*_{y})^{n_2}\left[\frac{w_{\pm, >}(\lambda|x-y|)}{|x-y|^{\frac{d}{2}-1}}\varphi_j(y)\varphi_i(x)\right]\right\}\right|\nonumber\\
&  \leq C\lambda^{-k}\langle x-\tau_i\rangle^{-\delta}\langle y-\tau_j\rangle^{-\delta}\left(\frac{1}{|x-y|^{\frac{d-1}{2}}}+\frac{1}{|x-y|^{d-2}}\right)
\end{align}
(the proof of \eqref{eq4.36.4} follows from the same arguments in proving \eqref{eq3.24.1} and \eqref{eq3.25.1}).
For  the term  \uppercase\expandafter{\romannumeral2}, by \eqref{eq3.51}, a direct computation yields that
\begin{equation}\label{eq4.36.5}
\left|\frac{d^k}{d\lambda^k}\left[\lambda^{\frac{d}{2}-1}\chi(\lambda)\frac{w_{\pm, >}(\lambda|x-y|)}{|x-y|^{\frac{d}{2}-1}}\varphi_j(y)\varphi_i(x)\right]\right|\leq C\lambda^{-k}\langle x-\tau_i\rangle^{-\delta}\langle y-\tau_j\rangle^{-\delta}\frac{1}{|x-y|^{\frac{d-1}{2}}}.
\end{equation}
Denote by
\begin{align}\label{eq4.36.6}
  \tilde{G}^{\pm}(x, y, \tau_i, \tau_j, \lambda)&:=\lambda^{\frac{d}{2}-1}(1-\chi(\lambda))(L^*_{x})^{n_1}(L^*_{y})^{n_2}\left[\frac{w_{\pm, >}(\lambda|x-y|)}{|x-y|^{\frac{d}{2}-1}}\varphi_j(y)\varphi_i(x)\right]\nonumber\\
  &+\left[\lambda^{\frac{d}{2}-1}\chi(\lambda)\frac{w_{\pm, >}(\lambda|x-y|)}{|x-y|^{\frac{d}{2}-1}}\varphi_j(y)\varphi_i(x)\right].
\end{align}
Similar to \eqref{eq4.36} we also have
\begin{align}
&|\langle R_0^{\pm}(\lambda^2)\varphi_j, \varphi_i\rangle|\nonumber\\
&=\left|\int_{\mathbb{R}^{2d}}{e^{\pm i\lambda|x-y|}\tilde{G}^{\pm}(x, y, \tau_i, \tau_j, \lambda)\,dxdy}+\int_{\mathbb{R}^{2d}}{\frac{w_{\pm, <}(\lambda|x-y|)}{|x-y|^{d-2}}\varphi_j(y)\varphi_i(x)\,dxdy}\right|\label{eq4.36.7}\\
&\leq C\int_{\mathbb{R}^d}\int_{\mathbb{R}^d}\langle x-(\tau_i-\tau_j)\rangle^{-\delta}\langle y\rangle^{-\delta}\left(\frac{1}{|x-y|^{\frac{d-1}{2}}}+\frac{1}{|x-y|^{d-2}}\right) \,dxdy\nonumber\\
&\leq C\langle\tau_0\rangle^{-\frac{d-1}{2}},\label{eq4.36.8}
\end{align}
where in the first equality, we use \eqref{eq4.36.2} and \eqref{eq4.36.6}; in the second inequality, we use \eqref{eq4.36.4}, \eqref{eq4.36.5} and the decay property of $w_{\pm, <}$ (see \eqref{eq3.52}); the last inequality follows from Lemma \ref{lm2.8}.

Therefore the Lemma follows from \eqref{eq4.33.1}, \eqref{eq4.36.1} and \eqref{eq4.36.8}.
\end{proof}
Based on Lemma \ref{lm4.2}, we have two important observations.

\textbf{Claim 1.}  $I+F_{N\times N}^{\pm}(\lambda^2)$ is strictly diagonally-dominant by rows if $\tau_0$ is sufficiently large.

\emph{Proof of Claim 1:}
In fact, we choose $\tau_0$ such that
\begin{equation}\label{eq4.38.1}
\tau_0>\left(\frac{2(N-1)C_1}{c_0}\right)^{\frac{2}{d-1}},
\end{equation}
where $c_0, C_1$ are given in \eqref{eq1.4} and \eqref{eq4.33} respectively. Then for any $\lambda>0$, by \eqref{eq4.33}, the spectral assumption \eqref{eq1.4} and the fact that $f^{\pm}_{ii}(\lambda^2)=\langle R_0^{\pm}(\lambda^2)\varphi, \varphi\rangle$, we have
\begin{equation*}
  \sum_{j=1, j\ne i}^{N}{|f^{\pm}_{ij}(\lambda^2)|}\leq (N-1)C_1\langle\tau_0\rangle^{-\frac{d-1}{2}}\leq \frac{c_0}{2}<1+f^{\pm}_{ii}(\lambda^2),\,\,\,\,\, 1\leq i\leq N.
\end{equation*}
Hence the claim follows. As a consequence of \textbf{Claim 1},  it follows that $\det( A^{\pm}(\lambda^2))\ne 0$, then by Theorem \ref{thm4.2},  we are allowed to use the Aronszajn-Krein formula \eqref{eq4.13}.\qed

Next, we show that the inverse matrix $\left(I+F_{N\times N}^{\pm}(\lambda^2)\right)^{-1}$ has a series expansion. To this end, we denote by
\begin{equation}\label{eq4.39}
 g^{\pm}_{\varphi}(\lambda^2):=\frac{1}{1+\langle  R_0^{\pm}(\lambda^2)\varphi, \varphi\rangle}.
\end{equation}
\textbf{Claim 2.} For any $\lambda>0$ and sufficiently large $\tau_0$, we have the series expansion
\begin{equation}\label{eq4.41}
\left(I+F_{N\times N}^{\pm}(\lambda^2)\right)^{-1}=\sum_{n=0}^{\infty}g^{\pm}_{\varphi}(\lambda^2){\left(-g^{\pm}_{\varphi}(\lambda^2)\cdot F_{\tau_0}^{\pm}(\lambda^2)\right)^n},
\end{equation}
where the $N\times N$ matrix $F_{\tau_0}^{\pm}$ is given by \eqref{eq4.41.5} below.

\emph{Proof of Claim 2:}  We note that by translation invariance, one has
$$
\langle  R_0^{\pm}(\lambda^2)\varphi, \varphi\rangle=\langle  R_0^{\pm}(\lambda^2)\varphi_i, \varphi_i\rangle.
$$
Then we make the following decomposition
\begin{equation}\label{eq4.41.4}
I+F_{N\times N}^{\pm}(\lambda^2)=\left(1+\langle  R_0^{\pm}(\lambda^2)\varphi, \varphi\rangle\right)\cdot I+F_{\tau_0}^{\pm}(\lambda^2),
\end{equation}
where
\begin{equation}\label{eq4.41.5}
  F_{\tau_0}^{\pm}(\lambda^2):=\begin{pmatrix}
        0 & \tiny{f_{12}^{\pm}(\lambda^2)} & \cdots  & \tiny{f_{1N}^{\pm}(\lambda^2)} \\
        \\
       \tiny{f_{21}^{\pm}(\lambda^2)} & 0 & \cdots & \tiny{f_{2N}^{\pm}(\lambda^2)}\\
         \vdots &  \vdots &  \ddots &   \vdots\\
        \tiny{f_{N1}^{\pm}(\lambda^2)} & \tiny{f_{N2}^{\pm}(\lambda^2)} & \cdots & 0
    \end{pmatrix}.
\end{equation}
In view of \eqref{eq4.39} and \eqref{eq4.41.4}, we see that
\begin{equation*}
 \left(I+F_{N\times N}^{\pm}(\lambda^2)\right)^{-1}=g^{\pm}_{\varphi}(\lambda^2)\left(I+g^{\pm}_{\varphi}\cdot F_{\tau_0}^{\pm}(\lambda^2)\right)^{-1}.
\end{equation*}
Then the expansion \eqref{eq4.41} follows if one can prove that the spectral radius satisfies
 $$
 \rho(g^{\pm}_{\varphi}\cdot F_{\tau_0}^{\pm}(\lambda^2))<1,
 $$
provided $\tau_0$ is  sufficiently large. In fact, let $\tau_0$ satisfy \eqref{eq4.38.1},
then by \eqref{eq4.33}, we have for all $\lambda>0$,
\begin{align*}
 |g^{\pm}_{\varphi}|\cdot\max_{i}\sum^N_{j=1,j\ne i}{|f_{ij}^{\pm}(\lambda^2)|}\leq \frac{1}{c_0}\cdot (N-1)C_1\langle\tau_0\rangle^{-\frac{d-1}{2}}\leq \frac12.
\end{align*}
Thus the spectral radius of $F_{\pm}$ satisfies
\begin{equation*}\label{eq4.40}
 \rho(g^{\pm}_{\varphi}\cdot F_{\tau_0}^{\pm}(\lambda^2))\leq \|g^{\pm}_{\varphi}\cdot F_{\tau_0}^{\pm}(\lambda^2)\|_{\infty}=|g^{\pm}_{\varphi}|\cdot\max_{i}\sum^N_{j=1,j\ne i}{|F_{ij}^{\pm}(\lambda^2)|}\leq \frac12.
\end{equation*}
This ends the proof of \textbf{Claim 2}.\qed

Now we are in a position to prove the dispersive estimate \eqref{eq1.17}.
Thanks to \eqref{eq4.41}, we are able to rewrite \eqref{eq4.13.2}  as
\begin{equation}\label{eq4.45}
e^{-itH}-e^{-itH_{0}}=(\Omega_0^+-\Omega_0^-)+\sum_{n=1}^{\infty}(-1)^n(\Omega_n^+-\Omega_n^-),
\end{equation}
where
\begin{equation}\label{eq4.45.1}
\Omega_0^{\pm}:=\frac{-1}{\pi i}\sum_{j=1}^N\int_0^{\infty}e^{-it\lambda^2}R_0^{\pm}(\lambda^2){P_j}R_0^{\pm}(\lambda^2)g^{\pm}_{\varphi}(\lambda^2)\lambda\,d\lambda,
\end{equation}
\begin{equation}\label{eq4.45.2}
\Omega_n^{\pm}:=\frac{-1}{\pi i}\sum_{i, j=1}^N\int_0^{\infty}e^{-it\lambda^2}R_0^{\pm}(\lambda^2)\varphi_i\left\langle R_0^{\pm}(\lambda^2)\cdot\,,\varphi_j \right\rangle b^{\pm}_{n, ij}(\lambda^2)\lambda\,d\lambda,
\end{equation}
and  $b^{\pm}_{n, ij}$  denotes the element in the i-th row and j-th column of the following matrix
\begin{equation}\label{eq4.45.3}
 (b_{n, ij})_{N\times N}:=g^{\pm}_{\varphi}(\lambda^2)\left(-g^{\pm}_{\varphi}(\lambda^2)\cdot F_{\tau_0}^{\pm}(\lambda^2)\right)^n,\qquad n\ge 1.
\end{equation}
 Note that $g^{\pm}_{\varphi}(\lambda^2)$ (given by \eqref{eq4.39}) equals $\frac{1}{1+F_{\pm}(\lambda^2)}$ in the rank one case. Then we have

\textbf{Fact 2.} In the high energy part, \eqref{eq1.4.1} holds when $F^{\alpha, h}_{\pm}(\lambda^2)$ is replaced by $g^{\pm}_{\varphi}(\lambda^2)$;
in the low energy part,  \eqref{equ6.2.3} and \eqref{equ6.2.3.3} in Lemma \ref{lm3.2} are valid when  $F^{\alpha, l}_{\pm}(\lambda^2)$ is replaced by $g^{\pm}_{\varphi}(\lambda^2)$.

We first deal with $\Omega_0^{+}-\Omega_0^{-}$. Observe that \eqref{eq4.45.1} coincides with \eqref{eq3.3} ( $\alpha=1$), except that $\varphi$ is replaced by $\varphi_j$. Then by \eqref{eq1.5} in  Theorem \ref{thm-1.1}, there is some $C=C(d, \varphi)>0$ such that
\begin{equation}\label{eq4.45.4}
\|\Omega_0^{+}-\Omega_0^{-}\|_{L^1-L^{\infty}}\leq CN\cdot t^{-\frac{d}{2}},\quad \,\,\,\text{for}\,\,\, t>0.
\end{equation}

Therefor it suffices to prove $\Omega_n^{+}-\Omega_n^{-}$ with $n\ge 1$. The analysis is more involved due to the term $b^{\pm}_{n, ij}(\lambda^2)$ and we can't reduce it to the rank one case directly.  Using the same cutoff function $\chi$ (see \eqref{eq3.3.1}) as in the rank one case, we write
\begin{align*}
  \Omega_n^{\pm}&=\frac{-1}{\pi i}\sum_{i, j=1}^N\int_0^{\infty}e^{-it\lambda^2}\chi(\lambda)R_0^{\pm}(\lambda^2)\varphi_i\left\langle R_0^{\pm}(\lambda^2)\cdot\,,\varphi_j \right\rangle b^{\pm}_{n, ij}(\lambda^2)\lambda\,d\lambda\\
  &-\frac{1}{\pi i}\sum_{i, j=1}^N\int_0^{\infty}e^{-it\lambda^2}(1-\chi(\lambda))R_0^{\pm}(\lambda^2)\varphi_i\left\langle R_0^{\pm}(\lambda^2)\cdot\,,\varphi_j \right\rangle b^{\pm}_{n, ij}(\lambda^2)\lambda\,d\lambda\\
  &:=\Omega_n^{\pm, l}+\Omega_n^{\pm, h}.
\end{align*}
We point out the following fact, which follows directly from the definition of $b^{\pm}_{n, ij}$.

\textbf{Claim 3.}  $b^{\pm}_{n, ij}$ is the sum of at most $N^{n-1}$ terms. And each term takes the form
$$
(-1)^n\left(g^{\pm}_{\varphi}(\lambda^2)\right)^{n+1}\cdot\prod^n_{s=1}f^{\pm}_{i_s, j_s}(\lambda^2), \,\,\,\text{where}\,\,i_s\ne j_s.
$$
We remark that this claim will be important to obtain the upper bound $CN^2$ and will be used in both high and low energy part.

Before proceeding, we observe that the function $f^{\pm}_{i_s, j_s}(\lambda^2)$ is not a "good" symbol (with respect to $\lambda$). For simplicity, we only consider $d=3$. By Lemma \ref{lm4.2},  $f^{\pm}_{i_s, j_s}(\lambda^2)$ is controlled by $\langle\tau_0\rangle^{-\frac{d-1}{2}}$. However, by \eqref{eq4.33.2}  one has
$$
\frac{d^k}{d\lambda^k}f^{\pm}_{i_s, j_s}(\lambda^2)=\frac{(\pm i)^k}{4\pi}\int_{\mathbb{R}^6}e^{\pm i\lambda|x-y|}|x-y|^{k-1}\varphi(x-\tau_{i_s})\varphi(y-\tau_{j_s})\,dxdy,
$$
we see that its second derivative may grow with respect to $\tau_0$. The issue that causes this problem is that there is an oscillatory term $e^{\pm i\lambda|x-y|}$. We shall overcome this difficulty by using the integral representation of $f^{\pm}_{i_s, j_s}(\lambda^2)$ (see \eqref{eq4.36} and \eqref{eq4.36.7}), then we put all the oscillatory terms together (see \eqref{eq4.51} \eqref{eq4.52} below). This allows us to reduce the analysis to the rank one case.

\emph{Step 1. High energy part.}
The kernel of  $\Omega_n^{\pm, h}$ can be expressed as
\begin{align}\label{eq4.46}
&\Omega_n^{\pm, h}(t, x, y)\nonumber\\
&=\sum_{i, j=1}^N\int_0^{\infty}e^{-it\lambda^2}(1-\chi(\lambda))b_{n,ij}(\lambda^2)R^{\pm}_0(\lambda^2, |x-x_1|)R^{\pm}_0(\lambda^2, |x_2-y|)\varphi(x_1)\varphi(x_2)\,dx_1dx_2.
\end{align}
Using \textbf{Claim 3}  as well as the expression of $f^{\pm}_{i_s, j_s}$ (see \eqref{eq4.33.2}, \eqref{eq4.36} for odd $d$ and \eqref{eq4.36.7} for even $d$), we further break it into two cases.

(\romannumeral1) When $d\ge 3$ is odd,  then by \eqref{eq4.33.2}, \eqref{eq4.36} we have
\begin{equation}\label{eq4.46.1}
\prod^n_{s=1}f^{\pm}_{i_s, j_s}(\lambda^2)=\int_{\mathbb{R}^{2nd}}e^{\pm i\lambda\sum_{s=1}^n{|y_s-z_s|}}\prod^n_{s=1}G^{\pm}(y_s,z_s, \tau_{i_s}, \tau_{j_s})dz_1dy_n\cdots dz_ldy_n.
\end{equation}
By  \textbf{Claim 3}, we see that  $\Omega_n^{\pm, h}(t, x, y)$ is a linear combination of at most $N^{n+1}$ terms. Moreover, each term takes the following form
\begin{align}\label{eq4.50}
&\sum_{k_1, k_2=0}^{\frac{d-3}{2}}C^{\pm}(k_1, k_2)\int_{\mathbb{R}^{2d(n+1)}}I_{n, k_1, k_2}^{\pm}(|x-x_1|, |x_2-y|, \sum_{s=1}^n{|y_s-z_s|})\nonumber\\
&\cdot \frac{\varphi_i(x_1)}{|x-x_1|^{d-2-k_1}}\frac{\varphi_j(x_2)}{|x_2-y|^{d-2-k_2}} \prod_{s=1}^nG^{\pm}(y_s,z_s, \tau_{i_s}, \tau_{j_s})dz_1dy_1\cdots dz_ndy_ndx_1dx_2,
\end{align}
where
\begin{align}\label{eq4.51}
&I_{n, k_1, k_2}^{\pm}(|x-x_1|, |x_2-y|, \sum_{s=1}^n{|y_s-z_s|})\nonumber\\
&=\int_0^{\infty}e^{-it\lambda^2\pm i\lambda\cdot(|x-x_1|+|x_2-y|+\Sigma_n)}(1-\chi(\lambda))\left(g^{\pm}_{\varphi}(\lambda^2)\right)^{n+1}\lambda^{k_1+k_2+1}d\lambda,
\end{align}
and
\begin{equation}\label{eq4.52}
 \Sigma_n=\sum_{s=1}^n{|y_s-z_s|}.
\end{equation}

When $d=3$, note that $k_1=k_2=0$ in \eqref{eq4.51}. Thanks to \textbf{Fact 2}, we apply  (\romannumeral2) in Corollary \ref{cor2.6} with $\psi_1=\left(g^{\pm}_{\varphi}(\lambda^2)\right)^{n+1}$, $\psi_2=\lambda$. Thus we have
\begin{equation}\label{eq4.52.2}
|I_{n, 0, 0}^{\pm}|\leq Ct^{-\frac32}\left\langle|x-x_1|+|x_2-y|+\Sigma_n\right\rangle.
\end{equation}

When $d>3$, the method to estimate \eqref{eq4.50} is parallel to \eqref{eq3.16}.
Indeed, we apply operators $L_{x_i}$ and $L^*_{x_i}$ to  \eqref{eq4.50} (see \eqref{eq3.16.1}--\eqref{eq3.17}), then
\begin{align*}
 \eqref{eq4.50}& = \sum_{k_1, k_2=0}^{\frac{d-3}{2}}C^{\pm}(k_1, k_2)\int_{\mathbb{R}^{2d(n+1)}}\tilde{I}_{n, k_1, k_2}^{\pm}(|x-x_1|, |x_2-y|, \sum_{s=1}^n{|y_s-z_s|})\nonumber\\
&\cdot G_1(x, x_1, \tau_i)G_2(x_2, y, \tau_j)\prod_{s=1}^nG^{\pm}(y_s,z_s, \tau_{i_s}, \tau_{j_s})dz_1dy_1\cdots dz_ndy_ndx_1dx_2,
\end{align*}
and
\begin{align}\label{eq4.52.1}
&\tilde{I}_{n, k_1, k_2}^{\pm}(|x-x_1|, |x_2-y|, \sum_{s=1}^n{|y_s-z_s|})\nonumber\\
&=\int_0^{\infty}e^{-it\lambda^2\pm i\lambda\cdot(|x-x_1|+|x_2-y|+\Sigma_n)}(1-\chi(\lambda))\left(g^{\pm}_{\varphi}(\lambda^2)\right)^{n+1}\lambda^{k_1+k_2-n_1-n_2+1}d\lambda,
\end{align}
where $n_1, n_2$ are given by \eqref{eq3.16.3}, $G_1(x, x_1, \tau_i),\, G_2(x_2, y, \tau_j)$ are given by \eqref{eq3.19.1} with $\varphi$ replaced by $\varphi_i$, $\varphi_j$ respectively. Similar to \eqref{eq3.24}-\eqref{eq3.25}, one has
$$
|G_1(x, x_1, \tau_i)|\leq  C\left(|x-x_1|^{-d+1}+|x-x_1|^{-\frac{d-1}{2}}\right)\cdot\langle x_1-\tau_i \rangle^{-\delta},\,\,\,\,\delta>d+\frac{3}{2},
$$
and
$$
| G_2(x_2, y, \tau_j)|\leq C\left(|x_2-y|^{-d+1}+|x_2-y|^{-\frac{d-1}{2}}\right)\cdot\langle x_2-\tau_j \rangle^{-\delta},\,\,\,\,\delta>d+\frac{3}{2}.
$$
By \textbf{Fact 2}, we apply  (\romannumeral2) in Corollary \ref{cor2.6} with $\psi_1=\left(g^{\pm}_{\varphi}(\lambda^2)\right)^{n+1}$, $\psi_2=\lambda^{k_1+k_2-n_1-n_2+1}$ (note that we have $k_1+k_2-n_1-n_2+1\leq \frac{d-1}{2}$). Then
\begin{equation}\label{eq4.52.3}
  \left|\tilde{I}_{n, k_1, k_2}^{\pm}(|x-x_1|, |x_2-y|, \sum_{s=1}^n{|y_s-z_s|})\right|\leq Ct^{-\frac d2}\langle|x-x_1|+|x_2-y|+\Sigma_n\rangle^{\frac{d-1}{2}}.
\end{equation}
Therefore by \eqref{eq4.52.2}, \eqref{eq4.52.3} and the fact that
$$
\langle|x-x_1|+|x_2-y|+\Sigma_n\rangle^{\frac{d-1}{2}}\leq \langle x-x_1\rangle^{\frac{d-1}{2}}+\langle x_2-y\rangle^{\frac{d-1}{2}}+\sum_{s=1}^n\langle z_s-y_s\rangle^{\frac{d-1}{2}},
$$
we obtain
\begin{align}\label{eq4.53}
 |\eqref{eq4.50}| &\leq C\sum_{k_1, k_2=0}^{\frac{d-3}{2}}t^{-\frac{d}{2}}
\int_{\mathbb{R}^{2d(n+1)}}\langle|x-x_1|+|x_2-y|+\Sigma_n\rangle^{\frac{d-1}{2}}\nonumber\\
&\cdot |G_1(x, x_1)G_2(x_2, y)| \prod_{s=1}^n|G^{\pm}(y_s,z_s, \tau_{i_s}, \tau_{j_s})|dz_1dy_1\cdots dz_ndy_ndx_1dx_2,\nonumber\\
&\leq Cn t^{-\frac{d}{2}} C(d, \varphi)^n\langle\tau_0\rangle^{-\frac{d-1}{2}(n-1)},
\end{align}
where in the second inequality we have used the following two observations: First, it follows from  \eqref{eq4.35.4} (the decay of $G^{\pm}$) and Lemma \ref{lm2.8} that
\begin{align*}
&\int_{\mathbb{R}^{2d}}\langle z_s-y_s\rangle^{\frac{d-1}{2}}|G^{\pm}(y_s,z_s, \tau_{i_s}, \tau_{j_s})|\,dz_sdy_s \\
 &\leq C\int_{\mathbb{R}^{2d}}(1+\frac{1}{|y_s-z_s|^{\frac{d-3}{2}}})\langle z_s-(\tau_{i_s}-\tau_{j_s})\rangle^{-\delta}\langle y_s\rangle^{-\delta}\,dz_sdy_s\\
 &\leq C,
\end{align*}
which doesn't depend on $\tau_0$; second, we have
\begin{align*}
  \int_{\mathbb{R}^{2(n-1)d}}\prod_{k=1,k\ne s}^n|G^{\pm}(y_k,z_k, \tau_{i_k}, \tau_{j_k})|&=\prod_{k=1,k\ne s}^n\int_{\mathbb{R}^{2d}}|G^{\pm}(y_k,z_k, \tau_{i_k}, \tau_{j_k})|dz_kdy_k\\
  &\leq C(d, \varphi)^{n-1}\langle\tau_0\rangle^{-\frac{d-1}{2}(n-1)}.
\end{align*}

(\romannumeral2) When $d\ge 4$ is even, then by \eqref{eq4.36.7}, we see that
 \begin{align}\label{eq4.54}
\prod^n_{s=1}f^{\pm}_{i_s, j_s}(\lambda^2)&=\sum_{k=0}^n\binom{n}{k}\int_{\mathbb{R}^{2nd}}e^{\pm i\lambda \cdot\Sigma_k}\prod_{s=1}^k{\tilde{G}^{\pm}(\lambda, z_s, y_s, \tau_{i_s}, \tau_{j_s})}\nonumber\\
&\cdot\prod_{s=k+1}^n{\frac{w_{\pm, <}(\lambda|z_s-y_s|)}{|z_s-y_s|^{d-2}}\varphi_{i_s}(z_s)\varphi_{j_s}(y_s)}\,dz_1dy_1\cdots dz_ndy_n\nonumber\\
&:=\sum_{k}\binom{n}{k}f_{k}^{\pm, n}(\lambda^2),
 \end{align}
 where $\Sigma_k:=\sum_{s=1}^k{|y_s-z_s|}$, which is defined in the same way as \eqref{eq4.52}. When $k=0$ or $k=n$, we use the convention that $\Sigma_0=0$ and $\prod_{s=m+1}^m=1$.

 By \textbf{Claim 3} and \eqref{eq4.54}, we see  from \eqref{eq4.46}  that  $\Omega_n^{\pm, h}(t, x, y)$ is a linear combination of at most $2^nN^{n+1}$ terms. Moreover, by \eqref{eq3.50}, each term can be written as the following sum
 \begin{equation*}
   \tilde{U}_{d, 1}^{\pm, h}(t, x, y)+\tilde{U}_{d, 2}^{\pm, h}(t, x, y)+\tilde{U}_{d, 3}^{\pm, h}(t, x, y)+\tilde{U}_{d, 4}^{\pm, h}(t, x, y),
 \end{equation*}
where $\tilde{U}_{d, j}^{\pm, h}(t, x, y)$ ($1\le j\le 4$) is defined in the same way as $U_{d, j}^{\pm, h}(t, x, y)$ (see \eqref{eq3.54.11}--\eqref{eq3.57}), except that we replace $F^{\alpha, h}_{\pm}(\lambda^2)$ by $$
\left(g^{\pm}_{\varphi}(\lambda^2)\right)^{n+1}\cdot f_{k}^{\pm, n}(\lambda^2)(1-\chi(\lambda)).
$$
Thus we can apply the same arguments used for $U_{d, j}^{\pm, h}$ to estimate $\tilde{U}_{d, j}^{\pm, h}$.


For the sake of completeness, we take $\tilde{U}_{d, 1}^{\pm, h}(t, x, y)$ as an example. Precisely, using the same computations as in \eqref{eq3.58.11}-\eqref{eq3.19.3}, we see that $\tilde{U}_{d, 1}^{\pm, h}(t, x, y)$ is a combination of
\begin{equation}\label{eq3.18.2}
C\int_{\mathbb{R}^{2(n+1)d}}{\tilde{I_k}^{\pm, h}(t)\tilde{G}_1(x, x_1, \tau_i)\tilde{G}_2(x_2, y, \tau_j)\prod_{s=k+1}^n\frac{\varphi_{i_s}(z_s)\varphi_{j_s}(y_s)}{|z_s-y_s|^{d-2}}\,dy_1dz_1\cdots dy_ndz_n dx_1dx_2},
\end{equation}
and
\begin{align}\label{eq3.19.33}
\tilde{I_k}^{\pm, h}(t)&=\int_0^{\infty}e^{-it\lambda^2\pm i \lambda(|x-x_1|+ |x_2-y|+\Sigma_k)}(1-\chi(\lambda))\nonumber\\
&\cdot (\lambda|x-x_1|)^{|\beta_1|}w^{(|\beta_1|)}_{\pm, >}(\lambda|x-x_1|)\cdot (\lambda|x_2-y|)^{|\beta_2|}w^{(|\beta_2|)}_{\pm, >}(\lambda|x_2-y|) \nonumber\\
&\cdot \prod_{s=1}^k{\tilde{G}^{\pm}(\lambda, z_s, y_s, \tau_{i_s}, \tau_{j_s})}\left(g^{\pm}_{\varphi}(\lambda^2)\right)^{n+1}\prod_{s=k+1}^n w_{\pm, <}(\lambda|z_s-y_s|) \lambda^{d-1-2([\frac{d-3}{4}]+1)}d\lambda,
\end{align}
where $\tilde{G}_1$ and $\tilde{G}_2$ are
 given by  \eqref{eq3.18} with $\varphi$ replaced by $\varphi_i$, $\varphi_j$ respectively. Similar to \eqref{eq3.24.1}-\eqref{eq3.25.1}, one has
 \begin{equation*}\label{eq3.24.11}
  |\tilde{G}_1(x, x_1, \tau_i)|\leq C\left(|x-x_1|^{-d+1}+|x-x_1|^{-\frac{d}{2}+1}\right)\cdot\langle x_1-\tau_i \rangle^{-\delta},\,\,\,\,\delta>d+\frac{3}{2},
\end{equation*}
and
\begin{equation*}\label{eq3.25.111}
  |\tilde{G}_2(x_2, y, \tau_j)|\leq C\left(|x_2-y|^{-d+1}+|x_2-y|^{-\frac{d}{2}+1}\right)\cdot\langle x_2-\tau_j \rangle^{-\delta},\,\,\,\,\delta>d+\frac{3}{2}.
\end{equation*}
Note that by \eqref{eq4.36.4},  \eqref{eq4.36.5}, \eqref{eq3.51} and \eqref{eq3.52}, we have
$$
\langle z_s-\tau_{i_s}\rangle^{\delta}\langle y_s-\tau_{j_s}\rangle^{\delta}\left(\frac{1}{|z_s-y_s|^{\frac{d-1}{2}}}+\frac{1}{|z_s-y_s|^{d-2}}\right)^{-1}\tilde{G}^{\pm}\in S^0_{\frac d2+1}\left((\frac{\lambda_0}{2}, \infty)\right),
$$
$$\lambda^{|\beta_1|}|x-x_1|^{\frac12+|\beta_1|}\cdot w^{(|\beta_1|)}_{\pm, >}(\lambda|x-x_1|)\cdot \lambda^{|\beta_2|}|x_2-y|^{\frac12+|\beta_2|}\cdot w^{(|\beta_2|)}_{\pm, >}(\lambda|x_2-y|)\in S^{-1}_{\frac d2+1}\left((\frac{\lambda_0}{2}, \infty)\right),$$
and
$$
w_{\pm, <}(\lambda|z_s-y_s|)\in S^0_{\frac d2+1}\left((\frac{\lambda_0}{2}, \infty)\right).
$$
Then by  \textbf{Fact 2}, (\romannumeral2) of Corollary \ref{cor2.6} and note that $d-1-2([\frac{d-3}{4}]+1)\leq \frac{d}{2}-1$, we have
\begin{align*}
 |\tilde{I_k}^{\pm, h}(t)|&\leq Ct^{-\frac{d}{2}}\langle |x-x_1|+ |x_2-y|+\Sigma_k\rangle^{\frac{d-1}{2}}|x-x_1|^{-\frac12}|x_2-y|^{-\frac12}\\
 &\cdot \prod_{s=1}^k\left(\frac{1}{|z_s-y_s|^{\frac{d-1}{2}}}+\frac{1}{|z_s-y_s|^{d-2}}\right) \langle z_s-\tau_{i_s}\rangle^{-\delta}\langle y_s-\tau_{j_s}\rangle^{-\delta}.
\end{align*}
Plugging this estimate into \eqref{eq3.19.33} and by the two observations after \eqref{eq4.53},  we deduce that
\begin{align*}\label{eq4.37}
 |\eqref{eq3.18.2}|&\leq  Ct^{-\frac{d}{2}}\int_{\mathbb{R}^{2d(l+1)}}\langle|x-x_1|+|x_2-y|+\Sigma_n\rangle^{\frac{d-1}{2}}|x-x_1|^{-\frac12}|x_2-y|^{-\frac12}\nonumber\\
 &\cdot|\tilde{G}_1(x, x_1, \tau_i)||\tilde{G}_2(x_2, y, \tau_j)| \prod_{s=1}^k\left(\frac{1}{|z_s-y_s|^{\frac{d-1}{2}}}+\frac{1}{|z_s-y_s|^{d-2}}\right) \nonumber\\
&\cdot\langle z_s-\tau_{i_s}\rangle^{-\delta}\langle y_s-\tau_{j_s}\rangle^{-\delta} \prod_{s=k+1}^n\frac{|\varphi_{i_s}(z_s)\varphi_{j_s}(y_s)|}{|z_s-y_s|^{d-2}}dz_1dy_1\cdots dz_ndy_n dx_1dx_2,\nonumber\\
&\leq Cn t^{-\frac{d}{2}} C(d, \varphi)^n\langle\tau_0\rangle^{-\frac{d-1}{2}(n-1)}.
\end{align*}
Therefore
$$|\tilde{U}_{d, 1}^{\pm, h}(t, x, y)|\leq Cn t^{-\frac{d}{2}} C(d, \varphi)^n\langle\tau_0\rangle^{-\frac{d-1}{2}(n-1)}.
$$
Combining (\romannumeral1) and (\romannumeral2), we conclude that
\begin{equation}\label{eq4.60}
  |\Omega_n^{\pm, h}(t, x, y)|\leq C t^{-\frac{d}{2}} N^{n+1}C(d, \varphi)^n\langle\tau_0\rangle^{-\frac{d-1}{2}(n-1)}.
\end{equation}

\emph{Step 2. Low energy part.} By \textbf{Claim 3}, we see that  $\Omega_n^{+, l}-\Omega_n^{-, l}$ is also a linear combination of at most $N^{n+1}$ terms. Moreover, each term can be written as the following form
$$
\Omega^{l}_{n,1}+\Omega^{l}_{n,2},
$$
where
\begin{align}
&\Omega^{l}_{n,1}:=\frac{-1}{\pi i}\int_0^{\infty}e^{-it\lambda^2}\chi(\lambda)\prod^n_{s=1}f^{+}_{i_s, j_s}(\lambda^2)\nonumber\\
&\cdot\left(R_0^{+}(\lambda^2)\varphi_i\left\langle R_0^{+}(\lambda^2)\cdot\,,\varphi_j \right\rangle g^{+}_{\varphi}(\lambda^2)^{n+1}\, -R_0^{-}(\lambda^2)\varphi_i\left\langle R_0^{-}(\lambda^2)\cdot\,,\varphi_j \right\rangle g^{-}_{\varphi}(\lambda^2)^{n+1}\right)\lambda\,d\lambda,\label{eq4.58}\\
&\Omega^{l}_{n,2}:=\frac{-1}{\pi i}\int_0^{\infty}e^{-it\lambda^2}\chi(\lambda) \left(\prod^n_{s=1}f^{+}_{i_s, j_s}(\lambda^2)-\prod^n_{s=1}f^{-}_{i_s, j_s}(\lambda^2)\right)\nonumber\\
&\cdot R_0^{-}(\lambda^2)\varphi_i\left\langle R_0^{-}(\lambda^2)\cdot\,,\varphi_j \right\rangle g^{-}_{\varphi}(\lambda^2)^{n+1}\lambda\,d\lambda,\label{eq4.59}.
\end{align}

$\bullet$  The estimate for $\Omega^{l}_{n, 1}$.

The expression of  $\Omega^{l}_{n,1}$ above is defined in the same way as $U_d^{+, l}- U_d^{-, l}$ in \eqref{eq3.5}, the main difference is that we replace $\alpha F^{\alpha, l}_{\pm}(\lambda^2)=\frac{\alpha\chi(\lambda)}{1+\alpha F^{\pm}(\lambda^2)}$ by
	$$
	\chi(\lambda)\left(g^{\pm}_{\varphi}(\lambda^2)\right)^{n+1}\cdot\prod^n_{s=1}f^{+}_{i_s, j_s}(\lambda^2).
	$$
Observe that for the term $\prod^n_{s=1}f^{+}_{i_s, j_s}(\lambda^2)$, we can apply the same arguments used in the high energy part  above.
Indeed, we embed  the integral representations  \eqref{eq4.46.1} and \eqref{eq4.54} into \eqref{eq4.58}. Then note that by the identity \eqref{eq.17.10} we can break the difference
$$
\left(R_0^{+}(\lambda^2)\varphi_i\left\langle R_0^{+}(\lambda^2)\cdot\,,\varphi_j \right\rangle g^{+}_{\varphi}(\lambda^2)^{n+1}\, -R_0^{-}(\lambda^2)\varphi_i\left\langle R_0^{-}(\lambda^2)\cdot\,,\varphi_j \right\rangle g^{-}_{\varphi}(\lambda^2)^{n+1}\right)
$$
into the sum of three terms. Therefore we are able to reduce the analysis of each term to  $U_{d, j}^{ l}$ ($1\le j\le 3$) in the rank one case (see \eqref{eq.l7}). To illustrate this, we only take one term (which corresponds to $U_{d, 1}^{ l}$, and $d\ge3$ is odd) as a model case. In this case, the kernel is the linear combination of
\begin{align}\label{eq4.61}
&\int_{\mathbb{R}^{2d(n+1)}}\tilde{I}_{n, k_1, k_2}^{l}(|x-x_1|, |x_2-y|, \Sigma_n)\nonumber\\
&\cdot \frac{\varphi_i(x_1)}{|x-x_1|^{d-2-k_1}}\frac{\varphi_j(x_2)}{|x_2-y|^{d-2-k_1}}\prod_{s=1}^nG^{\pm}(y_s,z_s, \tau_{i_s}, \tau_{j_s})dz_1dy_1\cdots dz_ndy_ndx_1dx_2,
\end{align}
where $0\le k_1, k_2\le \frac{d-3}{2}$ and
\begin{align*}
&\tilde{I}_{n, k_1, k_2}^{l}(|x-x_1|, |x_2-y|, \Sigma_n)\nonumber\\
&=\int_0^{\infty}e^{-it\lambda^2\pm i\lambda\cdot(|x-x_1|+|x_2-y|+\Sigma_n)}\chi(\lambda)\left(g^{+}_{\varphi}(\lambda^2)^{n+1}-g^{-}_{\varphi}(\lambda^2)^{n+1}\right)\lambda^{k_1+k_2+1}d\lambda,
\end{align*}
In view of \textbf{Fact 2}, we have
$$
\left(g^{+}_{\varphi}(\lambda^2)^{n+1}-g^{-}_{\varphi}(\lambda^2)^{n+1}\right)\lambda^{k_1+k_2+1}\in S^{d-1}_{\frac{d+1}{2}}\left((0,\, \lambda_0)\right).
$$
Then we apply (\romannumeral1) in Corollary \ref{cor2.6} with $b=d-1$, $K=\frac{d+1}{2}$ to derive that
$$
|\tilde{I}_{n, k_1, k_2}^{l}|\leq Ct^{-\frac{d}{2}}\langle|x-x_1|+|x_2-y|+\Sigma_n\rangle^{\frac{d-1}{2}}.
$$
Using the same arguments in the proof of \eqref{eq4.53},   we obtain that
 \begin{equation*}
   |\eqref{eq4.61}|\leq Cn t^{-\frac{d}{2}} C(d, \varphi)^n\langle\tau_0\rangle^{-\frac{d-1}{2}(n-1)}.
 \end{equation*}
 Therefore the kernel of $\Omega^{l}_{n,1}$ satisfies
 \begin{equation}\label{eq4.61.1}
   \sup_{x,y\in\mathbb{R}^d}|\Omega^{l}_{n,1}(t, x, y)|\leq Cn t^{-\frac{d}{2}} C(d, \varphi)^n\langle\tau_0\rangle^{-\frac{d-1}{2}(n-1)}.
 \end{equation}

%

$\bullet$  The estimate for $\Omega^{l}_{n, 2}$.

%
%
%
%
We point out  that the analysis of $\Omega^{l}_{n, 2}$ is also similar in spirit to the term $U_{d, 1}^{ l}$  in \eqref{eq.l7}. The main difference is that we need additional  decay of $\tau_0$ from
$$\prod^n_{s=1}f^{+}_{i_s, j_s}(\lambda^2)-\prod^n_{s=1}f^{-}_{i_s, j_s}(\lambda^2)  \qquad \text{in }\,\,\,\eqref{eq4.59}.$$
 Thus we shall proceed  slightly differently.
First we note that
\begin{equation}\label{eq4.63}
 \left(\prod^n_{s=1}f^{+}_{i_s, j_s}(\lambda^2)-\prod^n_{s=1}f^{-}_{i_s, j_s}(\lambda^2)\right)= \sum^n_{m=1}\prod^{m-1}_{s=1}f^{-}_{i_s, j_s}(\lambda^2)\,
 \cdot\,\left(f^{+}_{i_m, j_m}(\lambda^2)-f^{-}_{i_m, j_m}(\lambda^2)\right)\,\cdot\,\prod^n_{s=m+1}f^{+}_{i_s, j_s}(\lambda^2)
\end{equation}
The definition of $f^{\pm}_{i_m, j_m}(\lambda^2)$ and \eqref{equ6.2.8} show that
\begin{align}\label{eq4.64}
&f^{+}_{i_m, j_m}(\lambda^2)-f^{-}_{i_m, j_m}(\lambda^2) =\langle R_0^{+}(\lambda^2)\varphi_{j_m}, \varphi_{i_m}\rangle-\langle R_0^{-}(\lambda^2)\varphi_{j_m}, \varphi_{i_m}\rangle	\nonumber\\
&=\int_{\mathbb{R}^{2d}}\frac{\lambda^{\frac{d}{2}-1}}{|z_m-y_m|^{\frac{d}{2}-1}}\sum_{\pm}e^{\pm i\lambda|z_m-y_m|}\left(J_{\pm, >}(\lambda|z_m-y_m|)+J_{\pm, <}(\lambda|z_m-y_m|)\right) \varphi_{j_m}(z_m)\varphi_{i_m}(y_m)  \,dz_mdy_m\,,
\end{align}
where $J_{\pm, >}(\lambda|z_m-y_m|)$ and $J_{\pm,<}(\lambda|z_m-y_m|)$ satisfy \eqref{equ6.2.9} and \eqref{equ6.2.10} for  all  $d\geqslant 3$.

In view of \eqref{eq4.63},  we use \eqref{eq4.64} and apply \eqref{eq4.33.2}, \eqref{eq4.36}, \eqref{eq4.36.7} (the expressions of $f^{\pm}_{i_s, j_s}$) as well as the representations of the kernels $R_0^{\pm}(\lambda^2,x,y)$ into the integral \eqref{eq4.59}. Then the integral kernel of $\Omega^{l}_{n, 2}$ can be written as a sum of $n$ terms. We point out that the analysis of each term is very similar to  \emph{case 2} and \emph{case 3} in Theorem \ref{thm3.2} in the rank one case. For simplicity, we only give a brief sketch for even dimensions.

Using \eqref{eq4.54} and \eqref{eq4.64}, the kernel of $\Omega^{l}_{n, 2}$ can be represented as a sum of $n\cdot2^{n-1}$ terms with the following form:
$$\sum_{\pm}\tilde{\Omega}^{\pm,l}_{n, 2, 1}+\tilde{\Omega}^{\pm,l}_{n, 2, 2}+\tilde{\Omega}^{\pm,l}_{n, 2, 2}+\tilde{\Omega}^{\pm,l}_{n, 2, 4}\, ,$$
moreover, for $1\le r\le 4$,
\begin{align}\label{eq4.65}
\tilde{\Omega}^{\pm,l}_{n, 2, r}&:=C\int_{\mathbb{R}^{2(n+1)d}}\tilde{I}_{n, 2, r}^{\pm,l}(t) \varphi_i(x_1)\varphi_j(x_2) \nonumber \\ &\cdot\,\frac{\varphi_{i_m}(z_m)\varphi_{j_m}(y_m)}{|z_m-y_m|^{\frac{d}{2}-1}}\prod_{s=k_1+1}^{m-1}\frac{\varphi_{i_s}(z_s)\varphi_{j_s}(y_s)}{|z_s-y_s|^{d-2}}\,
	\prod_{s=k_2+1}^{n}\frac{\varphi_{i_s}(z_s)\varphi_{j_s}(y_s)}{|z_s-y_s|^{d-2}}\,dy_1dz_1\cdots dy_ndz_n dx_1dx_2,
 \end{align}
where $0\le k_1\le m-1$, $m\le k_2\le n$  and we use the convention that $\prod_{s=m+1}^m=1$.
In order to express $\tilde{I}_{n, 2, r}^{\pm,l}(t)$ in a unified form (with the same phase function), we denote by $$\tilde{w}_{+, <}(z):=e^{-iz}\cdot w_{+, <}(z).$$ Note that $\tilde{w}_{+, <}$ also satisfies \eqref{eq3.52}, and this is all we need for $\tilde{w}_{+, <}(z)$. Now the oscillatory integrals are
\begin{align*}
\tilde{I}_{n,2,1}^{\pm,l}(t)&:=\int_0^{\infty}e^{-it\lambda^2}e^{i\lambda(-|x-x_1|-|x_2-y|\pm|z_m-y_m|-\Sigma_{k_1}+\bar{\Sigma}_{k_2})}\chi(\lambda)g^{-}_{\varphi}(\lambda^2)^{n+1}\cdot\lambda^{\frac{3d}{2}-2}\nonumber \\
&\cdot\,\frac{w_{+, >}(\lambda|x-x_1|)}{|x-x_1|^{\frac{d}{2}-1}}\frac{w_{+, >}(\lambda|x_2-y|)}{|x_2-y|^{\frac{d}{2}-1}}H^{\pm}(\lambda)d\lambda,
\end{align*}

\begin{align*}
\tilde{I}_{n,2,2}^{\pm,l}(t)&:=\int_0^{\infty}e^{-it\lambda^2}e^{i\lambda(-|x-x_1|-|x_2-y|\pm|z_m-y_m|-\Sigma_{k_1}+\bar{\Sigma}_{k_2})}\chi(\lambda)g^{-}_{\varphi}(\lambda^2)^{n+1}\cdot\lambda^{d-1}\nonumber \\
&\cdot\,\frac{w_{+, >}(\lambda|x-x_1|)}{|x-x_1|^{\frac{d}{2}-1}}\frac{\tilde{w}_{+, <}(\lambda|x_2-y|)}{|x_2-y|^{d-2}}H^{\pm}(\lambda)d\lambda,		
\end{align*}

\begin{align*} \tilde{I}_{n,2,3}^{\pm,l}(t)&:=\int_0^{\infty}e^{-it\lambda^2}e^{i\lambda(-|x-x_1|-|x_2-y|\pm|z_m-y_m|-\Sigma_{k_1}+\bar{\Sigma}_{k_2})}\chi(\lambda)g^{-}_{\varphi}(\lambda^2)^{n+1}\cdot\lambda^{d-1}\nonumber \\
	&\cdot\,\frac{\tilde{w}_{+, <}(\lambda|x-x_1|)}{|x-x_1|^{d-2}}\frac{w_{+, >}(\lambda|x_2-y|)}{|x_2-y|^{\frac{d}{2}-1}}H^{\pm}(\lambda)d\lambda,		
\end{align*}

\begin{align*} \tilde{I}_{n,2,4}^{\pm,l}(t)&:=\int_0^{\infty}e^{-it\lambda^2}e^{i\lambda(-|x-x_1|-|x_2-y|\pm|z_m-y_m|-\Sigma_{k_1}+\bar{\Sigma}_{k_2})}\chi(\lambda)g^{-}_{\varphi}(\lambda^2)^{n+1}\cdot\lambda^{\frac{d}{2}}\nonumber \\
	&\cdot\,\frac{\tilde{w}_{+, <}(\lambda|x-x_1|)}{|x-x_1|^{d-2}}\frac{\tilde{w}_{+, <}(\lambda|x_2-y|)}{|x_2-y|^{d-2}}H^{\pm}(\lambda)d\lambda,		
\end{align*}	
where $\Sigma_{k_1}$ is given by \eqref{eq4.52}, and we also denote by
 $$\bar{\Sigma}_{k_2}:=\sum_{s=m+1}^{k_2}|z_s-y_s|,$$
 and
\begin{align} \label{eq4.57.1}
&H^{\pm}(\lambda):=\left(J_{\pm, >}(\lambda|z_m-y_m|)+J_{\pm, <}(\lambda|z_m-y_m|)\right)\,
\cdot\,\prod_{s=1}^{k_1}{\tilde{G}^{-}(\lambda, z_s, y_s, \tau_{i_s}, \tau_{j_s})} \nonumber \\
&\cdot\,\prod_{s=k_1+1}^{m-1} w_{-, <}(\lambda|z_s-y_s|)
\cdot\,\prod_{s=m+1}^{k_2}{\tilde{G}^{+}(\lambda, z_s, y_s, \tau_{i_s}, \tau_{j_s})}\,\cdot\,\prod_{s=k_2+1}^n w_{+, <}(\lambda|z_s-y_s|).
\end{align}
The estimate of $\tilde{\Omega}^{\pm,l}_{n, 2, j}$ is parallel to  $U_{d, 1, j}^{l}(t, x, y)$ (see \eqref{eq3.54.5}--\eqref{eq3.57.2}) for $j=1,2,3,4$. Indeed, note that by \eqref{eq4.36.4} and \eqref{eq4.36.5}, we have
\begin{equation*}
\langle z_s-\tau_{i_s}\rangle^{\delta}\langle y_s-\tau_{j_s}\rangle^{\delta}\left(\frac{1}{|z_s-y_s|^{\frac{d-1}{2}}}+\frac{1}{|z_s-y_s|^{d-2}}\right)^{-1}\tilde{G}^{\pm}\in S^0_{\frac d2+1}\left((0,\lambda_0)\right),	
\end{equation*}
Furthermore, by \eqref{equ6.2.5.2}--\eqref{equ6.2.7.2} (properties of $J_{\pm, \gtrless}$) and \eqref{equ6.2.5}--\eqref{equ6.2.7} (properties of $w_{\pm, \gtrless}$),
we have that if $$t^{-\frac12}|-|x-x_1|-|x_2-y|\pm|z_m-y_m|-\Sigma_{k_1}+\bar{\Sigma}_{k_2}|\leq1,$$
then by noting that
$$
|z_m-y_m|^{1-\frac{d}{2}}\cdot\left(J_{\pm, >}(\lambda|z_m-y_m|)+J_{\pm, <}(\lambda|z_m-y_m|)\right)\in S^{\frac{d}{2}-1}_{\frac d2+1}\left((0,\lambda_0)\right),	
$$
we have
\begin{align}\label{eq4.59.1}
|\tilde{I}_{n,2,r}^{\pm,l}(t)|&\leq Ct^{-\frac{d}{2}}|z_m-y_m|^{\frac{d}{2}-1}\cdot\,\prod_{s=1}^{k_1}\left( \langle z_s-\tau_{i_s}\rangle^{-\delta}\langle y_s-\tau_{j_s}\rangle^{-\delta}\left(\frac{1}{|z_s-y_s|^{\frac{d-1}{2}}}+\frac{1}{|z_s-y_s|^{d-2}}\right)\right)  \nonumber\\
&\cdot\,
\prod_{s=m+1}^{k_2}\left\{ \langle z_s-\tau_{i_s}\rangle^{-\delta}\langle y_s-\tau_{j_s}\rangle^{-\delta}\left(\frac{1}{|z_s-y_s|^{\frac{d-1}{2}}}+\frac{1}{|z_s-y_s|^{d-2}}\right)\right\} \nonumber\\
 &\cdot\,\left(\frac{1}{|x-x_1|^{\frac{d}{2}-1}}+\frac{1}{|x-x_1|^{d-2}}\right)\cdot\left(\frac{1}{|x_2-y|^{\frac{d}{2}-1}}+\frac{1}{|x_2-y|^{d-2}}\right).
\end{align}
If $$t^{-\frac12}|-|x-x_1|-|x_2-y|\pm|z_m-y_m|-\Sigma_{k_1}+\bar{\Sigma}_{k_2}|\geqslant1,$$
then  by noting that
$$
|z_m-y_m|^{\frac{1}{2}}\cdot\left(J_{\pm, >}(\lambda|z_m-y_m|)+J_{\pm, <}(\lambda|z_m-y_m|)\right)\in S^{-\frac{1}{2}}_{\frac d2+1}\left((0,\lambda_0)\right),	
$$
we have
\begin{align}\label{eq4.59.2}
	|\tilde{I}_{n,2,r}^{\pm,l}(t)|\leq& C t^{-\frac{d}{2}}|z_m-y_m|^{-\frac12}\cdot\,|-|x-x_1|-|x_2-y|\pm|z_m-y_m|-\Sigma_{k_1}+\bar{\Sigma}_{k_2}|^{\frac{d-1}{2}}
    \nonumber\\
	&\cdot\,\prod_{s=1}^{k_1}\left( \langle z_s-\tau_{i_s}\rangle^{-\delta}\langle y_s-\tau_{j_s}\rangle^{-\delta}\left(\frac{1}{|z_s-y_s|^{\frac{d-1}{2}}}+\frac{1}{|z_s-y_s|^{d-2}}\right)\right)  \nonumber\\
	&\cdot\,\prod_{s=m+1}^{k_2}\left( \langle z_s-\tau_{i_s}\rangle^{-\delta}\langle y_s-\tau_{j_s}\rangle^{-\delta}\left(\frac{1}{|z_s-y_s|^{\frac{d-1}{2}}}+\frac{1}{|z_s-y_s|^{d-2}}\right)\right)
\nonumber\\
 &\cdot\,\left(\frac{1}{|x-x_1|^{\frac{d-1}{2}}}+\frac{1}{|x-x_1|^{d-2}}\right)\cdot\left(\frac{1}{|x_2-y|^{\frac{d-1}{2}}}+\frac{1}{|x_2-y|^{d-2}}\right).
\end{align}
Plugging \eqref{eq4.59.1}  and \eqref{eq4.59.2} into \eqref{eq4.65},  and by the two observations after \eqref{eq4.53},  we therefore obtain that
\begin{equation}\label{eq4.59.3}
\sup_{x, y}|\Omega_{n,2}^{l}(t, x, y)|\leq C t^{-\frac{d}{2}}C(d, \varphi)^n\langle\tau_0\rangle^{-\frac{d-1}{2}(n-1)}.
\end{equation}
This, together with \eqref{eq4.61.1} and \textbf{Claim 3}, yields  that
\begin{equation} \label{eq4.88}
	\sup_{x, y}|\Omega_n^{+, l}(t, x, y)-\,\Omega_n^{-, l}(t, x, y)|\leq C t^{-\frac{d}{2}} N^{n+1}C(d, \varphi)^n\langle\tau_0\rangle^{-\frac{d-1}{2}(n-1)}.	
\end{equation}
Therefore combining \eqref{eq4.88} with \eqref{eq4.60} we deduce  that
\begin{equation}\label{eq4.88.1}
\sup_{x, y}|\Omega_n^{+}(t, x, y)-\,\Omega_n^{-}(t, x, y)|\leq C t^{-\frac{d}{2}} N^{n+1}C(d, \varphi)^n\langle\tau_0\rangle^{-\frac{d-1}{2}(n-1)}.
\end{equation}
Now  we choose a sufficiently large $\tau_0$ such that
$$
\tau_0>(2NC(d,\varphi))^{\frac{2}{d-1}},
$$
then
\begin{align*}
\left\| \sum_{n=1}^{\infty}(-1)^n(\Omega_n^+-\Omega_n^-)\right\|_{L^1-L^{\infty}}&\leq\sum_{n=1}^{\infty} CN^2t^{-\frac{d}{2}}\left(NC(d,\varphi)\langle\tau_0\rangle^{-\frac{d-1}{2}}\right)^{n-1} \\
  &\leq C N^2t^{-\frac{d}{2}}.
\end{align*}
 This completes the proof. \qed

%

\begin{remark}\label{rm4.1}
We note that the constant $N^2$ is caused by the term $\Omega_1^{+}-\,\Omega_1^{-}$ (see
\eqref{eq4.88}). It seems that the upper bound $N^2$ can't be improved from our proof. This is due to the fact that the following oscillatory integral estimates
$$
|t|^{\frac 12+b}\cdot \left|\int_0^{\infty}e^{-it\lambda^2\pm i\lambda x}(1-\chi(\lambda))\lambda^{b}d\lambda\right|\leq C|x|^{b}
$$
is sharp in the region $|t|^{-\frac 12}|x|>1$.
\end{remark}

\begin{remark}\label{rm4.2}
We compare our result with  the $L^p-L^{p'}$ decay estimate of Nier and Soffer in \cite[Theorem 1.1]{NS}, i.e., for $1<p\le2$,  $1<r<\min\{p,\frac{2d}{d+2}\}$ and $\left(\frac{1/p-1/2}{1/r-1/2}\right)<\theta\le  1 $,  they proved that
\begin{equation}\label{eq4.100}
\tau_0> (CN)^{\frac{1}{d(r-1)}}\Longrightarrow	\|e^{-itH}\|_{L^{p}-L^{p'}}\leq CN^{\theta}t^{-d(\frac{1}{p}-\frac12)}, \,\,\,\,\, t>0.
\end{equation}
Apply Riesz-Thorin interpolation between Theorem \ref{thm1.5} and the trivial $L^2$ conservation, it follows that for $1\le p\le 2$, $\tilde{\theta}=2(\frac1p-\frac12)$, one has
\begin{equation}\label{eq4.101}
\tau_0>(CN)^{\frac{2}{d-1}}\Longrightarrow	\|e^{-itH}\|_{L^{p}-L^{p'}}\leq CN^{2\tilde{\theta}}t^{-d(\frac{1}{p}-\frac12)}, \,\,\,\,\, t>0.
\end{equation}
Note that when $p$ is close to the endpoint 1, the upper bound of $N$ in \eqref{eq4.101} is worse than \eqref{eq4.100}, but the requirement for $\tau_0$ is much looser, since in \eqref{eq4.100}, $\tau_0\rightarrow \infty$ as $r\rightarrow 1$.

In the following we point out that   after a slight modification in our proof, we can also obtain \eqref{eq4.100}.
Indeed, by Riesz-Thorin interpolation theorem it suffices to prove that for $N\leq C{\tau_0}^{d(r-1)}$ , and $1<r<\min\{p,\frac{2d}{d+2}\}$, then
\begin{equation}\label{eq4.90}
\tau_0> (CN)^{\frac{1}{d(r-1)}}\Longrightarrow	  \|e^{-itH}\|_{L^r-L^{r'}}\leq CNt^{-d(\frac{1}{r}-\frac12)}, \,\,\,\,\, t>0.
\end{equation}
We only take $d=3$ as an example to show that \eqref{eq4.90} follows from the proof of Theorem \ref{thm1.5}.
In this case, we assume that $\varphi\in C^1(\mathbb{R}^3)$ and satisfies the decay assumption \eqref{eq1.3} with $\beta_0=1$.
Using the same notations as in  \eqref{eq4.45}--\eqref{eq4.45.2}, we have when $d=3$,
 $$\Omega_0^{\pm}(t,x,y):=\frac{C}{\pi i}\sum_{j=1}^N\int_{\mathbb{R}^3}\int_{\mathbb{R}^3}\int_0^{\infty}e^{-it\lambda^2\pm i\lambda(|x-x_1|+ |x_2-y|)}\frac{\varphi_j(x_1)}{|x-x_1|}\frac{\varphi_j(x_2)}{|x_2-y|}g^{\pm}_{\varphi}(\lambda^2)\lambda\,d\lambda\,dx_1dx_2.$$
In order to get additional decay concerning $\tau_j$, we apply the operator $L_{x_1}$ to the term $\frac{\varphi_j(x_1)}{|x-x_1|}$ (similar to the arguments in \eqref{eq3.16.1}-\eqref{eq3.24}), then $\Omega_0^{\pm}(t,x,y)$  can be written  as
 $$\frac{C}{\pi i}\sum_{j=1}^N\int_{\mathbb{R}^3}\int_{\mathbb{R}^3}\int_0^{\infty}e^{-it\lambda^2\pm i\lambda(|x-x_1|+ |x_2-y|)}G(x,x_1,\tau_j)\frac{\varphi_j(x_2)}{|x_2-y|}g^{\pm}_{\varphi}(\lambda^2)\,d\lambda\,dx_1dx_2,$$
 where $G(x,x_1,\tau_j)$ satisfies that

 $$|G(x,x_1,\tau_j)|\leq C\left(\frac{1}{|x-x_1|}+\frac{1}{|x-x_1|^2}\right)\left\langle x_1-\tau_j\right\rangle^{-\delta}.$$
 Since
 $g^{\pm}_{\varphi}(\lambda^2)\in S^0_1\left((0, \infty)\right)$, Lemma \ref{lm2.6}, together with Lemma \ref{lm2.8}, implies  that
 \begin{equation*}
  |\Omega_0^{\pm}(t,x,y)|\leq Ct^{-\frac12}\sum_{j=1}^N\left\langle x-\tau_j\right\rangle^{-1}\left\langle y-\tau_j\right\rangle^{-1}.
 \end{equation*}
 Note that the operator with integral kernel
 $\langle x-\tau_j\rangle^{-1}\langle y-\tau_j\rangle^{-1}$
 is bounded from  the Lorentz space $L^{\frac32,1}(\mathbb{R}^3)$ to the dual space $L^{3,\infty}(\mathbb{R}^3)$, moreover, this bound is independent with $\tau_i$ and $\tau_j$.
 Then
 \begin{equation}\label{eq4.88.11}
   \|\Omega_0^{+}-\Omega_0^{-}\|_{L^{\frac32,1}-L^{3,\infty}} \leq CNt^{-\frac12}.
 \end{equation}
Now we apply the off-diagonal Marcinkiewicz
 theorem (see \cite[Theorem 1.4.19]{gra}) to interpolate between \eqref{eq4.88.11} and \eqref{eq4.45.4}, it follows that when $1<r<\frac32$,
 \begin{equation*}
  \|\Omega_0^{+}-\Omega_0^{-}\|_{L^r-L^{r'}}\leq CNt^{-3(\frac{1}{r}-\frac12)}.
 \end{equation*}

  Similarly, by \textbf{ Claim 3}, $\Omega_n^{\pm}(t,x,y)$ can be written as a sum of at most $N^{n+1}$ terms, and each term has the following form:
\begin{equation} \label{eq4.99.1}
\int_{\mathbb{R}^6}\int_{\mathbb{R}^{6n}}I^{\pm}(t)G(x,x_1,\tau_i)\frac{\varphi_j(x_2)}{|x_2-y|}\prod_{s=1}^n\frac{\varphi_{j_s}(y_s)\varphi_{i_s}(z_s)}{|z_s-y_s|}\,\,dz_1dy_1\cdots dz_ndy_ndx_1dx_2,
\end{equation}
where
$$I^{\pm}(t)=\int_0^{\infty}e^{-it\lambda^2\pm i\lambda(|x-x_1|+ |x_2-y|)+\Sigma_n}\left( g^{\pm}_{\varphi}(\lambda^2)\right)^{n+1}d\lambda.$$
 Lemma \ref{lm2.6} also shows that
 $$|\eqref{eq4.99.1}|\leq Ct^{-\frac12}\left\langle \tau_0\right\rangle ^{-n} \left\langle x-\tau_i\right\rangle^{-1}\left\langle y-\tau_j\right\rangle^{-1}.$$
 Then
 \begin{equation*}
     \|\Omega_n^{+}-\Omega_n^{-}\|_{L^{\frac32,1}-L^{3,\infty}} \leq CN^{n+1}C(d,\varphi)^n\left\langle \tau_0\right\rangle ^{-n}t^{-\frac12}.
 \end{equation*}
 Using the same interpolation arguments between this and \eqref{eq4.88.1}, one has
\begin{equation} \label{eq.4.99.2}
\left\| \Omega_n^+-\Omega_n^-\right\|_{L^r-L^{r'}}\leq C N^{n+1}C(d,\varphi)^n\left\langle \tau_0\right\rangle ^{-n+1-3(r-1)}t^{-3(\frac{1}{r}-\frac12)}.
\end{equation}
If we choose  $N\leq \frac12C(d,\varphi){\tau_0}^{3(r-1)}$ with $1<r<\frac32$, \eqref{eq4.88} and \eqref{eq.4.99.2} prove the claim when $d=3$.
\end{remark}

\section{$L^1-L^{\infty}$ estimates for trace class perturbations--Proof of Theorem \ref{thm1.6}  }\label{sec5}
We first establish an  Aronszajn-Krein type formula for the perturbations in Theorem \ref{thm1.6}. To this end, we define $R(z):=(H-z)^{-1}$ for $z\in \mathbb{C}\setminus[0, \infty)$, where $H$ is defined in Theorem \ref{thm1.6}, and set $P:=\sum_{j=1}^{+\infty}P_j$. Following the same arguments  in \eqref{eq4.4}--\eqref{eq4.7} and noting that $PA=AP=A$, we obtain that $I+AR_0(z)P$ is invertible on $PL^2(\R^d)$ and
\begin{equation} \label{eq5.1.1}
R(z)=R_0(z)-R_0(z)P(I+AR_0(z)P)^{-1}AR_0(z), \,\,\,\,\text{for $z\in \mathbb{C}\setminus[0, \infty)$}.	
\end{equation}
Further,  by our assumption we have $\text{supp} \hat{\varphi_i}\cap \text{supp} \hat{\varphi_j}=\emptyset$, for $i\ne j,$
thus if  $z\in \mathbb{C}\setminus[0, \infty)$,  we have
$$I+AR_0(z)P=\sum_{j=1}^{+\infty}\left( 1+\lambda_j\left\langle R_0(z)\varphi_j,\,\varphi_j\right\rangle\right) P_j
\,\,\,\, \text{on }\,\,\,PL^2(\R^d). $$
In view of the above identity, the invertibility of $I+AR_0(z)P$ on $PL^2(\R^d)$ shows that
$$ 1+\lambda_j\left\langle R_0(z)\varphi_j,\,\varphi_j\right\rangle\neq 0, \qquad \forall \,\,i\ge 1,$$
moreover,  we have on $PL^2(\R^d)$,
$$(I+AR_0(z)P)^{-1}=\sum_{j=1}^{+\infty}\left( 1+\lambda_j\left\langle R_0(z)\varphi_j,\,\varphi_j\right\rangle\right)^{-1} P_j,
\,\,\,\,\text{for $z\in \mathbb{C}\setminus[0, \infty)$}. $$
This, together with \eqref{eq5.1.1} and the fact that $P_jA=\lambda_jP_j$, yields that
\begin{equation}\label{eq5.1.2}
R(z)=R_0(z)-\sum_{j=1}^{+\infty}\lambda_j\cdot \left( 1+\lambda_j\left\langle R_0(z)\varphi_j,\,\varphi_j\right\rangle\right)^{-1}R_0(z)P_jR_0(z), \,\,\,\,\text{for $z\in \mathbb{C}\setminus[0, \infty)$}.	
\end{equation}
Based on \eqref{eq5.1.2} and the limiting absorption principle for $R_0(z)$, we have
\begin{proposition}\label{pro5.1}
Let  A be given in Theorem \ref{thm1.6}, then $(0, \infty)\subset\sigma_{ac}(H)$ and we have the following Aronszajn-Krein type formula:
\begin{equation}\label{eq5.1.3}
R^{\pm}(\lambda^2)=R_0^{\pm}(\lambda^2)-\sum_{j=1}^{+\infty}\lambda_j\cdot \left( 1+\lambda_jf_{jj}^{\pm}(\lambda^2)\right) ^{-1}R_0^{\pm}(\lambda^2)P_jR_0^{\pm}(\lambda^2),\,\,\,\,\text{for $\lambda>0$}.
\end{equation}
\end{proposition}
\begin{proof}
We first note that by Remark $(\textbf{c}_1)$ (below Theorem \ref{thm1.6}), $M_j$ has a fixed lower bound, then \eqref{eq1.88.1} implies that
\begin{equation}\label{eq5.1.4}
	\sum_{j=1}^{+\infty}\lambda_j\cdot {M_j}^2<+\infty,
\end{equation}
 which, in turn, shows that the sequence $\{\lambda_j{M_j}^2\}_{j\ge 1}$ is uniformly bounded. Therefore, by our assumption (\romannumeral2) and  the continuity of $R_0(z)$ for $\text{Re}z\geq 0$ (or $\text{Re}z\leq 0$) in weighted $L^2$ spaces,  for any fixed $\lambda>0$, we can choose a sufficiently small $\varepsilon_0>0$ such that
$$|1+\lambda_j\left\langle R_0(\lambda^2\pm i\varepsilon)\varphi_j,\,\varphi_j\right\rangle|>c_0/2 $$
holds uniformly for $0\leq\varepsilon<\varepsilon_0$. In addition, by  assumption (\romannumeral1), one has for $\sigma>1/2$
$$\left\|  R_0(\lambda^2\pm i\varepsilon)P_j R_0(\lambda^2\pm i\varepsilon)\right\|_{L^2_\sigma-L^2_{-\sigma}}\leq
C \left\|R_0(\lambda^2\pm i\varepsilon)\right\|_{L^2_\sigma-L^2_{-\sigma}}^2\cdot {M_j}^2.$$
It follows from the above three inequalities that, for any fixed $\lambda>0$,
\begin{equation}\label{eq5.1.4.1}
\sum_{j=1}^{+\infty}\lambda_j\cdot \left( 1+\lambda_j\left\langle R_0(\lambda^2\pm i\varepsilon)\varphi_j,\,\varphi_j\right\rangle\right)^{-1}R_0(\lambda^2\pm i\varepsilon)P_jR_0(\lambda^2\pm i\varepsilon)
\end{equation}
is convergent uniformly for $0<\varepsilon<\varepsilon_0$ in weighted $L^2$ spaces.
Now choose $z=\lambda^2\pm i\varepsilon$ and let $\varepsilon\rightarrow 0$ in \eqref{eq5.1.2}, then combining
the limiting absorption principle with the uniform convergence of \eqref{eq5.1.4.1}, we obtain \eqref{eq5.1.3}, and the series is convergent in weighted $L^2$ spaces.

Finally the absolutely continuous spectrum is a direct consequence of the \eqref{eq5.1.3} and the fact that $R_0^{\pm}(\lambda^2)$ is locally uniformly bounded from $L^2_\sigma$ to $L^2_{-\sigma}$ with $\sigma>1/2$.
\end{proof}

\emph{Proof of Theorem \ref{thm1.6}.} First, we note that by Lemma \ref{lmA.1}, the spectrum of $H$  is purely absolutely continuous. Then applying \eqref{eq5.1.3} to Stone's formula, it is not difficult to obtain that
\begin{equation}\label{eq5.1.5}
	e^{-itH}-e^{-itH_0}=\sum_{j=1}^{+\infty}(e^{-itH_{\lambda_j}}-e^{-itH_0}),
\end{equation}
where $e^{-itH_{\lambda_j}}=e^{-it(H_0+\lambda_jP_j)}$ is the Sch\"{o}dinger operator with rank one perturbations. This guarantees the proof can be reduced to the rank one case. Note that the trace condition $\sum_{j=1}^{+\infty}\lambda_j<+\infty$ indicates the fact that the number of $\lambda_j,\,\,\lambda_j\geq1$ is finite.

In order to use the upper bound \eqref{eq1.5.0} in Theorem \ref{thm-1.1}, we divide into two cases:

\textbf{Case 1:} When $d\ge 3$, by \eqref{eq5.1.5} and Theorem \ref{thm-1.1}, we have

\begin{align*}
\|e^{-itH}-e^{-itH_0}\|_{L^1-L^{\infty}}
&\leq
\sum_{\lambda_j\geq 1}\|e^{-itH_{\lambda_j}}-e^{-itH_0}\|_{L^1-L^{\infty}}+\sum_{\lambda_j<1}\|e^{-itH_{\lambda_j}}-e^{-itH_0}\|_{L^1-L^{\infty}}\\
&\leq \sum_{\lambda_j\geq1}C(\lambda_j,\varphi_j)t^{-\frac d2}+
\sum_{\lambda_j< 1}C\lambda_j\cdot M^{2([\frac{d}{2}]+3)}t^{-\frac d2}\\
&\leq C(A)t^{-\frac d2},\,\,\,   t>0,
\end{align*}
where the first sum is convergent since the number of terms is finite, the second sum is convergent by the assumption \eqref{eq1.88.1}.

\textbf{Case 2:} When $d=1, 2$,
note that  $\text{supp} \hat{\varphi_i}\cap \text{supp} \hat{\varphi_j}=\emptyset,$ for $i\ne j$, therefore, there is at most one $j_0>0$ such that $\int_{\mathbb{R}^d}\varphi_{j_0}(x)\,dx\neq0$. Again by  \eqref{eq5.1.5} and Theorem \ref{thm-1.1}, we have
\begin{align*}
	&\|e^{-itH}-e^{-itH_0}\|_{L^1-L^{\infty}}\\
	&\leq \|e^{-itH_{\lambda_{j_0}}}-e^{-itH_0}\|_{L^1-L^{\infty}}+\sum_{\lambda_j\geq 1}\|e^{-itH_{\lambda_j}}-e^{-itH_0}\|_{L^1-L^{\infty}}+\sum_{\lambda_j< 1,j\neq j_0}\|e^{-itH_{\lambda_j}}-e^{-itH_0}\|_{L^1-L^{\infty}}\\
	&\leq C(A)t^{-\frac d2},\,\,\, t>0.
\end{align*}
This completes the proof of Theorem \ref{thm1.6}.\qed

\begin{remark}\label{rmk5.1}
In this remark, we show that there exist examples of $\{\varphi_j\}^{+\infty}_{j=1}$  satisfying the assumptions of Theorem \ref{thm1.6}.  First we choose  a $L^2$-normalized real-valued function $\varphi$ satisfying \eqref{eq1.2} and \eqref{eq1.3}.
We also assume that $\text{supp}\,\hat{\varphi}\subset B^d(0,1)$. Notice that
 $$
\|R^{\pm}_0(\lambda^2)\|_{L^2_{\sigma}-L^2_{-\sigma}}\leq C(k, \lambda_0)\cdot\lambda^{-1},\,\,\,\,\quad \lambda>\lambda_0.
$$
Then there exists a large number $L>1$, such that when $\lambda>L$, we have
\begin{equation}\label{eq5.2.1}
	|\left\langle R_0^{\pm}(\lambda^2)\varphi,\,\varphi\right\rangle|<1/2.
\end{equation}
Now we set  $A:=\sum_{j=1}^{+\infty}\lambda_j\langle\cdot\,, \varphi_j\rangle \varphi_j$ with
\begin{equation}\label{eq5.7}
 \varphi_j:=e^{i(L+2j)e_1\cdot x}\varphi(x),\quad\text{and}\quad  \lambda_j=2^{-j},
\end{equation}
where $e_1=(1,0,0,\cdots)\in \R^d$. Clearly we have $\hat{\varphi_j}=\hat{\varphi}(\xi-(L+2j)e_1)$ and thus
$$
\text{supp} \hat{\varphi_i}\cap \text{supp} \hat{\varphi_j}=\emptyset, \,\,\,\,i\ne j.
$$
It follows directly from the definition \eqref{eq5.7} that $\|\varphi_j\|_{L^2_{\sigma}}=\|\varphi\|_{L^2_{\sigma}}$, thus  \eqref{eq5.2.1} holds uniformly for $\varphi_j$. Furthermore,  observe that \eqref{eq1.2} and \eqref{eq1.3}
are satisfied for  $\varphi_j$ with
$$
M_j=C\cdot M\cdot (L+2j)^{[\frac d2]+1}
$$
for some absolute constant $M>0$.
 Since
  $$\text{Re}\left\langle R_0^{\pm}(\lambda^2)\varphi_j,\,\varphi_j\right\rangle=\int{\frac{|\hat{\varphi_j}(\xi)|^2}{|\xi|^2-\lambda^2}d\xi}\geq0
  ,\,\,\, \text{for $\lambda\leq L$},$$
this,  together with \eqref{eq5.2.1}, indicates that
$$
|1+\lambda_j f^{\pm}_{jj}(\lambda^2)|>1/2\quad  \text{holds\, uniformly \,for\,}\,\, \varphi_j.
$$ Furthermore, a direct computation shows that
 $$\sum_{j=1}^{+\infty}\lambda_j\cdot M_j^{2[\frac d2]+6}=C\cdot M^{2[\frac d2]+6}\sum_{j=1}^{+\infty}{2^{-j}\cdot (L+2j)^{([\frac d2]+1)(2[\frac d2]+6)}}<+\infty.$$

\end{remark}

\begin{appendix}

\renewcommand{\appendixname}{Appendix\,\,}
\section{Proof of Lemma  \ref{lm2.6}  and Corollary \ref{cor2.6}}\label{app}
\emph{ Proof of Lemma  \ref{lm2.6}.}
Let $\Omega=(0, r_0)$. We first write $\Omega=\Omega_1\cup\Omega_2$ where
\begin{equation*}
\Omega_1=\{\lambda\in\Omega;~\mbox{$|2\lambda+\frac xt|<\frac12|\frac xt|$}\}
\end{equation*}
and
\begin{equation*}
\Omega_2=\{\lambda\in\Omega;~\mbox{$|2\lambda+\frac xt|>\frac14|\frac xt|$}\}.
\end{equation*}
Choose the following partition of unity subordinate to this covering:
\begin{equation*}
\Phi_1(\lambda)=\Phi\left((2\lambda+\mbox{$\frac xt$})/\mbox{$\frac12|\frac xt|$}\right)~\mathrm{and}~\Phi_2(\lambda)=1-\Phi_1(\lambda),
\end{equation*}
where $\Phi\in C_0^\infty(\mathbb{R})$ such that $\Phi(\lambda)=1$ when $|\lambda|\leq\frac12$ and $\Phi(\lambda)=0$ when $|\lambda|\geq 1$.
Now $I(t,x)$ is split into
\begin{equation*}
I_j(t,x)=\int_{\Omega_j}e^{i(t\lambda^2+x\cdot\xi)}\Phi_j(\lambda)\psi(\lambda)d\lambda,~~j=1, 2.
\end{equation*}
We shall prove \eqref{eq6} with $I$ replaced by $I_j$, $j=1,2$.

{\bf Step 1. Estimates for $I_2$.} We shall decompose the integral properly and use the integration-by-parts arguments. To proceed, we further split $I_{2}$ into
\begin{equation*}
\begin{split}
I_{2}(t,x)
&=\int_{\Omega_2}e^{i(t\lambda^2+x\cdot\lambda)}\Phi(t^{\frac12}\lambda)\Phi_2(\lambda)\psi(\lambda)d\lambda\\
&+\int_{\Omega_2}e^{i(t\lambda^2+x\cdot\lambda)}(1-\Phi(t^{\frac12}\lambda))\Phi_2(\lambda)\psi(\lambda)d\lambda,\\
&:= I_{21}(t,x)+I_{22}(t,x).
\end{split}
\end{equation*}
Obviously, for the term $I_{21}$, using the fact that $\psi(\lambda)\in S_K^{b}(\Omega)$  we have
\begin{equation}\label{eqA1}
|I_{21}|\leq C\int_{|\lambda|\leq t^{-\frac12}}\lambda^b\,d\lambda\leq Ct^{-\frac{1+b}{2}}.
\end{equation}
In order to estimate the term $I_{22}$, we denote by
\begin{equation*}
Df:=\frac{1}{2t\lambda+x}\cdot\frac{d}{d\lambda}f,\qquad\quad D_{*}f:=-\frac{d}{d\lambda}(\frac{1}{2t\lambda+x}\cdot f).
\end{equation*}
Integration by parts $K$ times leads to
\begin{equation}\label{eqA2}
\begin{split}
I_{22}=&\int_{\Omega_2}e^{i(t\lambda^2+x\cdot\lambda)}D_*^K\left((1-\Phi(t^{\frac12}\lambda))\Phi_2(\lambda)\psi(\lambda)\right)d\lambda\\
=& \sum_{j=0}^KC_j\int_{\Omega_2}e^{i(t\lambda^2+x\cdot\lambda)}\frac{(2t)^j}{(2t\lambda+x)^{K+j}}\frac{d^{K-j}}{d\lambda^{K-j}}\left((1-\Phi(t^{\frac12}\lambda))\Phi_2(\lambda)\psi(\lambda)\right)d\lambda,
\end{split}
\end{equation}
where we have used the identity
\begin{equation*}
 \frac{d^{j}}{d\lambda^{j}}\frac{1}{2t\lambda+x}=\frac{(-2t)^j}{(2t\lambda+x)^{j+1}}.
\end{equation*}
To proceed, we also need  the following  facts. First, since $\lambda\in \Omega_2$, we have
\begin{equation*}
|2t\lambda+x|>\frac{2}{5|t\lambda|},
\end{equation*}
thus
\begin{equation}\label{eqA3}
|\frac{(2t)^j}{(2t\lambda+x)^{K+j}}|\leq Ct^{-K}\lambda^{-K-j}.
\end{equation}
Second, since $|\frac{d^{j}}{d\lambda^{j}}((1-\Phi(t^{\frac12}\lambda)))|\leq C\lambda^{-j}$ and $\psi(\lambda)\in S_K^{b}(\Omega)$, then
\begin{equation}\label{eqA4}
\left|\frac{d^{K-j}}{d\lambda^{K-j}}((1-\Phi(t^{\frac12}\lambda))\Phi_2(\lambda)\psi(\lambda))\right|\leq C\lambda^{b-K+j}.
\end{equation}
Now it follows from \eqref{eqA2}, \eqref{eqA3} and \eqref{eqA4} that
\begin{equation}\label{eqA5}
 |I_{22}|\leq Ct^{-K}\int_{t^{-\frac12}}^{r_0}{\lambda^{b-2K}\,d\lambda}\leq Ct^{-\frac{b+1}{2}}.
\end{equation}
Combining \eqref{eqA1} and \eqref{eqA5}, we obtain that
\begin{equation}\label{eqA6}
\begin{split}
|I_{2}(t,x)|\leq Ct^{-\frac{b+1}{2}}\leq\begin{cases}C|t|^{-\frac{1+b}{2}},\,\,\,\qquad \qquad  \mathrm{if}\,\,\,~|t|^{-\frac12}|x|\leq 1,\\
C|t|^{-\frac12-b}|x|^{b}, \,\,\, \,\,\,\quad\quad             \mathrm{if}\,\,\,~|t|^{-\frac12}|x|> 1,
\end{cases}
\end{split}
\end{equation}
which completes the proof of $I_{2}(t,x)$.

{\bf Step 2. Estimates for $I_1$.} We shall use van der Corput's Lemma after a scaling argument. In fact, a change of variables $\lambda\rightarrow x\lambda/t$ yields that
\begin{align}\label{eqA7}
I_1(t,x)&=\int_{\Omega_1}e^{i(t\lambda^2+x\cdot\lambda)}\Phi_1(\lambda)\psi(\lambda)d\lambda\nonumber\\
&=\frac{x}{t}\int_{\Omega'_1}e^{i\frac{x^2}{t}(\lambda^2+\lambda)}\tilde{\Phi}_1(\lambda)\psi(x\lambda/t)d\lambda,
\end{align}
where $\tilde{\Phi}_1(\lambda)=\Phi(4\lambda+2)$ and $\Omega'_1=\{\lambda:\, |2\lambda+1|<\frac12\}$. Note that $\frac{d^{2}}{d\lambda^{2}}(\lambda^2-\lambda)=2>0$, we apply van der Corput's Lemma (see e.g. \cite[p. 334]{St}) to \eqref{eqA7} and use the fact $\psi(\lambda)\in S_K^{b}(\Omega)$ to derive that
\begin{align}\label{eqA8}
|I_1(t,x)|&\leq |t|^{-\frac12}\left(|\psi(3x/4t)|+ |\frac{x}{t}|\int_{\Omega'_1}{|\psi'(x\lambda/t)|d\lambda}\right)\nonumber\\
 &\leq C|t|^{-\frac12-b}|x|^{b}.
\end{align}
Furthermore, when $|x|\leq |t|^{\frac12}$, \eqref{eqA8}  implies that
\begin{equation}\label{eqA9}
 |I_1(t,x)|\leq C|t|^{-\frac{1+b}{2}}.
\end{equation}
Therefore combing \eqref{eqA8} and \eqref{eqA9}, we prove the desired estimate for $I_1$

Note that the above arguments also work  for $\Omega=(r_0, \infty)$ and the proof of Lemma \ref{lm2.6} is complete.

\begin{remark}\label{rmk2.1}
The oscillatory integral $I(t,x)$	should be considered in the sense of distribution when $\Omega=(r_0, \infty)$. That is, we define define $I_\epsilon(t,x)$  with $\psi(\lambda)$ replaced by $e^{-\epsilon\lambda}\psi(\lambda)$ for $\epsilon\in(0,1)$ in \eqref{eq2}, then one has that $I_\epsilon(t,x)\rightarrow I(t,x)$ in the sense of distribution. It suffices to prove $I_\epsilon(t,x)$ satisfies \eqref{eq6} with constant $C$ independent of $\epsilon$. Then $I(t,x)$ should satisfy the same estimate.
\end{remark}

 \emph{Proof of Corollary \ref{cor2.6}.}
We first prove  (\romannumeral1). On the one hand, if $t^{-\frac12}|x|\leq 1$, then apply Lemma \ref{lm2.6} with $\psi(\lambda)\in S_K^b(\Omega)$, one has
$$
|I(t,x)|\leq Ct^{-\frac{1+b}{2}}\leq Ct^{-\frac{1+b}{2}}(1+|x|)^\frac{b}{2}.
$$
On the other hand, note that when $\Omega=(0, r_0)$, one has  $\psi(\lambda)\in S_K^b(\Omega)\subset S_K^{ b/2}(\Omega)$. Thus if $t^{-\frac12}|x|\geq 1$,  then we apply Lemma \ref{lm2.6} with  $\psi(\lambda)\in S_K^{b/2}(\Omega)$ to obtain that
$$
|I(t,x)|\leq  Ct^{-\frac{1+b}{2}}(1+|x|)^\frac{b}{2}.
$$
This proves \eqref{eq2.22}.

Next, we prove  (\romannumeral2). In order to prove \eqref{eq2.23}, we perform integration by parts $[\frac{d-1}{2}]$ times in the integral  \eqref{eq2}  by using the following identity
                 $$\frac{1}{2i\lambda t}\cdot\frac{d}{d\lambda}e^{it\lambda^2}=e^{it\lambda^2}.$$
Then we see that
\begin{align}\label{eqA11}
I(t,x)=\sum_{s_1+s_2+s_3+s_4=[\frac{d-1}{2}]}C_{s_1 s_2 s_3 s_4}t^{-[\frac{d-1}{2}]}x^{s_1}\int_\Omega e^{i(t\lambda^2+\lambda\cdot x)}\frac{d^{s_2}}{d\lambda^{s_2}}\psi_1(\lambda)\frac{d^{s_3}}{d\lambda^{s_3}}\psi_2(\lambda)\cdot \lambda^{-[\frac{d-1}{2}]-s_4} d\lambda.
\end{align}
Meanwhile, since $\psi_2(\lambda) \in S_{[ \frac d2] +1}^{[ \frac{d-1}{2}]}(\Omega)$ and $s_3\leq [\frac{d-1}{2}]$, we have
\begin{align*}
 \frac{d^{s_3}}{d\lambda^{s_3}}\psi_2(\lambda)\cdot \lambda^{-[\frac{d-1}{2}]-s_4}&\in S^{-s_3-s_4}_{1}(\Omega)\in S_{1}^{0}(\Omega),\qquad \text{if $d$ is odd},\\
 \frac{d^{s_3}}{d\lambda^{s_3}}\psi_2(\lambda)\cdot \lambda^{-[\frac{d-1}{2}]-s_4}&\in  S^{-s_3-s_4}_{2}(\Omega) \in S_{2}^{0}(\Omega),\qquad \text{if $d$ is even}.
\end{align*}
Note that $\frac{d^k}{d\lambda^k}\psi_1(\lambda)\in S_1^{0}(\Omega)$ for every $1 \leq k\leq [ \frac d2]$, then one has
\begin{align}
\frac{d^{s_2}}{d\lambda^{s_2}}\psi_1(\lambda)\frac{d^{s_3}}{d\lambda^{s_3}}\psi_2(\lambda)\cdot \lambda^{-[\frac{d-1}{2}]-s_4}&\in S_{1}^{0}(\Omega),\qquad \qquad\quad\,\,\,\,\,\text{if $d$ is odd}, \label{eqA12}\\
\frac{d^{s_2}}{d\lambda^{s_2}}\psi_1(\lambda)\frac{d^{s_3}}{d\lambda^{s_3}}\psi_2(\lambda)\cdot \lambda^{-[\frac{d-1}{2}]-s_4}& \in S_{2}^{1} (\Omega)\cap S^{\frac12}_{1} (\Omega),\qquad \text{if $d$ is even}.\label{eqA13}
\end{align}

When $d$ is odd, by \eqref{eqA12}, we apply Lemma \ref{lm2.6} with $K=1,\,b=0$. Then the oscillatory integral in \eqref{eqA11} is estimated by $Ct^{-\frac12}$ and \eqref{eq2.23} follows.

When  $d$ is even, in view of \eqref{eqA13}, we apply  Lemma \ref{lm2.6} with $K=1,b=\frac12$ if $t^{-\frac12}|x|\geq 1$, and $K=2, b=1$ if $t^{-\frac12}|x|\leq 1$. Then the oscillatory integral in \eqref{eqA11} is estimated by $Ct^{-1}$. Therefore \eqref{eq2.23} also holds and the proof is complete.

%
%
%
\section{Absence of singular spectrum}\label{app2}
\begin{lemma}\label{lmA.1}
The Hamiltonians ($H_{\alpha}$ or $H$) in Theorem \ref{thm-1.1}--\ref{thm1.6} have only absolutely continuous spectrum.
\end{lemma}
\begin{proof}
First we consider the following general case
\begin{equation}\label{eqB1}
H=-\Delta+A, \quad\,\,\,\,\,\,\, A:=\sum_{j=1}^{+\infty}\lambda_jP_j, \,\,\,\,\,  \lambda_j\ge 0.
\end{equation}
where $P_j=\langle\cdot\,, \varphi_j\rangle \varphi_j$ and orthonormal sets $\{\varphi_j\}_{j=1}^{\infty}$ satisfy the decay condition \eqref{eq1.2}, in addition, we also assume that $\sum_{j=1}^{+\infty}\lambda_j<\infty$. Then $A$ is a trace class perturbation. By Weyl's Theorem and the fact that $H\ge 0$, we have
 $$
 \sigma(H)=\sigma_{\text{ess}}(H)=[0,\,\infty).
 $$ Moreover, it's a classical result of Kato-Rosenblum (see e.g. in \cite[Chapter  \uppercase\expandafter{\romannumeral10}]{Ka}) that the wave operators exist and are complete and the absolutely continuous parts of $H$ and $H_0=-\Delta$ are unitary equivalent.

Second, let $H_{\alpha}$ be the operator in Theorem \ref{thm-1.1}, note by the decay condition \eqref{eq1.2}, the limiting absorption principle yields that the boundary values of the free resolvent $R^{\pm}_0(\lambda^2)$ are locally uniformly bounded  in $B(L^2_{\sigma}-L^2_{-\sigma})$ ($\sigma>\frac12$) for all $d\ge 1$ when $\lambda>0$, then  by the spectral assumption \eqref{eq1.4} and  the  Aronszajn-Krein formula \eqref{eq3.2}, it follows that the perturbed resolvent $R^{\pm}(\lambda^2)$ shares the same property. In particular, this implies that  $(0, \infty)$ is purely  absolutely continuous. Similarly, for the operators in Theorem \ref{thm1.2}--\ref{thm1.6}, by the spectral assumptions \eqref{eq1.13.1} and \eqref{eq1.88.0}, as well as the Aronszajn-Krein type formulae \eqref{eq4.13} and \eqref{eq5.1.3}, it follows that $(0, \infty)$ is purely  absolutely continuous (see also in Theorem \ref{thm4.2}).

Finally, we show that  the bottom of the spectrum of $H$ in \eqref{eqB1}, i.e., $\lambda=0$ is not an eigenvalue. Indeed, if $\lambda=0$ is an eigenvalue, then there exists some $0\ne u\in D(H)=H^2(R^d)$ such that
$$
Hu=-\Delta u+Au=0,
$$
since both $-\Delta$ and $A$ are nonnegative, this shows that $\|\nabla u\|_{L^2}=\sum_{\lambda_j}|(\varphi_j, u)|^2=0$. This implies that $u\equiv 0$, which is a contradiction. Thus $0$ cannot be an eigenvalue.

Summing up we have proved Lemma \ref{lmA.1}.
\end{proof}

\end{appendix}

\noindent
\section*{Acknowledgments}
S. Huang was supported by the National Natural Science Foundation of China under grants 12171178 and 12171442. H. Cheng and Q. Zheng were supported by the National Natural Science Foundation of China under grant 12171178.


\begin{thebibliography}{plain}

\bibitem{AS} M. Abramowitz,  I. A. Stegun, Handbook of mathematical functions with formulas, graphs, and
mathematical tables. Government Printing Office, Washington, D.C. 1964


\bibitem{Ag}  S. Agmon, Spectral properties of Schr\"{o}dinger operators and scattering theory. \textit{ Ann. Sc. Norm. Super. Pisa. Cl. Sci.} 4 (1975), 151-218.

\bibitem{BG} M. Beceanu, M. Goldberg, Schr\"{o}dinger dispersive estimates for a scaling-critical class of potentials.  \textit{Commun. Math. Phys.} 314 (2012), 471-481.

\bibitem{D} P. D' Ancona, L. Fanelli, $L^p$-boundedness of the wave operator for the one dimensional Schr\"{o}dinger operator. \textit{Comm. Math. Phys.} 268 (2006), no. 2, 415-438.

\bibitem{ES} M. B. Erdo\v{g}an, W. Schlag,  Dispersive estimates for Schr\"{o}dinger operators in the presence of a resonance and/or an eigenvalue at zero energy in dimension three: \uppercase\expandafter{\romannumeral1}.  Dynamics of PDE, 1 (2004), no. 4, 359-379.

\bibitem{ES2} M. B. Erdo\v{g}an, W. Schlag,     Dispersive estimates for Schr¡§odinger operators in the presence of a resonance and/or eigenvalue at zero energy in dimension three: \uppercase\expandafter{\romannumeral2}. \textit{J. Anal. Math.} 99 (2006), 199-248.

\bibitem{EGG} M. B. Erdo\v{g}an, M. Goldberg, W. R. Green, Dispersive estimates for four dimensional Schr\"{o}dinger and wave equations with obstructions at zero energy. \textit{Comm. Partial Differential Equations} 39 (2014), no. 10, 1936-1964.

 \bibitem{EG13} M. B. Erdo\v{g}an, W. R. Green, Dispersive estimates for Schr\"{o}dinger operators in dimension two with obstructions at zero energy. \textit{Trans. Amer. Math. Soc.} 365 (2013), no. 12, 6403-6440.

  \bibitem{EGG18} M. B. Erdo\v{g}an, M. Goldberg, W. R. Green,  On the $L^p$ boundedness of wave operators for two-dimensional Schr\"{o}dinger operators with threshold obstructions. \textit{J. Funct. Anal.} 274 (2018), no. 7, 2139-2161.


\bibitem{EG10}  M. B. Erdo\v{g}an, W. R. Green, Dispersive estimates for the Schr\"{o}dinger equation for $C^{\frac{n-3}{2}}$ potentials in odd dimensions. \textit{Int. Math. Res. Not. IMRN} 2010, no. 13, 2532-2565.

\bibitem{EG13B}  M. B. Erdo\v{g}an, W. R. Green,  A weighted dispersive estimate for Schr\"{o}dinger operators in dimension two.  \textit{Commun. Math. Phys.} 319 (2013), 791-811.


\bibitem{EGT}  M. B. Erdo\v{g}an, W. R. Green,  E. Toprak, On the fourth order Schr\"{o}dinger equation in three dimensions: dispersive estimates and zero energy resonances. \textit{J. Differential Equations} 271 (2021), 152-185.


\bibitem{EG22} M. B. Erdo\v{g}an, W. R. Green,  The $L^p$-continuity of wave operators for higher order Schr\"{o}dinger operators.  \textit{Adv. Math.} 2022, 404, 108450.



\bibitem{FSY} H. Feng, A. Soffer,  X. Yao, Decay estimates and Strichartz estimates of fourth-order Schr\"{o}dinger operator. \textit{J. Funct. Anal.} 274 (2018), no. 2, 605-658.

\bibitem{FSWY} H. Feng, A. Soffer, Z. Wu,  X. Yao, Decay estimates for higher-order elliptic operators. \textit{Trans. Amer. Math. Soc.} 373 (2020), no. 4, 2805-2859.


\bibitem{gra} L. Grafakos, Classical Fourier Analysis.  Graduate Texts in Mathematics, second ed., vol. 249, New York: Springer, 2008.

\bibitem{GS} M. Goldberg, W. Schlag, Dispersive estimates for Schr\"{o}dinger operators in dimensions one and three. \textit{Comm. Math. Phys.} 251 (2004) 157-178.

\bibitem{Gol}  M. Goldberg, Dispersive bound for the three-dimensional Schr\"{o}dinger equation with almost critical potentials. \textit{Geom.  Funct. Anal.} 16 (2006), no. 3, 517-536.

\bibitem{Gol2} M. Goldberg, Dispersive Estimates for the Three-Dimensional Schr\"{o}dinger Equation with Rough Potentials. \textit{Amer. J. Math.} 128 (2006) 731-750.

\bibitem{GV}  M. Goldberg, M. Visan, A counterexample to dispersive estimates for Schr\"{o}dinger operators in higher dimensions. \textit{Comm. Math. Phys.} 266 (2006)  211-238.

\bibitem{GW15} M. Goldberg, W. R. Green, Dispersive estimates for higher dimensional Schr\"{o}dinger operators with threshold eigenvalues \uppercase\expandafter{\romannumeral1}: The odd dimensional case.  \textit{J. Funct. Anal.} 269 (2015), no. 3, 633-682.


\bibitem{GW17} M. Goldberg, W. R. Green, Dispersive estimates for higher dimensional Schr\"{o}dinger operators with threshold eigenvalues \uppercase\expandafter{\romannumeral2}. The even dimensional case. \textit{J. Spectr. Theory} 7 (2017), no. 1, 33-86.

\bibitem{GW16}  M. Goldberg, W. R.  Green, The $L^p$ boundedness of wave operators for Schr\"{o}dinger operators with   threshold singularities, \textit{Adv. Math.} 303 (2016), 360-389.

\bibitem{HHZ} T. Huang, S. Huang and Q. Zheng, Inhomogeneous oscillatory integrals and global smoothing effects for dispersive equations. \textit{J. Differential Equations} 263 (2017), no. 12, 8606-8629.


\bibitem{J80} A. Jensen, Spectral properties of Schr\"{o}dinger operators and time-decay of the wave functions results in $L^2(\mathbb{R}^m)$, $m\ge 5$. \textit{Duke Math. J.} 47 (1980), no. 1, 57-80.

%

\bibitem{JN} A. Jensen, G. Nenciu, A unified approach to resolvent expansions at thresholds. \textit{Rev. Math. Phys.} 13 (2001), no. 6, 717-754.

\bibitem{JSS} J.-L. Journ\'{e}, A. Soffer, C. D. Sogge, Decay estimates for Schr\"{o}dinger operators. \textit{Comm. Pure Appl. Math.} 44 (1991), no. 5, 573-604.

\bibitem{Ka} T. Kato,  Perturbation Theory for Linear Operators, vol. 132, 2nd ed. Springer, New York 1976.

\bibitem{KK}  A.  Komech, E. Kopylova, Dispersion decay and scattering theory. John Wiley \& Sons, Inc., Hoboken, NJ, 2012.

\bibitem{Liaw}  C. Liaw,  Rank one and finite rank perturbations-survey and open problems. Preprint, arXiv:1205.4376v1.


\bibitem{NS}   F.  Nier, A. Soffer, Dispersion and Strichartz estimates for some finite rank perturbations of the Laplace operator.  \textit{J. Funct. Anal.} 198 (2003), no. 2, 511-535.

\bibitem{RS} I. Rodnianski, W. Schlag, Time decay for solutions of Schr\"{o}dinger equations with rough and time-dependent potentials.  \textit{Invent. Math.} 155 (2004), no. 3, 451-513.

\bibitem{RSS}   I. Rodnianski, W. Schlag,  A. Soffer, Dispersive analysis of charge transfer models. \textit{Comm. Pure Appl. Math.} 58 (2005), no. 2, 149-216.

\bibitem{Sch} W. Schlag, Dispersive estimates for Schr\"{o}dinger operators in dimension two. \textit{Comm. Math. Phys.} 257 (2005), no. 1, 87-117.

\bibitem{Sch07}  W.  Schlag, Dispersive estimates for Schr\"{o}dinger operators: a survey. Mathematical aspects of nonlinear dispersive equations. \textit{Ann. of Math. Stud.} 163 (2007), 255-285.  Princeton Univ. Press, Princeton, NJ.

\bibitem{Sch21} W. Schlag, On pointwise decay of waves. \textit{J. Math. Phys.} 62 (2021), 061509.

\bibitem{SW} B. Simon, T. Wolff, Singular continuous spectrum under rank one perturbations and localization
for random Hamiltonians.  \textit{Comm. Pure Appl. Math.}  39 (1986), no. 1, 75-90.

\bibitem{Simon}  B. Simon, Spectral analysis of rank one perturbations and applications, in: Mathematical Quantum
Theory. \uppercase\expandafter{\romannumeral2}. Schr\"{o}dinger operators, Vancouver, BC, 1993, American Mathematical Society, Providence, RI, 1995, 109-149

\bibitem{St} E. M. Stein,  Harmonic Analysis: Real-Variable Methods, Orthogonality, and Oscillatory Integrals. Princeton University Press, Princeton 1993.

\bibitem{Wed}  R. Weder, The $W^{k,p}$-continuity of the Schr\"{o}dinger wave operators on the line. \textit{Comm. Math. Phys.} 208 (1999), no. 2, 507-520.

\bibitem{Wey} H. Weyl, \"{U}ber gew\"{o}hnliche Differentialgleichungen mit Singularit\"{a}ten und die zugeh\"{o}rigen Entwicklungen willk\"{u}rlicher Funktionen. \textit{Math. Ann.}  68 (1910), no. 2, 220-269.


\bibitem{Ya} K. Yajima, The $W^{k,p}$ continuity of wave operators for Schr\"{o}dinger operators. \textit{J. Math. Soc. Japan} 47 (1995), no. 3, 551-581.

\bibitem{Ya95} K. Yajima, The $W^{k,p}$-continuity of wave operators for Schr\"{o}dinger operators. \uppercase\expandafter{\romannumeral3}. Even-dimensional cases $m\ge 4$. \textit{J. Math. Sci. Univ. Tokyo} 2 (1995), no. 2, 311-346.

\bibitem{Ya2} K. Yajima, Dispersive estimates for Schr\"{o}dinger equations with threshold resonance and eigenvalue. \textit{Comm. Math. Phy.}  259 (2005),  no. 2, 479-509.


\end{thebibliography}
\end{document}